\documentclass[11pt]{article}
\usepackage{a4wide}
\usepackage{color}
\usepackage{amssymb,amsmath,amsthm,amsfonts,latexsym}
\usepackage{tikz-cd}

\usepackage{amsmath,amsfonts,amssymb,amsthm,bbm} 
\usepackage{hyperref}

\usepackage[nobysame,initials]{amsrefs}
\usepackage{ifthen,latexsym,amssymb,amsmath,bbm,fixmath}
 \usepackage[shortlabels]{enumitem}

\newcommand{\C}[1]{{\protect\mathcal{#1}}}
\newcommand{\B}[1]{{\mathbold #1}}
\newcommand{\I}[1]{{\mathbbm #1}}
\renewcommand{\O}[1]{\overline{#1}}

\newcommand{\V}[1]{\mathbold{#1}}

\newcommand{\e}{\epsilon}
\newcommand{\floor}[1]{\lfloor #1\rfloor}
\newcommand{\me}{{\mathrm e}}
\renewcommand{\mid}{:}

\renewcommand{\ge}{\geqslant}
\renewcommand{\le}{\leqslant}

\renewcommand{\L}[1]{\ifthenelse{\equal{#1}{}}{\ell}{\ell\!\left(#1\right)}}
\renewcommand{\l}[1]{L(#1)}
\newcommand{\ProofOf}[1]{\medskip\noindent\textbf{Proof of~\ref{#1}.}}

\newif\ifnotesw\noteswtrue

\newcommand{\hide}[1]{}

\newcommand{\beq}[1]{\begin{equation}\label{#1}}
	\newcommand{\eeq}{\end{equation}}

\newtheorem{theorem}{Theorem}[section]

\newtheorem{lemma}[theorem]{Lemma}
\newtheorem*{claim*}{Claim}

\newtheorem{claim}[theorem]{Claim}

\newtheorem{proposition}[theorem]{Proposition}

\theoremstyle{definition}
\newtheorem{definition}[theorem]{Definition}

\theoremstyle{remark}
\newtheorem{remark}[theorem]{Remark}

\newcommand{\qqed}{} 

\newcommand{\tP}{\widetilde{\mathbb{P}}}

\newcommand{\W}{\mathcal{W}}
\newcommand{\tW}{\widetilde{\mathcal{W}}}
\newcommand{\tH}{\widetilde{\mathbb{H}}}
\newcommand{\FORB}{\mathbf{FORB}}
\newcommand{\RN}{\mathbf{RN}}

\newcommand{\fH}{\mathcal{H}}
\newcommand{\fG}{\mathcal{G}}

\newcommand{\ind}[1]{_{#1}}

\def\leukfrac#1/#2{\leavevmode
               \kern.1em
                \raise.9ex\hbox{\the\scriptfont0 ${}_#1$}
                \hskip -1pt\kern-.1em
                /\kern-.15em\lower.10ex\hbox{\the\scriptfont0 ${}_#2$}}

\makeatletter
\def\diam{\mathop{\operator@font diam}\nolimits}
\makeatother

\newcommand{\DZ}[1]{\cite[#1]{DemboZeitouni10ldta}}

\newcommand{\RS}[1]{\cite[#1]{RassoulaghaSeppelainen14cldigm}}

\newcommand{\Lo}[1]{\cite[#1]{Lovasz:lngl}}
\newcommand{\ra}{\ $\rightarrow$\ }

\newcommand{\Ak}{{\bf A}^{(k)}}
\newcommand{\Wk}{\mathcal{W}^{(k)}}
\newcommand{\tWk}{\widetilde{\mathcal{W}}^{(k)}}
\newcommand{\Walpha}{\mathcal{W}\times \Aalpha}
\newcommand{\tWalpha}{\widetilde{\mathcal{W}}^{(\V\alpha)}}

\newcommand{\tind}{t_{\mathrm{ind}}}
\newcommand{\T}[1]{\widetilde{#1}}
\newcommand{\tR}{\widetilde{\mathbb{R}}}
\newcommand{\tPap}{\tP_{\V a_n,\V p}}
\newcommand{\dd}{\,\mathrm{d}}
\newcommand{\Aalpha}{{\bf A}^{(\V\alpha)}}
\newcommand{\f}[1]{f^{#1}}
\newcommand{\tf}[1]{\T{f^{#1}}}

\newcommand{\Sball}{{\O B}_{\delta_\Box}}
\newcommand{\iP}{\I P}
\newcommand{\iR}{\I R}
\newcommand{\cI}[2]{I_{#2}^{(#1)}}

\newcommand{\NoZero}{}
\newcommand{\NZ}{\I N_{\ge 0}}
\newcommand{\Tk}[1]{\widetilde{(#1)}}
\renewcommand{\rho}{r}

\newcommand{\arxiv}[2]{#2 {\sf In arXiv: #1}}

\renewcommand{\arxiv}[2]{#1}

\begin{document}

\title{Large deviation principles for graphon sampling}

\hide{
\author{\textbf{Jan Greb\'\i k}\\
	UCLA Mathematics\\
	520 Portola Plaza\\
	Los Angeles, CA 90095, USA\\ 
	and\\
	Faculty of Informatics\\
	Masaryk University\\
	Botanick\'a 68A\\
	602 00 Brno, Czech Republic \and 
	\textbf{Oleg Pikhurko}\\
	Mathematics Institute and DIMAP\\
	University of Warwick\\
	Coventry CV4 7AL, UK}
}

\author{Jan Greb\'\i k\thanks{
	Universität Leipzig, Mathematisches Institut, D-04009, Leipzig, Germany. Supported by Leverhulme Research Project Grant RPG-2018-424, and by MSCA Postdoctoral Fellowships 2022 HORIZON-MSCA-2022-PF-01-01 project BORCA grant agreement number 101105722.} \and 
	Oleg Pikhurko\thanks{
	Mathematics Institute and DIMAP,
	University of Warwick,
	Coventry CV4 7AL, UK. Supported by ERC Advanced Grant 101020255 and Leverhulme Research Project Grant RPG-2018-424.}
	}

\maketitle

\begin{abstract}
We investigate possible large deviation principles (LDPs) for the $n$-vertex sampling from a given graphon with various speeds $s(n)$ and resolve all the cases except when the speed $s(n)$ is of order $n^2$. 

For quadratic speed $s=(c+o(1))n^2$, we establish an LDP for an arbitrary $k$-step graphon, which extends a result of
Chatterjee and Varadhan~[\emph{Europ.\ J.\ Combin.}, 32 (2011) 1000--1017] who did this for $k=1$ (that is, for the homogeneous binomial random graphs). This is done by reducing the problem to the LDP for stochastic $k$-block models established by Borgs, Chayes, Gaudio, Petti and Sen~[\emph{``A large deviation principle for block models''}, arxiv:2007.14508, 2020]. We also improve some results from this paper.
\end{abstract}






\section{Introduction}

\subsection{Overview of the new results}

For the reader's convenience, this section briefly and informally, summaries the main new 
results obtained in this paper. (Their full statements can be found in Sections~\ref{se:LDPDef}--\ref{se:SBM}.)

We study large deviations for the \emph{$n$-vertex sample} $\I G(n,W)$ from a given \emph{graphon} $W$, that is, a symmetric measurable function $W:[0,1]^2\to[0,1]$. This is a random graph on vertex set $[n]:=\{1,\dots,n\}$ produced as follows:  we pick $n$ elements $x_1,\dots,x_n\in [0,1]$ uniformly at random and make each pair $\{i,j\}$ an edge with probability $W(x_i,x_j)$, with all choices being mutually independent.
We investigate if a large deviation principle (LDP) holds for a given speed function $s(n)$, i.e.\ if there is a rate function $I$ from the appropriately defined topological space $\tW$ (that includes all finite graphs, up to isomorphism) to $[0,\infty]$ such that the probability of $\I G(n,W)$ hitting a subset $A\subseteq \tW$ can be related to $\exp(-s(n)\inf_{A} I)$ in the standard LDP sense (made precise in Definition~\ref{df:LDP} here).%
\hide{
For a given \emph{speed function} $s(n)$, we investigate if a \emph{large deviation principle} (LDP) holds in an appropriately defined compact metric space $\tW$ that include all finite graphs, meaning that there is a function $I:\tW\to [0,\infty]$ such that the probability that $G_n$ belongs to a set $A\subseteq \tW$ is ``well'' approximated (in a sense that can be made precise) by $\exp\left(-s(n)\inf_{x\in A} I(x)\right)$.

For a given \emph{speed function} $s(n)$ and a set $A$ in the appropriately defined compact metric space $\tW$ that include all finite graphs, we would like to estimate $-\frac1{s(n)}\log \I P[G_n\in A]$ as $n\to\infty$. 
If this quantity tends to a limit $I(x)$, possibly $0$ or $\infty$, whenever (in addition to $n\to\infty$) $A$ ``slowly shrinks'' to a point $x\in \tW$
then a \emph{large deviation principle} (LDP) holds with \emph{rate} $I$ and \emph{speed}~$s$. In this case, the general theory of large deviations allows us to estimate the probabilities of $G_n$ hitting general Borel subsets of~$\tW$, etc.}
The first such LDP result for graphon sampling was given in the ground-breaking work of Chatterjee and Varadhan~\cite{ChatterjeeVaradhan11} who established an LDP  for speed~$(c+o(1))n^2$ when $W$ is the constant graphon (and thus $\I G(n,W)$ is the standard binomial random graph).

We show that, for any graphon $W$, the only ``interesting''  speeds are $\Theta(n)$ and $\Theta(n^2)$; namely for all other speed functions either an LDP does not hold or it holds with an explicitly described rate function $I$ that assumes values $0$ and $\infty$ only. 

For speed $(c+o(1))n$, we establish an LDP for any graphon $W$, with a key ingredient of our proof being an extension by Eichelsbacher~\cite{Eichelsbacher97} of the classical LDP result of Sanov~\cite{Sanov57} for empirical measures.

For speed $(c+o(1))n^2$, we establish an LDP in the case when $W$ is a \emph{step graphon}, that is, the elements of $[0,1]$ are split into $k$ (finitely many) types and the probability of $\{i,j\}$ being an edge of $\I G(n,W)$ depends only on the types of $x_i$ and $x_j$. This case is related to the so-called \emph{stochastic $k$-block random graph model} where the difference to step graphon sampling is that the number of vertices of each type is given.
Our proof uses the LDP for stochastic block models that was recently obtained by Borgs, Chayes, Gaudio, Petti and Sen~\cite{BCGPS}. As a by-product of our line of research, we also improve some results in~\cite{BCGPS}. 

Finally, we completely resolve the LDP problem (for any given speed) for \emph{weighted graph sampling} from $W$, where one takes independent uniform $x_1,\dots,x_n\in [0,1]$ and output an edge-weighted complete graph on $[n]$ with the weight of $\{i,j\}$ being $W(x_i,x_j)$.



\subsection{Large deviations for general topological spaces}\label{se:LDPDef}

The theory of large deviations (see e.g.\ 
the books by Dembo and Zeitouni~\cite{DemboZeitouni10ldta}, Rassoul-Agha and Sepp\"al\"ainen~\cite{RassoulaghaSeppelainen14cldigm}, and Varadhan~\cite{Varadhan16ld})
 studies the probabilities of rare events on the exponential scale. 
This is formally captured by the following definition.

\begin{definition}\label{df:LDP}
	A sequence of Borel probability measures $(\mu_n)_{n\in \I N}$ on a 
	topological space $X$ satisfies a \emph{large deviation principle} (\emph{LDP} for short) with \emph{speed} 
	$s:\I N\to(0,\infty)$
	and \emph{rate function} $I:X\to[0,\infty]$ if
	\begin{itemize}[nosep]
		\item 
	$s(n)\to\infty$ as $n\to\infty$, 
	\item the function $I$ is \emph{lower semi-continuous} (\emph{lsc} for short), that is, for each $\rho\in\I R$
	the level set $\{I\le \rho\}:=\{x\in X\mid I(x)\le \rho\}$ is a closed subset of~$X$,
	\item the following \emph{lower bound} holds:
	\beq{eq:lowerGen}
		\liminf_{n\to\infty} \frac1{s(n)}\,{\log\, \mu_n(G)} 
		\ge -\inf_{x\in G} I(x),\quad \mbox{for every open $G\subseteq X$,}
		\eeq
		\item the following \emph{upper bound} holds:
		\beq{eq:upperGen}\limsup_{n\to\infty} \frac1{s(n)}\,{ \log\,\mu_n(F)} \le  -\inf_{x\in F} I(x),\quad \mbox{for every closed $F\subseteq X$.}
\eeq
\end{itemize}
\end{definition}

\hide{
As it is well-known (see e.g.\ \RS{Lemma~2.11}), if~\eqref{eq:lowerGen} and~\eqref{eq:upperGen} hold for some (not necessarily lsc) function $I:X\to [0,\infty]$ then we can replace $I$ without violating these bounds by its \emph{lower semi-continuous
	regularization} 
\beq{eq:lscR}
 I_{\mathrm{lsc}}(x):=\sup\left\{\inf_{y\in G} I(y)\mid G\ni x\mbox{ and $G\subseteq X$ is open}\right\},\quad x\in X;
\eeq 
furthermore (see e.g.\ \RS{Lemma~2.8}), 
 $I_{\mathrm{lsc}}$ is lower semi-continuous  and, in fact, $I_{\mathrm{lsc}}$ is the largest lsc function with $I_{\mathrm{lsc}}\le I$. 
If $X$ is a regular topological space then there can be at most one lower semi-continuous rate function satisfying Definition~\ref{df:LDP} (see e.g.\ \DZ{Lemma 4.1.4} or \RS{Theorem~2.13}). This motivates the requirement that $I$ is lsc in Definition~\ref{df:LDP}. 
}


Large deviations for various models of random graphs have been receiving much attention (see e.g.\ the survey by Chatterjee~\cite{Chatterjee16bams}), often motivated by specific applications, such as tails of subgraph counts in the binomial random graph (see e.g.\ \cite{BhattacharyaGangulyLubetzkyZhao17,ChatterjeeDembo16,HarelMoussetSamotij22,JansonOleszkewiczRucinski04,JansonRucinski02,JansonRucinski04,LubetzkyZhao15,LubetzkyZhao17rsa}) or in some of its constrained versions (\cite{DemboLubetzky18ecp,DionigiGarlaschelliHollanderMandje21,HollanderMandjesRoccaverdeStarreveld18,HollanderMarkering23}), 
spectral properties (\cite{ChakrabartyHazraHollanderSfragara22,HazraHollanderMarkering25}), algorithms for clustering in 2-block random graphs (\cite{Massoulie14stoc,MosselNeemanSly18}), particular fluctuations of subgraph counts in $W$-random graphs (\cite{BhattacharyaChatterjeeJanson23,DelmasDhersinSciauveau21rsa,FerayMeliotNikeghbali20,HladkyPelekisSileikis21}), etc. 
In addition to the random graph models discussed in this paper, 
an LDP for growing random graphs with a given degree distribution was obtained Dhara and Sen \cite{DharaSen22} by reducing to large deviations for a certain model which gives another generalisation of the binomial random graph $\I 
G(n,p)$; these results were strengthened by Markering~\cite{Markering23}.


\subsection{Binomial random graph and graphons}

A basic but central model
is the \emph{binomial random graph} $\I G(n,p)$, where the vertex set is $[n]$ and 
each pair of vertices is an edge with probability~$p$, independently of other pairs. A large deviation principle for $\I G(n,p)$ for constant $p\in (0,1)$ was established in a ground-breaking paper of
Chatterjee and Varadhan~\cite{ChatterjeeVaradhan11} as follows. (See also the exposition of this proof in  Chatterjee's book~\cite{Chatterjee17ldrg}.) 

As it turns out, the ``right'' setting is to consider \emph{graphons}, that is,  measurable symmetric functions $[0,1]^2\to [0,1]$. On the set $\C W$  of all {graphons}, one can define the so-called \emph{cut-distance} $\delta_\Box$, which is a pseudo-metric on $\C W$. (See Section~\ref{graphons} for all missing definitions related to graphons.) Consider the factor space 
$$
\tW:=\{\T U: U\in\C W\},
$$ 
where $\T U:=\{V\in\C W\mid \delta_\Box(U,V)=0\}$ consists of all graphons \emph{weakly isomorphic} to~$U$.
 The space $(\tW,\delta_\Box)$ naturally appears in the limit theory of dense graphs, see e.g.\ the book by
Lov\'asz~\cite{Lovasz:lngl}. In particular, a graph $G$ on $[n]$ can be identified with the graphon $\f{G}$ where we partition $[0,1]$ into intervals of length $1/n$ each and let $\f{G}$ be the $\{0,1\}$-valued step function that encodes the adjacency relation.
This way, $\I G(n,p)$ gives a (discrete) probability measure $\tP_{n,p}$ on $(\tW,\delta_\Box)$, where $\tP_{n,p}(A)$ for $A\subseteq \tW$ is the probability that the sampled graph,
when viewed as a graphon up to weak isomorphism, belongs to the set~$A$. 

Also, recall that, for $p\in [0,1]$,  the \emph{relative entropy} $h_p:[0,1]\to [0,\infty]$ is defined by
\beq{eq:hp}
 h_p(\rho):=
 p \ \L{\frac{r}{p}}+(1-p) \ \L{\frac{1-r}{p}},
 \quad \rho\in[0,1],
\eeq
 where $\L{}:[0,\infty]\to [-1/\me,\infty]$ is defined by
 \beq{eq:L}
 \L{x}:=x\log (x),\quad  \mbox{for $x\in [0,\infty]$,}
 \eeq
 with the agreement that $0\log(0)=0$ and $\infty\log(\infty)=\infty$.

\begin{theorem}[Chatterjee and Varadhan~\cite{ChatterjeeVaradhan11}]
	\label{th:CV}
	Let $p\in [0,1]$. The function $I_p:\C W\to [0,\infty]$ defined by 
	\beq{eq:IpCV}
	 I_p(U):=\frac12 \int_{[0,1]^2} h_p(U(x,y))\dd x\dd y,\quad U\in\C W,
	 \eeq
	 gives a well-defined function $\tW\to [0,\infty]$ (that is, $I_p$ assumes the same value at any two graphons at $\delta_\Box$-distance 0) which is lower semi-continuous on $(\tW,\delta_\Box)$. Moreover,
	 the sequence of measures $(\tP_{n,p})_{n\in\I N}$ on $(\tW,\delta_\Box)$ satisfies an LDP with speed $n^2$ and rate function~$I_p$.\qqed
	 \end{theorem}

\subsection{Sampling from a graphon}

There are various important inhomogeneous random graph models that generalise the binomial random graph $\I G(n,p)$. The main objective of this paper is to initiate a systematic study of large deviations for the following model. 

For a graphon $W\in\C W$ and an integer $n\in\I N$, the \emph{$n$-vertex $W$-sample} $\I G(n,W)$ is the random graph on $[n]$ generated by first sampling
uniform elements $x_1,\dots,x_n\in [0,1]$ and then making each pair $\{i,j\}$ an edge with probability $W(x_i,x_j)$, where all choices are mutually independent. As it easily follows from standard results (e.g.\ \Lo{Lemma~10.23}), the distribution of $\I G(n,W)$ will remain the same if we replace $W$ by any weakly isomorphic graphon, so we could write $\I G(n,\T W)$. However, here and in  many similar situations, we prefer to talk about graphons rather than equivalence classes of graphons, just for brevity and notational convenience.

Let $\tR_{n,W}$ be the corresponding (discrete) measure on~$\tW$ where we take the equivalence class $\tf{G}\in\tW$ of the sampled graph~$G\sim \I G(n,W)$. When $W$ is the constant function $p$, we get exactly the binomial random graph $\I G(n,p)$ and $\tR_{n,W}=\tP_{n,p}$.

Let us give some definitions that are needed to state the first batch of our LDP results, namely that, informally speaking,  the only ``interesting'' speeds for the sequence $(\tR_{n,W})_{n\in\I N}$ of measures  are $\Theta(n)$ and $\Theta(n^2)$, as well as a general LDP for speed~$\Theta(n)$. Recall that we work on the graphon space $\tW$ endowed with the cut-distance $\delta_\Box$.

Let $W\in \C W$ be a graphon.
We assign to $W$ two $\delta_\Box$-closed subsets $\RN(W)$ and $\FORB(W)$ of $\tW$ that satisfy $W\in \RN(W)\subseteq \FORB(W)$ as follows.

The set $\FORB(W)$ consists by definition of those graphons $U$ such that $\tind(F,W)=0$ implies $\tind(F,U)=0$ for every graph $F$, where after relabelling $V(F)$ to be $[n]$ we define
\begin{equation}\label{eq:TInd}
\tind(F,U):=\int_{[0,1]^{{n}}} \prod_{\{i,j\}\in E(F)} U(x\ind{i},x\ind{j}) \prod_{\{i',j'\}\in E(\O F)} \left(1-U(x\ind{i'},x\ind{j'})\right) \dd x_1\dots\dd x_n
\end{equation}
to be the \emph{induced (homomorphism) density} of $F$ in~$U$. (Equivalently for $V(F)=[n]$, we can define $\tind(F,U)$ as the probability that $\I G(n,U)$ is equal to $F$.)
 In other words, $\FORB(W)$ consists of those graphons that do not contain induced subgraphs not present in~$W$. Since each function $\tind(F,\_)$ is well-defined and continuous on $(\tW,\delta_\Box)$ (see e.g.\ \Lo{Lemma~10.23}), the set $\FORB(W)$ is closed and depends only on the weak isomorphism class $\T W$.

The set $\RN(W)$ consists, roughly speaking, of graphons that can be formed from $W$ by changing the measure on the interval $[0,1]$.
Formally, suppose that $\mu$ is a Borel probability measure on $[0,1]$ with $\mu\ll\lambda$, that is, $\mu$
is absolutely continuous with respect to the Lebesgue measure~$\lambda$. Fix any 
measure-preserving map $\phi:([0,1],\lambda)\to([0,1],\mu)$, which exists by the Isomorphism Theorem for standard probability atomless spaces, and define $W_\mu\in \W$ by
 \beq{eq:WMu}
 W_\mu(x,y):=W(\phi(x),\phi(y)),\quad x,y\in [0,1].
 \eeq
 While this definition depends on the choice of $\phi$, the weak isomorphism class of $W_\mu$ and the other properties of $W_\mu$ that we use in this paper 
 do not. For example, $\tind(F,W_\mu)$ is given by the formula in~\eqref{eq:TInd} except $\dd x_i$ is replaced by $\dd\mu(x_i)$.
 The reader familiar with the theory of generalised graphons (in the sense of~\Lo{Section~13.1}) may prefer to define $W_\mu$ as the generalised graphon obtained from $W$ by changing the underlying probability space from $([0,1],\lambda)$ to $([0,1],\mu)$. We define $\RN(W)$ to  be the $\delta_\Box$-closure of the set of graphons $U$ for which there is a  Borel probability measure $\mu$ on $[0,1]$ with $\mu\ll \lambda$ such that $U$ is  weakly isomorphic to~$W_\mu$.

For $Y\subseteq \tW$ we denote as $\chi_Y$ the function that is $0$ on $Y$ and $\infty$ otherwise.
In the case when $Y=\{W\}$ is a singleton, we write simply $\chi_W$.
Note that $\chi_Y$ is lsc whenever $Y$ is $\delta_\Box$-closed.

Let us illustrate the above notions on a simple example. Namely, suppose that $W$ is a \emph{$k$-step graphon}, meaning that there is a partition $[0,1]=A_1\cup\dots\cup A_k$ into measurable sets such that $W$ is a constant $p_{i,j}$ on each product $A_i\times A_j$, $i,j\in [k]$. 
Let us assume additionally that each $A_i$ has positive measure.
It is easy to see that $\RN(W)$ consists exactly of weak equivalence classes of $k$-step graphons that can be produced from $W$ by stretching or shrinking $A_1,\dots,A_k$.
In other words, every element of $\RN(W)$ comes from a $k$-step graphon $U$ with parts $B_1,\dots,B_k$ such that the (constant) value of $U$ on $B_i\times B_j$ is the same as the value of $W$ on $A_i\times A_j$ for every $i,j\in [k]$. The structure of the set $\FORB(W)$ is more complicated. For example, if $p_{i,i}\not=0$ for some $i\in [k]$, then every graph has positive induced density in $W$ and thus $\FORB(W)=\tW$ consists of all graphons. If $p_{i,i}=0$ and $p_{i,j}\in (0,1)$ for every $i\not=j$ then $\FORB(W)$ can be shown to consist of  $\widetilde U\in\tW$ such that $U$ has \emph{chromatic number} at most $k$ (meaning that there is 
a  measurable partition $B_1\cup\dots\cup B_k$ of $[0,1]$  such that $U$ is 0 a.e.\ on all diagonal products $B_i\times B_i$).  Suppose that we consider an LDP for $(\tR_{n,W})_{n\in\I N}$ for some speed~$s$. Informally, the price in probability to change substantially the number of sampled $x_i\in [0,1]$ that are in some part $A_t$ is $\me^{-\Theta(n)}$ while analogous anti-concentration probability for edge rounding is~$\me^{-\Theta(n^2)}$. 
Thus at speed $s(n)= o(n)$, we do not see any deviations in the limit and the rate function is $\chi_W$. For speed $s(n)\in \omega(n)\cap o(n^2)$ (that is, $s(n)/n\to\infty$ and $s(n)/n^2\to 0$ as $n\to\infty$), any vertex distribution per parts $A_t$ is ``for free'' (while edge rounding deviations are still ``too expensive''), so LDP sees exactly $\RN(W)$. Likewise, for speed $\omega(n^2)$, every possible sampled graph must admit a vertex partition $V_1\cup\dots\cup V_k$ such that for each $(i,j)\in [k]^2$ with $p_{i,j}=0$ (resp.\ $p_{i,j}=1$) all pairs in $V_i\times V_j$ are non-edges (resp.\ edges, apart the diagonal pairs  if $i=j$) and each such graph occurs with probability at least $\me^{-O(n^2)}$; the corresponding set of graphons in the limit as $n\to\infty$ is easily seen to be $\FORB(W)$.

Staying with the above $k$-step example $W$, it is not hard to derive from the classical theorem of Sanov~\cite{Sanov57}
that $(\tR_{n,W})_{n\in\I N}$ admits an LDP for speed $n$ with rate function $I(U)$ which is equal to $\infty$ unless $U$ is a $k$-step graphon, say with steps $B_1,\dots,B_k$, when 
$$
I(U):=\min_\sigma \sum_{i\in [k]} \lambda(B_{\sigma(i)})\log\left(\frac{\lambda(B_{\sigma(i)})}{\lambda(A_i)}\right),
$$
 with the minimum taken over all permutations $\sigma$ of $[k]$ such that,  for all $i,j\in [k]$, the value of $U$ on $B_{\sigma(i)}\times B_{\sigma(j)}$ is the same as the value of $W$ on $A_i\times A_j$ (and $I(U)$ is defined to be $\infty$ if no such $\sigma$ exists).
 The natural (and correct) guess of the rate function for general $W\in\W$  for speed $s(n)=n$ is
\beq{eq:KW}
K_{W}(\T U):=\inf \left\{ \int_{[0,1]} \L{\frac{\dd\mu}{\dd\lambda}} \dd\lambda: \delta_\Box(U,W_\mu)=0, \ \mu \ll \lambda \right\},\quad U\in\W,
 \eeq
where $\frac{\dd\mu}{\dd\lambda}$ is the Radon--Nikodym derivative and $\mu$ is taken over Borel probability measures on $[0,1]$ absolutely continuous with respect to the Lebesgue measure $\lambda$ (and, as everywhere in the paper, $\L{s}:=x\log x$ is the function defined in~\eqref{eq:L}).
  Proposition~\ref{pr:K}\ref{it:KLsc} shows that the function $K$ is lsc on $(\tW,\delta_\Box)$.

Our first result settles LDP for $\tR_{n,W}$ for every speed other than $\Theta(n^2)$. We use the standard asymptotic notation ($o$, $O$, $\omega$, and $\Theta$) as $n\to\infty$.

\begin{theorem}\label{thm:GeneralSpeeds}
Take any $W\in\C W$ and any function $s:\I N\to(0,\infty)$ with $s(n)\to\infty$ as $n\to\infty$. We consider the existence of an LDP for the sequence $(\tR_{n,W})_{n\in\I N}$ of measures on the topological space $(\tW,\delta_\Box)$ with speed~$s$.
\begin{enumerate}[(a)]
	\item\label{it:GS1} If $s(n)= o(n)$, then there is an LDP with rate function~$\chi_W$.
    \item\label{it:GSn} If $s(n)=(c+o(1))n$ for some $c\in (0,\infty)$ then there is an LDP with rate function~$\frac1c\,K_{W}$.
	\item\label{it:GS2} If $s(n)\in o(n^2)\cap \omega(n)$, then there is an LDP with rate function~$\chi_{\RN(W)}$.
	\item\label{it:GS3} If $s(n)= \omega(n^2)$, then there is an LDP with rate function~$\chi_{\FORB(W)}$.
\end{enumerate}
\end{theorem}

It is easy to see that Theorem~\ref{th:CV} implies an LDP with rate function $\frac1c\, h_p$ if $W=p$ a.e.\ and $s(n)=(c+o(1))n^2$.

The results of Theorem~\ref{thm:GeneralSpeeds} in the cases when $s(n)=o(n^2)$  are proved by first establishing an LDP for \emph{weighted graph sampling}, where we take independent uniform elements $x_1,\dots,x_n\in [0,1]$ and output the corresponding edge weighted complete graph $H$ on $[n]$, where $w_H(\{i,j\})$, the weight on the edge $\{i,j\}$, is equal to $W(x_i,x_j)$. 
Let $(\tH_{n,W})_{n\in\I N}$ be the corresponding distribution on $\tW$ induced by taking the equivalence class $\tf{H}\in \tW$ of the sampled weighted graph~$H$, where $\f{H}$ is the $n$-step graphon with parts of measure $1/n$ whose values are given by the weights of~$H$. Our Lemma~\ref{lm:ExpEq} implies via basic results on exponential equivalence that, for speeds $o(n^2)$, an LDP for $(\tH_{n,W})_{n\in\I N}$ automatically gives an LDP for $(\tR_{n,W})_{n\in\I N}$ with the same rate function, and vice versa. Working with $(\tH_{n,W})_{n\in\I N}$ is much more convenient and we completely resolve the LDP problem for this model in Theorem~\ref{thm:GeneralSpeedsH}. In general, the original sequence 
$(\tR_{n,W})_{n\in\I N}$ exhibits a different LDP behaviour than $(\tH_{n,W})_{n\in\I N}$ for speeds $\Omega(n^2)$ since the second  step (edge rounding) may result in $\Omega(n^2)$ ``atypically'' rounded edges with probability $\me^{-\Theta(n^2)}$. In the case when $W$ is $\{0,1\}$-valued (when each edge weight of $H\sim \tH_{n,W}$ is 0 or 1), the edge rounding step involves no randomness and $\tR_{n,W}=\tH_{n,W}$. Thus Theorem~\ref{thm:GeneralSpeedsH} fully resolves the LDP problem for $(\tR_{n,W})_{n\in\I N}$ for such graphons. More generally, the LDP problem for $(\tR_{n,W})_{n\in\I N}$ is fully resolved for any $W$ with $\RN(W)=\FORB(W)$, see the discussion in Section~\ref{se:concluding}.



\hide{
\begin{theorem}\label{thm:RNo(n)}
Let $W\in\C W$ and $s:\I N\to(0,\infty)$.
If $s(n)\in \omega(n)$, then $(\tH_{n,W})_{n\in\I N}$ on $(\tW,\delta_\Box)$ satisfies an LDP with speed $s(n)$ and rate function~$\chi_{\RN(W)}$.
\end{theorem}

\begin{theorem}\label{thm:speed 1/n}
Let $W\in\C W$.
Then the function $K_W:\tW\to[0,\infty]$ is lower semi-continuous with respect to the metric $\delta_\Box$.
Moreover, the sequences of measures $(\tR_{n,W})_{n\in\I N}$ and $(\tH_{n,W})_{n\in\I N}$ on $(\tW,\delta_\Box)$ satisfy an LDP with speed $n$ and rate function~$K_{W}$.
\end{theorem}

Together with Theorems~\ref{thm:GeneralSpeeds},~\ref{th:ourLDP} this completely describes LDPs in step graphon random graph model and $\{0,1\}$-graphon random graph models.
The latter case follows from Theorem~\ref{thm:RNo(n)} together with the fact that $\tR_{n,W}=\tH_{n,W}$ for every $n\in \I N$.
}

\subsection{LDP for step graphons for speed $\Theta(n^2)$}\label{Intro:n^2}

Here (in Theorem~\ref{th:ourLDP}) we prove an LDP for $(\tR_{n,W})_{n\in\I N}$ and speed $(c+o(1))n^2$,
when $W$ is a step graphon. Suppose that $W$ has $k$ non-null parts $A_1,\dots,A_k$ and its values are encoded by a symmetric  $k\times k$-matrix $\V p=(p_{i,j})_{i,j=1}^k$. 
It comes as no surprise that if we consider large deviations for $(\tR_{n,W})_{n\in\I N}$ at speed 
$\Theta(n^2)$ then the rate function depends only on $(p_{i,j})_{i,j\in [k]}$ but not on the (non-zero) measures of the parts $A_i$.

We need some preparation to define the rate function for the given matrix $\V p$. For a non-zero real $k$-vector $\V\alpha=(\alpha_1,\dots,\alpha_k)\in[0,\infty)^k\NoZero$, let $(\cI{\V\alpha}{1},\dots,\cI{\V\alpha}{k})$ denote the partition of $[0,1]$ into consecutive intervals such that each interval $\cI{\V\alpha}{i}$ has length $\alpha_i/\|\V\alpha\|_1$. Also, define the function $J_{\V\alpha,\V p}:\tW\to [0,\infty)$ by
 \beq{eq:JAlphaP}
J_{\V\alpha,\V p}(\T U):=\inf_{V\in \T U}\,  \frac12\sum_{i,j\in [k]} \int_{\cI{\V\alpha}{i}\times \cI{\V\alpha}{j}} h_{p_{i,j}}(V(x,y))\dd x\dd y,\quad U\in {\C W}.
 \eeq
 Note that the function $J_{\V\alpha,\V p}$ will not change if we multiply the vector $\V\alpha$ by any positive scalar. Finally, define
	\beq{eq:R}
	R_{\V p}(\T U):=\inf_{\V\alpha\in [0,1]^k\atop \alpha_1+\ldots+\alpha_k=1} J_{\V\alpha,\V p}(\T U),\quad U\in\C W.
	\eeq

\begin{theorem}\label{th:ourLDP}
	Let $W$ be a $k$-step graphon with $k$ non-null parts whose values are encoded by a symmetric $k\times k$ matrix $\V p\in [0,1]^{k\times k}$.	Let $s:\I N\to (0,\infty)$ be any function such that, for some $c\in (0,\infty)$, we have $s(n)=(c+o(1))n^2$.

Then the function $R_{\V p}:\tW\to [0,\infty]$ is lower semi-continuous with respect to the metric~$\delta_\Box$. Moreover, the sequence of measures $(\tR_{n,W})_{n\in\I N}$ on $(\tW,\delta_\Box)$ satisfies an LDP with speed $s(n)$ and rate function~$\frac1c\,R_{\V p}$.
	\end{theorem}

For $k=1$, this gives the LDP result of Chatterjee and Varadhan~\cite{ChatterjeeVaradhan11} (that is, Theorem~\ref{th:CV}).

\subsection{Stochastic block models}\label{se:SBM}

Sampling from a step graphon is closely related to a \emph{stochastic block model}, which provides another important generalisation of the binomial random graph $\I G(n,p)$ and which we define now. Here we are given
a  symmetric $k\times k$ matrix $\V p=(p_{i,j})_{i,j\in [k]}\in [0,1]^{k\times k}$ with entries in~$[0,1]$ and an integer vector $\V a=(a_1,\dots,a_k)\in\NZ^k$, where 
$
\NZ:=\{0,1,2,\ldots\}
$
 denotes the set of non-negative integers. The corresponding probability distribution on $\tW$ is defined as follows.  Let $n$ be equal to $\|\V a\|_1=a_1+\dots+a_k$ and  let $(A_1,\dots, A_k)$ be the partition of $[n]$ into $k$ consecutive intervals with $A_i$ having $a_i$ elements. The random graph $\I G(\V a,\V p)$ on $[n]$ is produced by making each pair $\{x,y\}$ of $[n]$ an edge with probability $p_{i,j}$ where $i,j\in [k]$ are the indices with $x\in A_i$ and $y\in A_j$, with all choices made independently of each other. Output the weak isomorphism class $\tf{G}\in\tW$ of the graphon $\f{G}$ of the generated graph~$G\sim \I G(\V a,\V p)$ on~$[n]$. Let $\tP_{\V a,\V p}$ be the correponding (discrete) probability measure on~$\tW$.

Let us point out the difference between the random graphs $\I G((a_1,\dots,a_k),(p_{i,j})_{i,j\in [k]})$ and $\I G(a_1+\dots+a_k,W)$ where $W$ is a $k$-step graphon whose values are encoded by a matrix $(p_{i,j})_{i,j\in [k]}$.
In the former model, we have exactly $a_i$ vertices in the $i$-th block for each $i\in [k]$. In the latter model, each vertex is put into one of the $k$ blocks with the probabilities given by the measures of $A_1,\dots,A_k$, independently of the other vertices; thus the number
of vertices in each block is binomially distributed. Of course, if $k=1$, then both models coincide with the binomial random graph $\I G(a_1,p_{1,1})$.

Borgs, Chayes, Gaudio, Petti and Sen~\cite{BCGPS} extended Theorem~\ref{th:CV} to $k$-block stochastic models as follows. For a real-valued function $I$ on a topological space $X$, let its \emph{lower semi-continuous
	regularization} be
\beq{eq:lscR}
 I_{\mathrm{lsc}}(x):=\sup\left\{\inf_{y\in G} I(y)\mid \mbox{$G\subseteq X$ is open and }G\ni x\right\},\quad x\in X.
\eeq 
 As it is well-known (see e.g.\ \RS{Lemma~2.8}), 
 $I_{\mathrm{lsc}}$ is lower semi-continuous  and, in fact, $I_{\mathrm{lsc}}$ is the largest lsc function with $I_{\mathrm{lsc}}\le I$. 

 \begin{theorem}[Borgs, Chayes, Gaudio, Petti and Sen~\cite{BCGPS}]
 	\label{th:BCGPS}
 	Let $\V\alpha\in \NZ^k\NoZero$ be a non-zero integer $k$-vector and let $\V p\in [0,1]^{k\times k}$ be a symmetric $k\times k$ matrix. Then the sequence of measures $(\tP_{n\V\alpha,\V p})_{n\in\I N}$ on $(\tW,\delta_\Box)$ satisfies an LDP with speed $(n\,\|\V\alpha\|_1)^2$ and rate function~$(J_{\V\alpha,\V p})_{\mathrm{lsc}}$.\qqed
 	\end{theorem}

Our paper uses Theorem~\ref{th:BCGPS} to derive Theorem~\ref{th:ourLDP}, even though we initially proved Theorem~\ref{th:ourLDP} independently of the work by Borgs et al~\cite{BCGPS}, by first proving an LDP for what we call \emph{$k$-coloured graphons} (that are defined in Section~\ref{ColGraphons}). Since our original proof of Theorem~\ref{th:ourLDP} is quite long and shares many common steps with the proof from~\cite{BCGPS} (with both being built upon the method of
Chatterjee and Varadhan~\cite{ChatterjeeVaradhan11}), we decided to derive Theorem~\ref{th:ourLDP} from the results in~\cite{BCGPS} with a rather short proof. An extra benefit of this approach is that we also improve two aspects of the LDP of Borgs et al~\cite{BCGPS} as follows.

First, we prove that the function $J_{\V\alpha,\V p}:\tW\to [0,\infty]$ is  lower semi-continuous (so, in particular, there is no need to take
 the lower semi-continuous
 regularization in Theorem~\ref{th:BCGPS}):
 
 \begin{theorem}\label{th:JLSC}
 For every symmetric matrix $\B p\in [0,1]^{k\times k}$ and every non-zero real $k$-vector $\V \alpha\in [0,\infty)^k$, the function $J_{\V\alpha,\V p}:\tW\to[0,\infty]$ is lower semi-continuous with respect to the metric~$\delta_\Box$.
 \end{theorem}

Note that if no entry of the fixed matrix $\V p$ is 0 or 1 then the lower semi-continuity of $J_{\V\alpha,\V p}$ follows from a (more general) result of Markering~\cite[Theorem~2.3]{Markering23}.

Second, we extend Theorem~\ref{th:BCGPS} by allowing the fraction of vertices assigned to a part to depend on $n$ as long as it converges to any finite (possibly irrational) limit.

\begin{theorem}\label{th:GenLDP}
	Fix 
	any symmetric $k\times k$ matrix $\V p\in [0,1]^{k\times k}$ and a non-zero real $k$-vector $\V\alpha=(\alpha_1,\dots,\alpha_k)\in [0,\infty)^k\NoZero$.  Let
	$$
	\V a_n=(a_{n,1},\dots,a_{n,k})\in \NZ^k\NoZero,\quad \mbox{for $n\in\I N$},
	$$ 
	be arbitrary non-zero integer $k$-vectors  such that $\lim_{n\to\infty} a_{n,i}/n=\alpha_i$ for each $i\in [k]$. 
	Then the sequence of measures $(\tP_{\V a_n,\V p})_{n\in\I N}$ on $(\tW,\delta_\Box)$ satisfies an LDP with speed $\|\V a_n\|_1^2$ and rate function~$J_{\V\alpha,\V p}$.
\end{theorem}

Note that an earlier (unpublished) manuscript by the authors \cite{GrebikPikhurko:LDPStepW} contained Theorems~\ref{th:JLSC} and~\ref{th:GenLDP}, with the same proofs as are presented here. Thus it is fully superseded by this paper.

\subsection{Organisation of the paper}

This paper is organised as follows.

In Section~\ref{prelim} we give further definitions (repeating some definitions from the Introduction) and provide some standard or easy results that we will need later. 

In Section~\ref{sec:KW}, we present basic results concerning the function $K_W$ that was defined in~\eqref{eq:KW}. In Section~\ref{sec:H_{n,W}}, we resolve all cases of the LDP problem for the measures $\tH_{n,W}$ coming from the weighted graph sampling. In Section~\ref{sec:P vs H}, we 
show that the sequences $(\tH_{n,W})_{n\in\I N}$ and $(\tR_{n,W})_{n\in\I N}$ are exponentially equivalent for any speed~$s(n)=o(n^2)$. This automatically gives LDPs for $(\tR_{n,W})_{n\in\I N}$ in Theorem~\ref{thm:GeneralSpeeds} for speeds~$o(n^2)$. The remaining unproved case of Theorem~\ref{thm:GeneralSpeeds}, namely when $s(n)/n^2\to\infty$ is done in Section~\ref{se:RnW}.

The rest of the paper deals with speeds in $\Theta(n^2)$.
Section~\ref{ColGraphons} introduces $k$-coloured graphons.
They are used in Section~\ref{Rate} to prove that the functions $J_{\V a,\V p}$ and $R_{\V p}$ are lower semi-continuous (with this result stated as Theorem~\ref{th:lsc}). The large deviation principles stated in Theorems~\ref{th:GenLDP} and~\ref{th:ourLDP} are proved in Section~\ref{GenLDP} and~\ref{ourLDP} respectively.


\section{Preliminaries}\label{prelim}

Recall that $\I N$ (resp.\ $\NZ$) denotes the set of positive (resp.\ non-negative) integers.
For $n\in \I N$, we write $[n]$ for the set $\{1,\dots n\}$. 
Also, recall that $\L{x}:=x\log x$ for $x\in [0,\infty]$ was defined by~\eqref{eq:L}. All logarithms in this paper are base~$\me$.

Let $X$ be a set.
Its \emph{indicator function} $\I 1_X$ assumes value $1$ for every $x\in X$ and 0 otherwise. Recall the related function $\chi_X$ which is $0$ on $X$ and $\infty$ elsewhere. The $k$-fold product of $X$ is denoted by~$X^k$. Also, ${X\choose k}:=\{Y\subseteq X\mid |Y|=k\}$ denotes the set of all $k$-subsets of~$X$.

A measurable space $(\Omega,\C A)$ is called \emph{standard}
if there is a Polish topology on $\Omega$ whose Borel $\sigma$-algebra is equal to~$\C A$. Given a measure
$\mu$ on $(\Omega,\C A)$, we call a subset of $\Omega$ \emph{$\mu$-measurable} (or just \emph{measurable} if $\mu$ is understood) if it belongs to the \emph{completion} of $\C A$
by $\mu$, that is, the $\sigma$-algebra generated by $\C A$ and $\mu$-null sets.
We will usually omit $\sigma$-algebras from our notation.

Unless specified otherwise, the interval $[0,1]$ of reals is always equipped with the Lebesgue measure, denoted by~$\lambda$. By $\mu^{\oplus k}$ we denote the completion of the $k$-th power of a measure $\mu$. For example, $\lambda^{\oplus k}$ is the Lebesgue measure on~$[0,1]^k$.

Denote as $\Ak$ the set of all ordered partitions of $[0,1]$ into $k$ measurable sets.
For a non-zero vector $\V\alpha=(\alpha_1,\dots,\alpha_k)\in [0,\infty)^k$, we let $\Aalpha\subseteq\Ak$ to be the set of all ordered partitions of $[0,1]$ into $k$ measurable sets such that the $i$-th set has Lebesgue measure exactly~$\alpha_i/\|\V\alpha\|_1$. Also, $(\cI{\V\alpha}{1},\dots,\cI{\V\alpha}{k})\in\Aalpha$ denotes the partition of $[0,1]$ into consecutive intervals whose lengths are given by~$\V\alpha/\|\V\alpha\|_1$  (where each dividing point is assigned  to e.g.\ its right interval for definiteness). If $\V a\in(0,\infty)^n$ consists of $n$ equal entries, then we abbreviate $\cI{\V\alpha}{j}$ to $I^n_j$. In other words, $I^n_j:=[\frac{j-1}{n},\frac{j}{n})$ for $j\in [n-1]$ and $I^n_n:=\left[\frac{n-1}n,1\right]$.

\arxiv{
We will need the following result. 
\hide{ (whose proof can be found in e.g.~\cite[Theorem 3.4.23]{Srivastava98cbs}).

\begin{theorem}[Isomorphism Theorem for Measure Spaces]\label{th:ISMS}
	For every two atomless standard probability spaces, 
	there is a measure-preserving Borel isomorphism between them.
\end{theorem}	
}%

\begin{theorem}\label{th:ISMS}
	For every two atomless standard measure spaces 
$(\Omega,\mu)$ and $(\Omega',\mu')$, 
and Borel subsets $A\subseteq \Omega$ and $A'\subseteq \Omega'$ with $0<\mu(A)=\mu'(A')<\infty$,
	there is a measure-preserving Borel isomorphism between $A$ and~$A'$.
\end{theorem}

\begin{proof} The case when $A=\Omega$ and $A'=\Omega'$ amounts to the Isomorphism Theorem for Measure Spaces, whose proof can be found in e.g.~\cite[Theorem 3.4.23]{Srivastava98cbs}. The general case follows by restricting everything to $A$ and $A'$, and noting that the obtained measure spaces are still standard by e.g.\ \cite[Theorem 3.2.4]{Srivastava98cbs}.\end{proof}	
}{}

We write $\fG_n$ for the set of  graphs on the vertex set $[n]$. Thus, for example, $|\fG_n|=2^{{n\choose 2}}$. For a graph $G=(V,E)$, its \emph{complement} is $\O G:=\left(V,{V\choose 2}\setminus E\right)$.

Also, let $\fH_n$ denote the set of all edge-weighted complete graphs on the vertex set $[n]$ with weights from $[0,1]$.
If $H\in \mathcal{H}_n$, then we write $w_H:{[n]\choose 2}\to [0,1]$ for the edge-weight function of~$H$, often abbreviating $w_H(\{i,j\})$ to $w_H(i,j)$.

\subsection{Graphons}\label{graphons}

A \emph{graphon} $U$ is a function  $U:[0,1]^2\to[0,1]$ which is \emph{symmetric} (that is,
$U(x,y)=U(y,x)$ for all $x,y\in [0,1]$) and \emph{measurable} (that is, for every $a\in \I R$, the level set $\{W\le a\}$ is a Lebesgue measurable subset of $[0,1]^2$). 
Recall that we denote the set of all graphons by~$\mathcal{W}$.  We call a graphon $W$ \emph{constant} if there is $c\in [0,1]$ such that the measure of $[0,1]^2\setminus W^{-1}(c)$ is~0.

We define the pseudo-metric $d_\Box:\C W^2\to [0,1]$ by 
\beq{eq:CutNorm}
d_\Box(U,V):=\sup_{A,B\subseteq [0,1]} \left|\int_{A\times B}\left(U-V\right) \dd\lambda^{\oplus 2}\right|,\quad U,V\in \mathcal{W},
 \eeq
where the supremum is taken over all pairs of measurable subsets of $[0,1]$.
For a function $\phi:[0,1]\to [0,1]$ and a graphon $U$, 
the \emph{pull-back} $U^\phi$ of $U$ along $\phi$ is defined 
by 
$$
 U^{\phi}(x,y):=U(\phi(x),\phi(y)),\quad x,y\in [0,1].
 $$
The \emph{cut-distance} $\delta_\Box:\C W^2\to [0,1]$ can be defined as
\beq{eq:CutDistance}
\delta_{\Box}(U,V):=\inf_{\phi,\psi} d_\Box (U^\phi,V^\psi),
\eeq
where the infimum is taken over all measure-preserving maps $\phi,\psi:[0,1]\to [0,1]$.
We refer the reader to \Lo{Section~8.2} for more details and, in particular, to \Lo{Theorem 8.13} for some alternative definitions that give the same distance. It can be easily verified that $\delta_\Box$ is a pseudo-metric on~$\C W$.
Recall that we denote the metric quotient by $\tW$ and  the equivalence class of $U\in \mathcal{W}$ by~$\widetilde{U}$. For brevity, we may write $W$ instead of $\T W$ when the stated property is independent of the choice of a representative in~$\T W$.

Define the \emph{homomorphism density} of a finite graph  $F$ on $[n]$ in a graphon $W$ by
$$
t(F,W):=\int_{[0,1]^{n}} \prod_{\{i,j\}\in E(F)} W(x\ind{i},x\ind{j}) \dd\lambda^{\oplus n}.
$$
This can be easily related to the induced homomorphism density defined in~\eqref{eq:TInd}: namely, $t(F,W)$ is the sum of $\tind(F',W)$ over all graphs $F'$ on $[n]$ with $E(F')\supseteq E(F)$.
Also, the random $n$-vertex sample $\I G(n,W)$ (and the corresponding measure $\iR_{n,W}$ on $\W$) can be equivalently defined as the random graph on $[n]$ taking each possible value $F$ with probability $\tind(F,W)$.

We will need the following auxiliary results. The first one, due to Lov\'asz and Szegedy~\cite{LovaszSzegedy07gafa}*{Theorem~5.1} (see also \Lo{Theorem 9.23}), states that the metric space $(\tW,\delta_\Box)$ is compact. 
\arxiv{This result is also the special case $k=1$ of our more general Theorem~\ref{th:compact colorod graphon} whose proof closely follows that of~\cite{LovaszSzegedy07gafa} and is sketched here.}{}


The next result is a consequence of 
the Counting Lemma and the Inverse Counting Lemma in  Borgs et al~\cite[Theorem 2.7(a)--(b)]{BCLSV08} (see also~\Lo{Lemma~10.32 and Theorem~11.3}). It gives that the $\delta_\Box$-topology on $\tW$ is generated by the homomorphism density maps $W\mapsto t(F,W)$ over all graphs~$F$. 

\begin{theorem}\label{th:ConvOnW} A sequence of graphons $(W_n)_{n\in\I N}$ converges to $W\in\W$ in the cut-distance if and only if $\lim_{n\to\infty} t(F,W_n) =t(F,W)$ for every graph~$F$.\qqed\end{theorem}

We will also need the following result. \arxiv{}{Since we could not find this result the literature, we present its (routine but technical) proof in the arxiv version of this paper~\cite{GrebikPikhurko23arxiv}.}

\begin{lemma}\label{lm:AlmostMP}
	Let $U$ be a graphon and $\phi:[0,1]\to [0,1]$ be a measurable function such that the \emph{push-forward measure} $\phi_*\lambda$ (defined by $(\phi_*\lambda)(X):=\lambda(\phi^{-1}(X))$ for measurable $X\subseteq [0,1]$) satisfies $\phi_*\lambda\ll \lambda$, that is, $\phi_*\lambda$
 is absolutely continuous with respect to the Lebesgue measure~$\lambda$. Then  $U^\phi$ is a graphon. Moreover, 
	if the Radon-Nikodym derivative $D:=\frac{\dd (\phi_*\lambda)}{\dd \lambda}$ satisfies $D(x)\le 1+\e$ for $\lambda$-a.e.\ 
	$x\in[0,1]$ then
	$
	\delta_\Box(U,U^\phi)\le 2\e.
	$
\end{lemma}

\arxiv{
\begin{proof}  
		The function $U^\phi:[0,1]^2\to [0,1]$ is clearly symmetric so we have to show that it is measurable.
	Since the pre-image under $\phi$ of any $\lambda$-null set is again $\lambda$-null by our assumption $\phi_*\lambda\ll \lambda$, there is a Borel map  $\psi:[0,1]\to[0,1]$  such that the set 
	$$X:=\{x\in [0,1]:\psi(x)\not=\phi(x)\}$$ 
	is $\lambda$-null. (For a proof, see e.g.~\cite[Proposition~2.2.5]{Cohn13mt}.)
Take any $\rho\in\I R$. The set 
$$
 A:=\{(x,y)\in [0,1]^2\mid U(x,y)\le \rho\}
 $$ 
 is measurable so by e.g.~\cite[Proposition~1.5.2]{Cohn13mt} there are $B,N\subseteq [0,1]^2$ such that $B$ is Borel, $N$ is $\lambda^{\oplus 2}$-null and $A\bigtriangleup B\subseteq N$, where $A\bigtriangleup B:=(A\setminus B)\cup (B\setminus A)$ denotes the \emph{symmetric difference} of the sets $A$ and $B$. The pre-image of $N$ under the Borel map $\psi^{\oplus2}(x,y):=(\psi(x),\psi(y))$ is also $\lambda^{\oplus 2}$-null for otherwise
 this would contradict the absolute continuity $(\phi_*\lambda)^{\oplus 2}\ll \lambda^{\oplus 2}$ (which follows from
$\phi_*\lambda\ll \lambda$ by the Fubini-Tonelli Theorem for Complete Measures).
 Thus the level set 
 $\{U^\phi\le \rho\}$
 is Lebesgue measurable since its symmetric difference with the Borel set $(\psi^{\oplus2})^{-1}(B)$ is a subset of the null set $(\psi^{\oplus2})^{-1}(N)\cup (X\times [0,1])\cup ([0,1]\times X)$. As $\rho\in\I R$ was arbitrary, $U^\phi$ is a measurable function and thus a graphon.

For the second part, it will be convenient to use the following generalisation of a graphon.
Namely, by
a \emph{generalised graphon} we mean a triple $(V,\Omega,\mu)$ where $(\Omega,\mu)$ is an atomless standard probability space  and $V:(\Omega^2,\mu^{\oplus 2})\to [0,1]$ is a symmetric measurable function. In the special case $(\Omega,\mu)=([0,1],\lambda)$ we get the (usual) notion of a graphon. Most definitions
and results extend with obvious modifications from graphons to generalised graphons  (see \Lo{Chapter~13.1} for details).
For example, the $n$-vertex random sample from $(V,\Omega,\mu)$ is obtained
by taking $\mu$-distributed $x_1,\dots,x_n\in \Omega$ and making $\{i,j\}$ an edge with probability $V(x_i,x_j)$, with all choices being independent.
In particular, we will need the facts that if $(V,\Omega,\mu)$ is a generalised graphon and
$\phi: (\Omega',\mu')\to(\Omega,\mu)$ is a measure-preserving  map between standard probability spaces,
then the function $V^\phi$ is measurable (which can be proved by adapting the proof of the first part of the lemma) and
\beq{eq:51}
\delta_\Box((V,\Omega,\mu),(V^\phi, \Omega',\mu'))=0,
\eeq
where we define $V^\phi(x,y):=V(\phi(x),\phi(y))$ for $x,y\in \Omega'$ and $\delta_\Box$ is the extension of the cut-distance to generalised graphons via the obvious analogues of~\eqref{eq:CutNorm} and~\eqref{eq:CutDistance}.

Let us return to the proof of the second part of the lemma.
	We can assume that the set $\{x\in [0,1]\mid D(x)\not=1\}$ has positive measure for otherwise $\phi$ is a measure-preserving map and $\delta_\Box(U,U^\phi)=0$ by~\eqref{eq:51}, as required. 
	Let $(V,[0,1]^2,\lambda^{\oplus 2})$ be the generalised graphon defined by $V((x,y),(x',y')):=U(x,x')$, for $x,y,x',y'\in [0,1]$. Thus $V=U^\pi$, where
	$\pi:[0,1]^2\to [0,1]$ is the (measure-preserving) projection on the first coordinate. 
	By~\eqref{eq:51}, it holds that
	$$
	\delta_\Box ((V,[0,1]^2,\lambda^{\oplus 2}),U)=0.
	$$
	
	By changing $D$ on a $\lambda$-null set, we can make it a Borel function with $D(x)\le 1+\e$ for every $x\in [0,1]$. Then
	$$
	\Omega:=\{(x,y)\in [0,1]\times \I R\mid 0\le y\le D(x)\}
	$$ 
	is a Borel subset of $[0,1]\times \I R$ (see e.g.\ \cite[Example 5.3.1]{Cohn13mt}) and thus induces a standard measurable space. Let $\mu$ be the restriction of the Lebesgue measure on $\I R^2$ to~$\Omega$. Define $W:\Omega^2\to [0,1]$ by $W((x,y),(x',y')):=U(x,x')$ for $(x,y),(x',y')\in \Omega$. Thus $U^\phi$ and $(W,\Omega,\mu)$ are measure-preserving pull-backs of the generalised graphon $(U,[0,1],\phi_*\mu)$ 
along respectively the map $\phi$ and the projection $\Omega\to[0,1]$ on the first coordinate.
	 Therefore we have by~\eqref{eq:51} and the Triangle Inequality for $\delta_\Box$ that
	$$
\delta_\Box(U^\phi,(W,\Omega,\mu))\le \delta_\Box(U^\phi,(U,[0,1],\phi_*\mu))+\delta_\Box((U,[0,1],\phi_*\mu),(W,\Omega,\mu))=0.
$$  
	
	Thus it suffices to show that the cut-distance $\delta_\Box$ between $(V,[0,1]^2,\lambda^{\oplus 2})$ and $(W,\Omega,\mu)$ is at most~$2\e$.
	The functions $V$ and $W$ and the measures $\lambda^{\oplus 2}$ and $\mu$
	coincide on~$X^2$, where $X:=[0,1]^2\cap \Omega$. 
	Since $D\not= 1$ on a set of positive measure, it holds that $\mu(X)<1$. 
The Borel subsets $[0,1]^2\setminus X$ and $\Omega\setminus X$ of $\I R^2$ have the same positive Lebesgue measure and thus,
	by Theorem~\ref{th:ISMS}, there is a Borel measure-preserving bijection $\psi$ between them. By letting $\psi$ be the identity function on $X$, we get a Borel measure-preserving bijection $\psi:\Omega\to [0,1]^2$. 
	Since $D\le 1+\e$, we have that $\Omega\setminus X\subseteq [0,1]\times [1,1+\e]$ has measure at most~$\e$. 	
The function $W$ and the pull-back $V^\psi$, as maps $\Omega^2\to [0,1]$, coincide on the set $X^2$ of measure at least $(1-\e)^2\ge 1-2\e$. It follows that the $d_\Box$-distance between them is at most $2\e$. (Indeed, when we compute it via the analogue of~\eqref{eq:CutNorm}, the integrand is bounded by 1 in absolute value and is non-zero on a set of measure at most~$2\e$.) This finishes the proof of the lemma.\end{proof}
}{}

Informally speaking, the following result states that if we delete a small subset of $[0,1]$ and stretch the rest of a graphon uniformly then the new graphon is close to the original one.

\begin{lemma}\label{lm:Delete01}
	Let $U\in \C W$ be a graphon, $s\in (0,1]$ be a non-zero real and $\phi:[0,1]\to [0,1]$ be the map that sends $x$ to~$sx$. 
	Then $U^\phi$ is a graphon and $\delta_\Box(U,U^\phi)\le 2(\frac1s -1)$.
\end{lemma}
\begin{proof}  Clearly, the push-forward $\phi_*\lambda$ is the uniform probability measure on $[0,s]$ so the Radon-Nikodym derivative $\frac{\dd( \phi_*\lambda)}{\dd\lambda}$ is a.e.\ $1/s$ on~$[0,s]$ and $0$ on~$[s,1]$. The result now follows from Lemma~\ref{lm:AlmostMP}. 
\end{proof}

\arxiv{
\begin{remark} Without affecting $(\tW,\delta_\Box)$, we could have defined a graphon as a Borel symmetric function $[0,1]^2\to [0,1]$ and required that the measure-preserving maps in the definition of $\delta_\Box$ are Borel. Then some parts could be simplified (for example, the first claim of Lemma~\ref{lm:AlmostMP} would not be necessary as the function $U^\phi$ would be Borel for every Borel~$\phi$). However, 
		we prefer to use the (now standard) conventions from Lov\'asz' book~\cite{Lovasz:lngl}.			
\end{remark}
}{}

\subsection{Topology}\label{se:Topology}

The (pseudo-)metrics on the graphon space $\W$ that we use are $d_{\Box}$, $d_1$ and $d_{\infty}$ with the last two coming from the $L^1$ and $L^\infty$ norms respectively, where we view $\W$ as a subset of the vector space of bounded measurable functions on $[0,1]^2$. When we consider $[0,1]^n$, we use the $L^\infty$-metric, denoting it also by~$d_\infty$: $d_\infty(x,y)=\max\{|x_i-y_i|\mid i\in [n]\}$ for $x,y\in [0,1]^n$.
Let 
 $$
 B_d(x,r):=\{y\mid d(x,y)<r\}\quad \mbox{and}\quad \O B_d(x,r):=\{y\mid d(x,y)\le r\}
 $$ 
 denote respectively the open and closed balls of radius $r$ around a point $x$ in a metric~$d$. 
If it is clear from the context we write, e.g., $B_\infty(x,r)$ instead of $B_{d_\infty}(x,r)$, etc.


Let $X$ be a topological space and $\mu$ be a Borel probability measure on~$X$. Recall that $L^1(X,\mu)$ is the space of all integrable functions, $L^\infty(X,\mu)$ is the space of all essentially bounded functions, $C(X)$ is the space of all continuous functions and $\mathcal{P}(X)$ is the space of all Borel probability measures on~$X$.
Some of the topologies that will consider on these spaces are as follows:
\begin{itemize}
	\item the \emph{weak$^*$ topology on $\mathcal{P}(X)$}, where $\nu_n\xrightarrow{weak^*} \nu$ if and only if $\int_{X} f \dd\nu_n\to \int_{X} f \dd\nu$ for every $f\in C(X)$,
		\item the \emph{weak$^*$ topology on $L^1(X,\mu)$}, where $f_n\xrightarrow{weak^*} f$ if and only if $\int_{X} f_ng \dd\mu\to \int_{X} fg \dd\mu$ for every $g\in C(X)$,
	\item the \emph{$\tau^k$-topology on $\mathcal{P}(X)$} for given $k\in \I N$, which is the coarsest topology such that, for every bounded Borel function $a:X^k\to \mathbb{R}$, the map $\mu\mapsto \int_{X^k} a \dd\mu^{\oplus k}$ is continuous,
	\item the \emph{weak topology on $L^1(X,\mu)$}, where $f_n\xrightarrow{weak} f$ if and only if $\int_{X} f_ng \dd\mu\to \int_{X} fg \dd\mu$ for every $g\in L^\infty(X,\mu)$.
\end{itemize}
  By the \emph{weak} topology on $\{\mu\in \mathcal{P}([0,1])\mid \mu\ll \lambda\}$, we will mean the topology inherited from the weak topology on $L^1([0,1],\lambda)$ via the natural embedding $\mu\mapsto \frac{\dd\mu}{\dd\lambda}$. 
  On this set, the weak topology coincides with the $\tau^1$-topology, since every measurable function can be made Borel by changing it on a null-set. 
  Also, note that if we view each element $f$ of  $L^1(X,\mu)$ as a Borel measure $f \dd\mu$ defined by $(f \dd\mu)(A):=\int_{A} f(x) \dd\mu(x)$ on $X$, then on the intersection with $\mathcal{P}(X)$, the corresponding weak$^*$ topologies coincide.

Some classical results that we will need are collected in Appendix~\ref{app:Func}.

\section{Properties of the function $K_W$}\label{sec:KW}

We start with some basic observations concerning the function $K_W$ that was defined in~\eqref{eq:KW}.

\begin{proposition}\label{pr:K}
Let $W,U,V\in \C W$.
Then we have the following:
\begin{enumerate}[(a)]
	\item\label{it:KNonNeg} $K_W$ is non-negative,
	\item\label{it:KZero} $K_W(U)=0$ if and only if $\delta_\Box(W,U)=0$,
	\item\label{it:KDense} $K_W^{-1}([0,\infty))$ is a dense subset of $\RN(W)$ with respect the $\delta_\Box$-metric,
	\item\label{it:KLsc} $K_W:(\tW,\delta_\Box)\to [0,\infty]$ is a lower semi-continuous function,
	\item\label{it:KNonConst} if $W$ is not a constant graphon, then $\RN(W)$ contains at least two (not weakly isomorphic) graphons and	there is $U\in \C W$ such that $K_W(U)\in (0,\infty)$.
\end{enumerate}
\end{proposition}
\begin{proof} Recall that $\L{x}:=x\log(x)$ for $x\in[0,\infty]$. 
Its restriction to $[0,\infty)$ is a convex continuous function.

\ProofOf{it:KNonNeg} Since $\L{}$ is a convex function on $[0,\infty)$
and $\int\frac{\dd\mu}{\dd\lambda} \dd\lambda=1$ for every probability measure $\mu\ll\lambda$, we have by Jensen's inequality that $\int \L{\frac{\dd\mu}{\dd\lambda}} \dd\lambda\ge \L{\int\frac{\dd\mu}{\dd\lambda} \dd\lambda}=\L{1}=0$, giving the required inequality $K_W(U)\ge 0$ for every~$U$. 


\medskip
Before we proceed with the remaining parts of the proposition, we formulate an auxiliary claim.

\begin{claim}\label{cl:Jensen}
Let a measure $\mu$ and a sequence $(\mu_n)_{n\in\I N}$ of measures in $\{\nu\in\mathcal{P}([0,1])\mid\nu \ll\lambda\}$ satisfy $\mu_n\xrightarrow{weak}\mu$ (that is, 
$\frac{d \mu_n}{d \lambda}\xrightarrow{weak}\frac{\dd\mu}{\dd\lambda}$).
Then we have 
$$
\liminf_{n\to\infty} \int_{[0,1]}\L{\frac{\dd\mu_n}{\dd\lambda}}\dd\lambda\ge \int_{[0,1]}\L{\frac{\dd\mu}{\dd\lambda}}\dd\lambda,\quad\mbox{as $n\to\infty$}.
$$
\end{claim}
\begin{proof}
For $a\in [-1/\me,\infty)$ define
 $$
 Y_a:=\left\{f\in L^1([0,1],\lambda): f\ge 0\mbox{ $\lambda$-a.e.},\ 
 \int_{[0,1]} f\dd\lambda\le 1,\
 \int_{[0,1]} \L{f}\dd\lambda\le a\right\}.
 $$
 
Let us show that for every $a\in [-1/\me,\infty)$, the subset $Y_a\subseteq L^1([0,1],\lambda)$ is closed in the $L^1$-topology. Take any functions $f_n$ in $Y_a$ that converge in the $L^1$-norm to some~$f\in L^1([0,1],\lambda)$ as $n\to\infty$. Of course, $f\ge 0$ $\lambda$-a.e.\ and 
$\int_{[0,1]} f\dd\lambda\le 1$.
By passing to a subsequence, we can assume that $f_n(x)\to f(x)$ for $\lambda$-a.e.\ $x\in [0,1]$. By the continuity of the function $\L{}$, we have that $\L{f_n(x)}\to \L{f(x)}$ for $\lambda$-a.e.\ $x\in [0,1]$. Since $\L{}$ is bounded from below by $-1/\me$, it holds by Fatou's lemma that
 $$
 a\ge \liminf_{n\to\infty}\int_{[0,1]} \L{f_n}\dd\lambda \ge \int_{[0,1]} \left(\liminf_{n\to\infty} \L{f_n}\right)\dd\lambda = \int_{[0,1]} \L{f}\dd\lambda.
 $$
 Thus $f\in Y_a$ and we conclude that the set $Y_a$ is closed as required.
 
Since $\L{}$ is a convex function, the set $Y_a$ is also convex for every~$a$. Thus, by Mazur's lemma (Theorem~\ref{th:B9}), the set $Y_a$ is also weakly closed.

Let $a:=\liminf_{n\to\infty} \int_{[0,1]} \L{\frac{\dd\mu_n}{\dd\lambda}}\dd\lambda$. We have $a\ge -1/\me$. Suppose that $a<\infty$ as otherwise the claim holds trivially. For any $\e>0$ there are infinitely many $n$ with $\frac{\dd\mu_n}{\dd\lambda}\in Y_{a+\e}$; thus the weak limit $\frac{\dd\mu}{\dd\lambda}$ also belongs to the weakly closed set $Y_{a+\e}$. Since $\e>0$ was arbitrary, it holds that $\frac{\dd\mu}{\dd\lambda}\in Y_a$, as desired.\end{proof}

\ProofOf{it:KZero} If $\delta_\Box(U,W)=0$ then $K_W(U)\le 0$ (since we can take $\mu:=\lambda$ in the definition of $K_W$) while $K_W(U)\ge 0$ by Part~\ref{it:KNonNeg}, as desired. 

For the forward implication, suppose that $K_W(U)=0$.
By the definition of $K_W$, there is a sequence $(\mu_n)_{n\in \I N}$ of elements of $\mathcal{P}([0,1])$ such that $\mu_n\ll\lambda$, $\delta_\Box(W_{\mu_n},U)=0$ for every $n\in \I N$, and $\int \L{\frac{\dd\mu_n}{\dd\lambda}}\! \dd\lambda\to 0$ as $n\to\infty$.

Let $k\in \I N$.
It is easy to see that $\mu_n^{\oplus k}\ll\lambda^{\oplus k}$ and that
$$\frac{\dd\mu_n^{\oplus k}}{\dd\lambda^{\oplus k}}(x)=\prod_{i\in[k]} \frac{\dd\mu_n}{\dd\lambda}(x\ind{i})$$
holds for $\lambda^{\oplus k}$-almost every $x=(x\ind{1},\dots,x\ind{k})\in [0,1]^k$.
A direct use of Fubini-Tonelli's theorem gives
\begin{eqnarray*}
\int_{[0,1]^k} \L{\frac{\dd\mu_n^{\oplus k}}{\dd\lambda^{\oplus k}}} \dd\lambda^{\oplus k}
&=&\int_{[0,1]^k} \left(\prod_{i\in[k]} \frac{\dd\mu_n}{\dd\lambda}(x\ind{i})\right) \left(\sum_{i\in[k]} \log\left(\frac{\dd\mu_n}{\dd\lambda}(x\ind{i})\right)\right)\dd\lambda^{\oplus k} (x)\\
&=&k\int_{[0,1]} \L{\frac{\dd\mu_n}{\dd\lambda}} \dd\lambda,\quad\mbox{for every $n\in \I N$.}
\end{eqnarray*}
This tends to $0$ as $n\to\infty$ and thus is bounded. Since $\L{}\ge -1/\me$, the integrals above remain bounded if we replace $\L{}$ by the non-decreasing function $G(x):=\max\{\L{x},0\}$, $x\in [0,\infty]$. By the de la Vall\' ee-Poussin Theorem (Theorem~\ref{APP:VP}) applied to the function $G$, the set $\mathcal{F}:=\{\frac{\dd\mu_n^{\oplus k}}{\dd\lambda^{\oplus k}}\mid n\in\I N\}$ is uniformly integrable. Thus, by the Dunford--Pettis Theorem (Theorem~\ref{App:DP}) its weak closure is compact in the weak topology. By the Eberlein--\v Smulian Theorem (Theorem~\ref{APP:ES}) we can pass to a subsequence of $n$ such that $\frac{\dd\mu_n^{\oplus k}}{\dd\lambda^{\oplus k}}$ weakly converges to $\frac{d\nu_k}{\lambda^{\oplus k}}$ for some $\nu_k\in \mathcal{P}([0,1]^k)$ with $\nu_k\ll\lambda^{\oplus k}$.

By iteratively repeating the above step for each $k=1,2,\dots$ and doing diagonalisation, we can assume that, for every $k$, it holds that $\frac{\dd\mu_n^{\oplus k}}{\dd\lambda^{\oplus k}}\xrightarrow{weak}\frac{d\nu_k}{\lambda^{\oplus k}}$ as $n\to\infty$. 

Set $\mu:=\nu_1$. Let us show that, for each $k\ge 2$, the measures $\nu_k$ and $\mu^{\oplus k}$ are the same. By Dynkin's Theorem (see e.g.\ \cite[Corollary 1.6.3]{Cohn13mt}), it is enough to show that these two measures coincide on each product $A_1\times \dots\times A_k$ of Borel sets since such products are closed under intersection and generate the Borel $\sigma$-algebra of $[0,1]^k$. By weak convergence, we have
$$
\int \I 1_{A_1\times\dots\times A_k} d\nu_k=\lim_{n\to\infty} \int \I 1_{A_1\times\dots\times A_k} \dd\mu_n^{\oplus k} = \lim_{n\to\infty} \prod_{i=1}^k \int \I 1_{A_i}\dd\mu_n=  \prod_{i=1}^k  \int \I 1_{A_i}d\mu,
$$
 as desired.

Next we exploit the fact that $\delta_\Box$-convergence is equivalent to the homomorphism densities convergence (Theorem~\ref{th:ConvOnW}). 
For every graph $F=([k],E(F))$, we have by weak convergence that
\begin{equation*}
\begin{split}
t(F,U)=t(F,W_{\mu_n})= &\int_{[0,1]^{{k}}} \frac{\dd\mu_n^{\oplus k}}{\dd\lambda^{\oplus k}}(x)\prod_{\{i,j\}\in E(F)} W(x\ind{i},x\ind{j}) \dd\lambda^{\oplus k}(x) \\
\xrightarrow{n\to \infty} & \int_{[0,1]^{{k}}} \frac{\dd\mu^{\oplus k}}{\dd\lambda^{\oplus k}}(x)\prod_{\{i,j\}\in E(F)} W(x\ind{i},x\ind{j}) \dd\lambda^{\oplus k}(x)=t(F,W_\mu).
\end{split}
\end{equation*}
 Thus $\delta_\Box(U,W_\mu)=0$.
It remains to show that $\mu=\lambda$, i.e., $W_\mu=W$.
Claim~\ref{cl:Jensen} gives that $\int_{[0,1]} \L{\frac{\dd\mu}{\dd\lambda}}\dd\lambda=0$. Recall that the proof of Part~\ref{it:KNonNeg} shows that this integral is non-negative by applying Jensen's inequality to the strictly convex function $\L{x}$. Since we have equality, it must be the case that $\frac{\dd\mu}{\dd\lambda}$ is constant $\lambda$-a.e. Thus the probability measures $\mu$ and $\lambda$ coincide, proving~\ref{it:KZero}.

\ProofOf{it:KDense} It follows from the definition that $K_W^{-1}([0,\infty))\subseteq \RN(W)$.
To show that this subset is dense, it is enough to show that for every $\mu\in \mathcal{P}([0,1])$ such that $\mu\ll\lambda$ there is a sequence $(\mu_n)_{n\in \I N}\subseteq \mathcal{P}([0,1])$ such that $\delta_\Box(W_{\mu_n},W_\mu)\to 0$ and, for every $n\in \I N$, we have $\mu_n\ll\lambda$  and $K_W(W_{\mu_n})<\infty$.

Given $n\in\I N$, let $X_n:=\{x\in [0,1]:\frac{\dd\mu}{\dd\lambda}(x)\ge n\}$ and $A_n:=[0,1]\setminus X_n$. If $\lambda(X_n)=0$, then we let $\mu_n=\mu$. Otherwise let $\dd\mu_n:=g_n\dd\lambda$, where we define
$g_n:[0,1]\to\I R$ to be the constant $\frac{1}{\lambda(X_n)}\,\int_{X_n} \frac{\dd\mu}{\dd\lambda} \dd\lambda=\mu(X_n)/\lambda(X_n)<\infty$ on $X_n$ and let $g_n$ be $\frac{\dd\mu}{\dd\lambda}$ on~$A_n$.
Then we have $\int_{[0,1]} g_n \dd\lambda=1$ and $\| g_n\|_\infty \le \max\{\mu(X_n)/\lambda(X_n),n\}<\infty$  for every $n\ge 1$. Also, we can take the measure $\mu_n\ll\lambda$ in the definition of
$K_W(W_{\mu_n})$ as an infimum over a set of measures, obtaining that
$$
K_W(W_{\mu_n})\le 
\int \L{g_n}\dd\lambda\le  \L{\max\{\mu(X_n)/\lambda(X_n),n\}}<\infty.
$$
Moreover, it is easy to see that $\mu^{\oplus k}(A^{k}_n)=\mu_n^{\oplus k}(A^{k}_n)$ tends to $1$ as $n\to\infty$ for every fixed $k\in \I N$.
For any graph $F=([k],E(F))$, we have
\begin{equation*}
\begin{split}
t(F,W_{\mu_n})= & \int_{A_n^{k}} \prod_{\{i,j\}\in E(F)} W(x\ind{i},x\ind{j}) \dd\mu_n^{\oplus k}(x)+\int_{[0,1]^k\setminus A_n^{k}} \prod_{\{i,j\}\in E(F)} W(x\ind{i},x\ind{j}) \dd\mu_n^{\oplus k}(x) \\
= & \int_{A_n^{k}} \prod_{\{i,j\}\in E(F)} W(x\ind{i},x\ind{j}) \dd\mu^{\oplus k}(x)+\int_{[0,1]^k\setminus A_n^{k}} \prod_{\{i,j\}\in E(F)} W(x\ind{i},x\ind{j}) \dd\mu_n^{\oplus k}(x) \\
\xrightarrow{n\to \infty} & \int_{[0,1]^{k}} \prod_{\{i,j\}\in E(F)} W(x\ind{i},x\ind{j}) \dd\mu^{\oplus k}(x) + 0\ =\ t(F,W_\mu)
\end{split}
\end{equation*}
and the claim follows.

\ProofOf{it:KLsc} We have to show that the function $K_W:\tW\to [0,\infty]$ is lsc.
Take an arbitrary sequence $(U_n)_{n\in\I N}$ of graphons convergent to a graphon $U$ in the $\delta_\Box$-distance.
We need to show that $K_W(U)\le \liminf_{n\to\infty} K_W(U_n)$.
Assume that $\liminf_{n\to\infty} K_W(U_n)<\infty$ as otherwise there is nothing to do. By passing to a subsequence we can in fact assume that $\sup_{n\in\I N}K_W(U_n)<\infty$.
After passing to a further subsequence, we can pick a sequence $(\mu_n)_{n\in \I N}$ in $\mathcal{P}([0,1])$ such that $\delta_{\Box}(U_n,W_{\mu_n})=0$, $\mu_n\ll\lambda$ and
$$\liminf_{n\to\infty} \int_{[0,1]} 
\L{\frac{\dd\mu_n}{\dd\lambda}} \dd\lambda=\liminf_{n\to\infty} K_W(U_n).$$
By using Theorems~\ref{App:DP},~\ref{APP:VP} and~\ref{APP:ES} similarly as above and passing to a subsequence, we find $\mu\in \mathcal{P}([0,1])$ such that $\mu\ll\lambda$ and $\frac{\dd\mu_n^{\oplus k}}{\dd\lambda^{\oplus k}}\xrightarrow{weak} \frac{\dd\mu^{\oplus k}}{\dd\lambda^{\oplus k}}$ for every $k\in \I N$.
Similarly as in the proof of Part~\ref{it:KZero}, one gets
$$t(F,U_n)=t(F,W_{\mu_n})\to t(F,W_\mu).$$
This implies $\delta_\Box(U,W_\mu)=0$ and $K_W(U)\le \int_{[0,1]} \L{\frac{\dd\mu}{\dd\lambda}}\dd\lambda$.
Another use of Claim~\ref{cl:Jensen} implies $K_W(U)\le \liminf_{n\to\infty} K_{W}(U_n)$ because $\frac{\dd\mu_n}{\dd\lambda}\xrightarrow{weak} \frac{\dd\mu}{\dd\lambda}$.
This finishes the proof.

\ProofOf{it:KNonConst} Suppose that $W$ is not a constant graphon. First, we find $U,V\in \RN(W)$ such that $\delta_\Box(U,V)>0$.

Since $W$ is not constant there are $a<b\in [0,1]$ such that $A:=W^{-1}([0,a ])$ and $B:=W^{-1}([b,1])$ both have non-zero measure. Let $\epsilon:=\frac{(b-a)}{9}>0$. 
Using Lebesgue's density theorem, we can find pairwise disjoint intervals $X_A,Y_A,X_B,Y_B$ of the same positive measure such that
$$\lambda^{\oplus 2}\left(\,(X_{A}\times Y_{A})\cap A\, \right)\ge (1-\epsilon)\, \lambda^{\oplus 2}(X_{A}\times Y_{A})$$
$$\lambda^{\oplus 2}\left(\,(X_{B}\times Y_{B})\cap B \right)\ge (1-\epsilon)\, \lambda^{\oplus 2}(X_{B}\times Y_{B}).$$
We use these intervals to define distinct $U,V\in \RN(W)$.

\begin{definition}\label{def:Lebesgue}
For disjoint intervals $D_0,D_1\subseteq [0,1]$, 
we define the graphon $W_{(D_0,D_1)}$ as follows. For $(x,y)\in [0,1]^2$, let
$i,j\in \{0,1\}$ be the indices satisfying $x\in I^2_i$ and $y\in I^2_j$, and define 
$$
W_{(D_0,D_1)}(x,y):=W(\psi^{-1}_{D_i}(x)+i/2,\psi^{-1}_{D_j}(y)+j/2),
$$
where, for an interval $D$ with end-points $c<d$, $\psi_{D}(z):=\frac{z-c}{2(d-c)}$ is the linear maps that bijectively maps $D$ to $[0,1/2]$. 
\end{definition}

Informally speaking, $W_{(D_0,D_1)}$ is obtained from $W$ by putting half of the whole mass uniformly on $D_0$ and half on~$D_1$. 

Set $U:=W_{(X_{A},Y_{A})}$ and $V:=W_{(X_{B},Y_{B})}$. Clearly, $\T U,\T V\in\RN(W)$. Thus $U$ is at most $a$ (resp.\ $V$ is at least $b$) on $(0,1/2)\times (1/2,1)$ apart a subset of measure at most $\epsilon$.
Let us show that $\delta_\Box(U,V)>0$.
Let $\varphi:[0,1]\to[0,1]$ be an arbitrary measure preserving bijection. 
We can assume that at least half measure of the pre-image $\varphi^{-1}(\,(0,\frac12)\,)$ lies in the first half of $[0,1]$ by the symmetry between $X_{A}$ and~$Y_{A}$. Thus $C:=\varphi^{-1}(\,(0,\frac12)\,)\cap (0,\frac12)$ has measure at least $1/4$. By the measure preservation of $\varphi$, we have that the measure of $D:=\varphi^{-1}(\,(\frac12,1)\,)\cap (\frac12,1)$ is
 $$
 \textstyle
 \lambda(D)=\frac12-\lambda\left(\varphi^{-1}(\,(\frac12,1)\,)\cap (0,\frac12)\right) = \frac12-\lambda\left((0,\frac12)\setminus C\right)= \lambda(C).
 $$
Then we have
\begin{equation*}
\begin{split}
\int_{C\times D} V^{\varphi}(x,y) \dd\lambda^{\oplus 2}\ &=\ \int_{\varphi(C)\times \varphi(D)} V(x,y) \dd\lambda^{\oplus 2}\ \ge\  b\lambda^{\oplus 2}(C\times D)-\frac{\epsilon}{4} \\
\int_{C\times D} U(x,y) \dd\lambda^{\oplus 2}\ &\le  \ a\lambda^{\oplus 2}(C\times D)+\frac{\epsilon}{4}.
\end{split}
\end{equation*}
 and thus
$$d_{\Box}(U,V^\varphi)\ge (b-a)\lambda(C_i\times D_i)-\frac{\epsilon}{2}\ge \frac{(b-a)-8\epsilon}{16}>\frac{(b-a)}{144}.$$
Since $\varphi$ was an arbitrary measure preserving function, we have $d_{\Box}(U,V)>0$, as desired.

Finally, take any $\T V\in \RN(W)$ not weakly isomorphic to~$W$, which exists by above. By Part~\ref{it:KDense}, there is $\T U\in\RN(W)$ such that $K_W(U)<\infty$ and $\delta_\Box(U,V)<\delta_\Box(W,V)$. Then $\delta_\Box(U,W)>0$ and, by Parts~\ref{it:KNonNeg} and \ref{it:KZero}, we have $K_W(U)>0$. Thus $K_W(U)\in (0,\infty)$, finishing the proof.\end{proof}

\section{LDPs for $(\tH_{n,W})_{n\in \I N}$}\label{sec:H_{n,W}}

In this section, we completely resolve the LDP problem for the measures $\tH_{n,W}$, which come from sampling uniform $x\in [0,1]^n$ and taking the graphon of the edge-weighted graph $H$ on $[n]$ with $w_H(i,j):=W(x_i,x_j)$.

\begin{theorem}\label{thm:GeneralSpeedsH}
Take any graphon $W\in\C W$ and  any function $s:\I N\to(0,\infty)$ with $s(n)\to\infty$ as $n\to\infty$. We consider the existence of an LDP for the sequence $(\tH_{n,W})_{n\in\I N}$ of probability measures on the metric space $(\tW,\delta_\Box)$ with speed~$s$.
\begin{enumerate}[(a)]
	\item\label{it:GSH1} If $s(n)= o(n)$ then there is an LDP with rate function~$\chi_W$.
	\item\label{it:GSH2} If $s(n)=(c+o(1))n$  for some $c\in (0,\infty)$ then there is an LDP with rate function~$\frac1c\,K_W$. 
	\item\label{it:GSH3} If $s(n)=\omega(n)$, then  there is an LDP with rate function~$\chi_{\RN(W)}$.
	\item \label{it:GSHConst} If $W$ is a constant graphon, then LDP holds (regardless of the speed $s$) with rate function $\chi_W=K_W=\chi_{\RN(W)}$.
    \item\label{it:GSH4} If $W$ is not a constant graphon and $s(n)$ does not satisfy any of the assumptions \ref{it:GSH1}--\ref{it:GSH3}, then there is no LDP (that is, no function $I:\tW\to [0,\infty]$ satisfies Definition~\ref{df:LDP}).
 \end{enumerate}
\end{theorem}


Before we prove Theorem~\ref{thm:GeneralSpeedsH}, let us show that the measures $(\tH_{n,W})_{n\in \I N}$ are supported around $\RN(W)$ in the metric $\delta_\Box$ (which implies that any rate function for any speed must be $\infty$ outside of the closed set~$\RN(W)$).

\begin{proposition}\label{pr:concentration 1/n}
Let $W\in \C W$ and $n\in \I N$.
Then
$$\tH_{n,W}\left(\left\{\widetilde{U}\in \tW:\delta_\Box(\widetilde{U},\RN(W))\le \frac{1}{n}\right\}\right)=1.$$
\end{proposition}
\begin{proof} 
The proof consists of two steps.
In the first step we show that the claim holds whenever $W$ attains only finitely many values.
In the second step we use the fact that such graphons are dense in $\W$ in the $L^\infty$-norm to extend the claim to any $W$. It will be convenient in the proof to use also the measures $(\mathbb{H}_{n,W})_{n\in \I N}$ on the space $\mathcal{W}$. Recall that this is the version of $(\tH_{n,W})_{n\in \I N}$, where we do not take the equivalence class $\widetilde{f^H}$ of the sampled weighted graph $H$ but simply output~$f^H$.

\medskip\noindent{\bf Step (I).}
Suppose first that $W$ attains only finitely many values, i.e., the set $W\!\left([0,1]^2\right)$ is finite.
Then the measures $\mathbb{H}_{n,W}$ are atomic for every $n\in \I N$.
Let $U\in \W$ be such that $\mathbb{H}_{n,W}(\{U\})>0$.
Note that by the definition of $\mathbb{H}_{n,W}$ there is $H\in \mathcal{H}_n$ such that $U=f^H$ and if we put
$$
X_H:=\left\{x\in [0,1]^{{n}}:\mbox{for all }i<j\mbox{ in } [n] \ \ W(x\ind{i},x\ind{j})=w_H(i,j)\right\},
$$
then we have $\lambda^{\oplus n}\left(X_H\right)>0$.
By $d_1\ge d_\Box$, it is enough to find a sequence $(V_k)_{k\in \I N}$ in $\W$ such that 
$\lim_{k\to \infty}d_1\left(V_k,U\right)\le \frac{1}{n}$ and $\widetilde{V}_k\in \RN(W)$ for each~$k$.

Recall that $B_\infty(x,r)$ denotes the (open) radius-$r$ ball around $x\in [0,1]^{{n}}$ in the $L^\infty$-metric.
By the Lebesgue Density Theorem (for the $L^\infty$-metric on $[0,1]^{{n}}$), the set of density points of $X_H$ has measure $\lambda^{\oplus n}(X_H)>0$ so it contains a point ${{x}}\in (0,1)^{{n}}$ such that ${{x}}\ind{i}\not= {{x}}\ind{j}$ for every $i<j$  in $[n]$. Recall that ${x}$ being a density point of $X_H$ means that 
$$\lim_{\e\to 0}\frac{\lambda^{\oplus n}(X_H\cap B_\infty\left({{x}},\e\right))}{\lambda^{\oplus n}(B_\infty\left({{x}},\e\right))} =1.$$

Let $k\in \I N$ be large enough. Then the intervals $J_{i,k}:=\left[{{x}}\ind{i}-\frac{1}{k},{{x}}\ind{i}+\frac{1}{k}\right]$, for $i\in [n]$, are pairwise disjoint.
Similarly to Definition~\ref{def:Lebesgue}, let
$V_k:=W_{(J_{1,k},\dots,J_{n,k})}$
 be obtained from $W$ by putting mass $1/n$ uniformly into each of the intervals $J_{1,k},\dots,J_{n,k}$.
By definition, $V_k\in \RN(W)$.
Let $i<j$ in $[n]$ and define
$$A_{i,j,k}:=\left\{(a,b)\in I^n_{i}\times I^n_j:V_k(a,b)\not= w_H(i,j)\right\}.$$
Since the intervals $I_1^n,\ldots,I_n^n$ in $V_k$ correspond to the invervals $J_{1,k},\ldots,J_{n,k}$ in $W$, we have by Fubini-Tonelli's Theorem that
\begin{equation*}
\begin{split}
\lambda^{\oplus 2}\left(A_{i,j,k}\right)\,=\, & \textstyle\frac{(2k)^2}{n^2}\, \lambda^{\oplus 2}\left(\left\{(a,b)\in B_\infty\!\left(({{x}}\ind{i},{{x}}\ind{j}),\frac{1}{k}\right): W(a,b)\not =w_H(i,j) \right\}\right) \\
\,=\,& \textstyle\frac{(2k)^n}{n^2}\, \lambda^{\oplus n}\left(\left\{{{z}}\in B_\infty\left({{x}},\frac{1}{k}\right): W({{z}}\ind{i},{{z}}\ind{j})\not =w_H(i,j) \right\}\right)\\
\,\le\, & \textstyle\frac{(2k)^n}{n^2}\, \lambda^{\oplus n}(B_\infty\!\left({{x}},\frac{1}{k}\right)\setminus X_H)\,=\,
\frac{\lambda^{\oplus n}(B_\infty\left({{x}},\frac{1}{k}\right)\setminus X_H)}{n^2\,\lambda^{\oplus n}(B_\infty\!\left({{x}},\frac{1}{k}\right))},
\end{split}
\end{equation*}
 which tends to $0$ as $k\to\infty$ since ${x}$ is a density point for $X_H$. 

Noticing that, for $1\le i<j\le k$, $A_{i,j,k}$ is exactly the subset of $I^n_i\times I^n_j$ where $U$ and $V_k$ differ we have
\begin{equation*}
\begin{split}
d_1\left(V_k,U\right)\ =\ &\int_{[0,1]^2}\left|V_k-U\right| \dd\lambda^{\oplus 2} \\
\ \le\ & \sum_{i\in [n]} \lambda^{\oplus 2}\left(I^n_i\times I^n_i\right)+2\sum_{1\le i<j\le n} \lambda^{\oplus 2}\left(A_{i,j,k}\right)\ \xrightarrow{k\to \infty}\  \frac{1}{n}.
\end{split}
\end{equation*}
This finishes the first step of the proof.

\medskip\noindent{\bf Step~(II).} Let $W$ be an arbitrary graphon and let $n\in\I N$.
For $k\in \I N$, let $W_k\in \W$ be such that $d_\infty(W,W_k)\le \frac1k$ and $W_k$ attains only finitely many values. Note that
for every Borel probability measure $\mu\in \mathcal{P}([0,1])$ that is absolutely continuous with respect to $\lambda$, we have
\[
\delta_\Box\left(W_\mu,(W_k)_\mu\right)\le \frac1k.
\]
 Indeed, we can assume that we use the same measure-preserving map $\phi:([0,1],\lambda)\to ([0,1],\mu)$ when defining $W_\mu$ and $(W_k)_\mu$ as in~\eqref{eq:WMu}. Then $\delta_\Box\left(W_\mu,(W_k)_\mu\right)\le \sup_{A,B}|\int_{A\times B}(W-W_k)\mu^{\oplus2}|$ while the function $|W-W_k|$ is at most $\frac1k$ except on a set which is null in $\lambda^{\oplus 2}$ and thus also in $\mu^{\oplus 2}\ll \lambda^{\oplus 2}$.

It follows that 
 \begin{equation}\label{eq:blabla2}
 \begin{split}
 \left\{\widetilde{U}\in \tW:\delta_\Box\left(\widetilde{U},\RN(W_k)\right)\le \frac{1}{n}\right\}\subseteq
  \left\{\widetilde{V}\in \tW:\delta_\Box\left(\widetilde{V},\RN(W)\right)\le \frac{1}{n}+\frac1k\right\}.
 \end{split}
 \end{equation}

In the natural coupling between $\mathbb{H}_{n,W}$ and $\mathbb{H}_{n,W_k}$ where we use the same ${{x}}\in [0,1]^n$, the edge weights on each pair differ at most by $\frac1k$ and thus the $d_\infty$-distance between the sampled graphons is at most~$\frac1k$ with probability~1.  Since $\delta_\Box\le d_\infty$ and the left-hand side of \eqref{eq:blabla2} has the $\tH_{n,W_k}$-measure $1$ by Step~(I), we conclude that the $\tH_{n,W}$-measure of the $(\frac1n+\frac2k)$-neighbourhood of $\RN(W)$ in $\delta_\Box$ is also 1. Now the proposition follows
by the $\sigma$-additivity of~$\tH_{n,W}$ since $k\in\I N$ was arbitrary.
\end{proof}

Now we can now resolve the LDP problem for $\tH_{n,W}$ for every speed.

\begin{proof}[Proof of Theorem~\ref{thm:GeneralSpeedsH}] It is more convenient to prove its parts in a different order than stated.

\ProofOf{it:GSH2}
Recall that we are given $W\in \C W$ and have to show that $(\tH_{n,W})_{n\in \I N}$ satisfies the LDP on $\left(\tW,\delta_\Box\right)$ with speed $s(n)=(c+o(1))n$ and rate function $\frac1c\,K_W$. In brief, the result will be proved by using a result of Eichelsbacher~\cite{Eichelsbacher97} (Theorem~\ref{th:E97} here) and the Contraction Principle.
It is enough to do the case when $c=1$ and $s(n)=n$.  

Recall that $\mathcal{P}([0,1])$ is the space of all Borel probability measures on $[0,1]$ and,
for $k\in \I N$, $\tau^k$ is the coarsest topology on $\mathcal{P}([0,1])$ such that,
for every bounded Borel function $a:[0,1]^k\to \mathbb{R}$,
 the map
\beq{eq:TauK}
 \mu\mapsto \int_{[0,1]^k} a \dd\mu^{\oplus k},\quad \mu\in \mathcal{P}([0,1]),
 \eeq
 is continuous.
Let $\tau^\infty$ be the  topology generated by $\bigcup_{k\in \mathbb{N}} \tau^k$, that is, 
the coarsest topology on $\mathcal{P}([0,1])$ such that the map in~\eqref{eq:TauK} is continuous  for all $k\in\I N$ and all Borel functions~$a$.

The topology $\tau^k$ will be useful to us since it makes various maps involving $k$-fold integrals
continuous. However, one has to be careful as some natural properties fail. For example, \cite[Exercise~7.3.18]{DemboZeitouni10ldta} shows that the map $\mathcal{P}([0,1])\to \mathcal{P}([0,1]^k)$ that sends a measure $\mu$ to its $k$-th power $\mu^{\oplus k}$ for $k\ge 2$ is not continuous with respect the $\tau^1$-topologies on these two spaces.

Let $q_n:[0,1]^{{n}}\to \mathcal{P}([0,1])$ be the map that assigns to $x\in [0,1]^{{n}}$ the measure
$$
 q_n(x):=\frac{1}{n}\sum_{i\in [n]} \delta_{x\ind{i}}.
$$

\begin{claim}\label{clapp:basic}
The map $q_n:[0,1]^{{n}}\to (\mathcal{P}([0,1]),\tau^\infty)$ is Borel.
\end{claim}

\begin{proof}[Proof of Claim] It is enough to show that,  for every $k\in\I N$, $t\in\I R$, and every Borel function $a:[0,1]^k\to \I R$, the pre-image under $q_n$ of the set $\{\mu\in \mathcal{P}([0,1])  \mid \int a\mu^{\oplus k}<t\}$ is Borel since these sets generate the topology~$\tau^\infty$. The pre-image is exactly the set $\{x\in [0,1]^{{n}}\mid \phi(x)< t\}$, where $\phi(x)$ is the integral of $a$ under the $k$-th power of $q_n(x)=\frac{1}{n}\sum_{i\in [n]} \delta_{x\ind{i}}$. Note that $\phi$ is the sum of $n^{-k}a(x\ind{i_1},\dots,x\ind{i_k})$ over all $n^k$ possible choices of $i_1,\dots,i_k\in [n]$. This is a Borel function of~$x$. Thus its level set $\{\phi<t\}\subseteq [0,1]^{{n}}$ is also Borel, as desired. 
\end{proof}

Recall that $\lambda$ is the Lebesgues measure on $[0,1]$. Define 
$
 \Lambda_n:=(q_n)_*\,\lambda^{\oplus n}
 $
 to be the push forward of $\lambda^{\oplus n}$ via the map~$q_n$. By Claim~\ref{clapp:basic},  $\Lambda_n$ is a well-defined Borel probability measure on the topological space $(\mathcal{P}([0,1]),\tau^\infty)$.

Let $\mu\in \mathcal{P}([0,1])$. Define the \emph{Kullback–Leibler divergence} (or \emph{relative entropy}) of $\mu$ with respect to the Lebesgue measure $\lambda$ by
\beq{eq:H}
H(\mu|\lambda):=\left\{
\begin{array}{ll}
\int_{[0,1]} \L{\frac{\dd\mu}{\dd\lambda}}  \dd\lambda, & \mbox{if $\mu\ll \lambda$,}\\
\infty,& \mbox{otherwise.}
\end{array}
\right.
\eeq

The following result for $m=2$ was proved by Eichelsbacher~\cite[Theorem 1]{Eichelsbacher97} who remarked
(\cite[Page 911]{Eichelsbacher97})
that his result can be generalised to arbitrary~$m$. (A more general result can be found in Eichelsbacher and Schmock \cite[Theorems 1.7(c) and 1.10]{EichelsbacherSchmock02}: namely Theorem~\ref{th:E97} follows from the special case when the family $\Phi$ in~\cite{EichelsbacherSchmock02} is taken to be the set of all bounded measurable functions $[0,1]^m\to \I R$.)

\begin{theorem}[Eichelsbacher~\cite{Eichelsbacher97}] 
\label{th:E97}
Let $m\in\I N$ and let $\C Q^{\oplus  m}_n$ be the law of the $m$-th power
$(q_n(x_n))^{\oplus m}\in \mathcal{P}([0,1]^m)$ where $x_n$ is a uniform element of $[0,1]^{{n}}$. For $\nu\in \mathcal{P}([0,1]^m)$, if $\nu$ is the $m$-th power $\eta^{\oplus m}$ of some measure $\eta\in\mathcal{P}([0,1])$, then define 
$$
I(\nu):=H(\eta|\lambda);
$$
 otherwise we define $I(\nu):=\infty$. 

Then the sequence of measures $(\C Q^{\oplus  m}_n)_{n\in\I N}$ on the topological space $(\mathcal{P}([0,1]^m),\tau^1)$ satisfies an LDP with speed $n$ and rate function~$I$. Moreover, the rate function $I$ is \emph{good}, that is, 
for every $t\in [0,\infty)$ the level set $\{I\le t\}$ is $\tau^1$-compact.
\qqed
\end{theorem}

Let us show that, for every $m\in\I N$, the set $\mathcal{P}_{\mathrm{power}}([0,1]^m)$ of measures $\mu\in \mathcal{P}([0,1]^m)$ such that $\mu=\nu^{\oplus m}$ for some $\nu\in \mathcal{P}([0,1])$ is closed in the $\tau^1$-topology. 
By Dynkin's theorem (e.g.\ \cite[Corollary 1.6.3]{Cohn13mt}), this set can be equivalently described as the set of those $\mu\in \mathcal{P}([0,1]^m)$ such that for all measurable sets $A_1,\dots,A_m\subseteq [0,1]$ it holds that
 \beq{eq:ProdMeasure}
 \mu(A_1\times\dots\times A_m)=\prod_{i=1}^m \mu\left([0,1]^{i-1}\times A_i\times [0,1]^{m-i}\right).
 \eeq
 So take any measure $\mu$ which is not in this set and fix some $A_1,\dots,A_m\subseteq [0,1]$ violating~\eqref{eq:ProdMeasure}. Then there is $\e>0$ such that~\eqref{eq:ProdMeasure} is still violated if we modify each value of $\mu$ by less than~$\e$. Thus if we take the indicator functions of the $m+1$ sets appearing in~\eqref{eq:ProdMeasure} and require that the $\nu$-integral of each differs from the corresponding value of $\mu^{\oplus m}$ by less than $\e$, then we obtain a $\tau^1$-open set of measures $\nu$, which contains $\mu$ and is disjoint from $\mathcal{P}_{\mathrm{power}}([0,1]^m)$. Hence, the latter set is indeed a $\tau^1$-closed.

Since for each integer $n\in\I N$ the support of the measure $\C Q^{\oplus  m}_n$ lies inside the closed set $\mathcal{P}_{\mathrm{power}}([0,1]^m)$, the LDP of Theorem~\ref{th:E97} can be restricted to this set, with each level set of the rate function $I$ still being compact. On the the other hand, the map $\nu\mapsto \nu^{\oplus m}$ gives a homeomorphism between the topological spaces $(\mathcal{P}([0,1]),\tau^m)$ and $(\mathcal{P}_{\mathrm{power}}([0,1]^m),\tau^1)$. Also, under this map, the push-forward of $\Lambda_n$ is precisely~$\C Q^{\oplus  m}_n$ while $I$ corresponds to~$H(\_|\lambda)$. Thus,  for every $m\in \I N$, the sequence $(\Lambda_n)_{n\in \mathbb{N}}$ on $\left(\mathcal{P}([0,1]),\tau^m\right)$ satisfies an LDP with speed $n$ and rate function $H(\_|\lambda)$. Since $\tau^\infty$ is the coarsest topology generated by all $\tau^m$, we obtain the following claim by the Dawson–G\"artner theorem
(\cite{DawsonGartner87}, or see e.g.\ \cite[Theorem~4.6.1]{DemboZeitouni10ldta}).

\begin{claim}\label{app:eichelsbacher Schmock}
The sequence $(\Lambda_n)_{n\in \mathbb{N}}$ satisfies an LDP on $\left(\mathcal{P}([0,1]),\tau^\infty\right)$ with speed $n$ and rate function $H(\_|\lambda)$.\qed
\end{claim}

Fix a representation of $W:[0,1]^2\to [0,1]$ as a symmetric Borel function that satisfies $W(x,x)=0$ for all $x\in [0,1]$. We remark that the requirement to have a value 0 on the diagonal corresponds to the fact  that weighted graphs in our definition do not have loops. Next, we extend the definition of $W_\mu$ to an arbitrary $\mu\in \mathcal{P}([0,1])$ by taking some measure-preserving map $\phi:([0,1],\lambda)\to ([0,1],\mu)$ and defining $W_\mu(x,y):=W(\phi(x),\phi(y))$ for $(x,y)\in [0,1]^2$.
Note that since $\mu$ can have atoms we need to have $W$ fixed as an everywhere defined Borel function (and not as a measurable function defined a.e.). 
Write $\widetilde{W}_\mu$ for the equivalence class of $W_\mu$ in $\tW$.

Define a map $\Delta_W:\mathcal{P}([0,1])\to \tW$ by $\mu\mapsto \widetilde{W}_\mu$.
Even though the definition of $\Delta_W$ depends on the fixed representative of $W$ as a Borel function, it can be easily verified that, irrespective of this representation, the push-forwards of the Lebesgue measure $\lambda^{\oplus n}$ to the space~$\tW$ along the maps in the diagram below are the same. Also, the following properties do not depend on the choice of a representative.

\begin{claim}
The map $\Delta_W$ is $(\tau^\infty,\delta_\Box)$-continuous.
\end{claim}
\begin{proof}
 By Theorem~\ref{th:ConvOnW}, it is enough to show that, for any fixed graph $F=([k],E(F))$, the real-valued function $t(F,\Delta_W(\_)):\mathcal{P}([0,1])\to\I R$ is $\tau^\infty$-continuous. This function
sends a measure $\mu\in \mathcal{P}([0,1])$ to $t(F,W_\mu)=\int_{[0,1]^k} a(x) \dd\mu^{\oplus k}(x)$, where 
$a(x):=\prod_{\{i,j\}\in E(F)} W(x\ind{i},x\ind{j})$ for $x\in [0,1]^k$.  It is, in fact, $\tau^k$-continuous by definition, since $a$ is a Borel function.
The claim is proved.
\end{proof}

For $n\in \I N$, consider now the map $q_{n,W}:[0,1]^n\to\W$  that is defined by 
$$q_{n,W}(x):=f^{x,W},\quad \mbox{$x\in [0,1]^n$},$$
where $f^{x,W}$ is the step graphon that is obtained from the weighted graph $H$ on $[n]$  that corresponds to the $x$-sample from $W$, i.e.,  $w_H(i,j)=W(x\ind{i},x\ind{j})$.
Note that since $W$ is a fixed Borel function the map $q_{n,W}$ is well-defined.

It follows easily from the definitions that the diagram
\[
  \begin{tikzcd}
      \left[0,1\right]^{{n}} \arrow{r}{q_{n,W}} \arrow{d}{q_n} & \mathcal{W} \arrow[two heads]{d}  \\
     \mathcal{P}([0,1]) \arrow{r}{\Delta_W} & \widetilde{\mathcal{W}}
  \end{tikzcd}
\]
is commutative.
In particular, for every $n\in \I N$, the measure $\tH_{n,W}$ is the push-forward of $\Lambda_n$ via~$\Delta_W$. (Recall that $\Lambda_n$ is the push-forward of $\lambda^{\oplus n}$ via $q_n$.)

By Claim~\ref{app:eichelsbacher Schmock}, the sequence $(\Lambda_n)_{n\in \I N}$ satisfies LDP on $(\mathcal{P}([0,1]),\tau^\infty)$ with speed $n$ and good rate function $H(\_|\lambda)$. By the Contraction Principle 
(see e.g.\ \cite{DemboZeitouni10ldta}*{Theorem~4.2.1}),
the sequence $(\tH_{n,W})_{n\in\I N}$ of measures on $(\tW,\delta_\Box)$ satisfies an LPD with speed $n$ and rate $K_W$, finishing the proof of Part~\ref{it:GSH2}


\medskip Let us turn to the remaining parts of Theorem~\ref{thm:GeneralSpeedsH}. 

\ProofOf{it:GSH1} Recall that the claimed rate function is $\chi_W$. Let $\mathcal{O}\subseteq \tW$ be an open set.
Suppose that the equivalence class $\widetilde{W}$ of $W$ belongs to $\mathcal{O}$ for otherwise $\inf_{\mathcal{O}}\chi_W=\infty$ and the LDP lower bound \eqref{eq:lowerGen} vacuously holds. Since the cut-distance between the graphon $W$ and its $n$-vertex sample $G\sim \tH_{n,W}$ tends to 0 with probability $1-o(1)$ as $n\to\infty$ (by, for example, \cite[Lemma~10.16]{Lovasz:lngl}), the probability that $\f{H}\in \mathcal {O}$ tends to 1 and the LDP lower bound \eqref{eq:lowerGen} trivially holds. (Recall that  we assume that $s(n)\to\infty$ as $n\to\infty$.)

Next, let $F\subseteq \tW$ be a closed set.
Suppose that $\widetilde{W}\not\in F$ for otherwise the LDP upper bound is vacuously true.  By Part~\ref{it:GSH2} applied to the closed set $F$ with speed $n$,  we have that 
\begin{equation*}
  \limsup_{n\to \infty}\frac{1}{n} \log\left(\tH_{n,W}(F)\right)\le  -\inf_{\widetilde{U}\in F} K_W(U). 
\end{equation*}
 Observe that the right-hand side is strictly less than $0$. Indeed, 
 since the function $K_W$ is lower semi-continuous by  Proposition~\ref{pr:K}\ref{it:KLsc}, it attains its infimum over the compact set $F$ at some $\widetilde{V}\in F$. 
By Proposition~\ref{pr:K}\ref{it:KZero}, we have $K_W(V)>0$. Thus $\inf_{\widetilde{U}\in F} K_W(U)= K_W(V)>0$.

Consequently, by $s(n)=o(n)$, we have 
\begin{equation*}
\begin{split}
 \limsup_{n\to \infty}\frac{1}{s(n)} \log\left(\tH_{n,W}(F)\right) =-\infty =
-\inf_{\widetilde{U}\in F} \chi_W(U)
\end{split}
\end{equation*}
and that shows that the LDP upper bound \eqref{eq:upperGen} holds as well.

\ProofOf{it:GSH3} Here $s(n)/n\to\infty$ and the stated rate is $\chi_{\RN(W)}$.  Let $\mathcal{O}\subseteq \tW$ be an open set. Suppose that $\mathcal{O}\cap \RN(W)\not =\emptyset$ as otherwise
the LDP lower bound \eqref{eq:lowerGen} vacuously holds. By the already proved Part~\ref{it:GSH2} (for speed $n$) we have
\begin{equation*}
\begin{split}
\liminf_{n\to \infty}\frac{1}{n} \log\left(\tH_{n,W}(\mathcal{O})\right)\ge -\inf_{\widetilde{U}\in \mathcal{O}} K_W(U).
\end{split}
\end{equation*}
Recall that $\RN(W)$ is the closure of the set of graphons $W_\mu$ over all Borel probability measures  $\mu\ll \lambda$; thus we can find such $\mu$ so that $\widetilde{W_\mu}$ belongs to the open set~$\mathcal{O}$. Pick $r>0$ such that the $\delta_\Box$-ball of  radius $r$ around $\widetilde{W_\mu}$ lies entirely inside~$\mathcal{O}$. By the integrability of $\frac{\dd\mu}{\dd\lambda}$ with respect to $\lambda$, there is finite $R$ such that the $\lambda$-measure of the level set $X:=\{\frac{\dd\mu}{\dd\lambda}\le R\}$ is at least $1-r/2$. Let the probability measure $\nu$ be zero on $[0,1]\setminus X$ and be the restriction of $\mu$ to $X$ scaled up by factor $1/\mu(X)$. 
Since the measures $\mu^{\oplus2}$ and $(\frac1{\mu(X)}\,\nu)^{\oplus2}$ coincide on a set $X^2$ of measure 
at least $(1-r/2)^2>1-r$,
we have $\delta_\Box(W_\mu,W_\nu)<r$. Thus $W_\nu$ still belongs to the open set $\mathcal{O}$. Thus $\inf_{\mathcal{O}} K_W\le K_W(W_\nu)\le \L{R}<\infty$. Consequently, since $s(n)/n\to\infty$, we have
\begin{equation*}
\begin{split}
 \liminf_{n\to \infty}\frac{1}{s(n)} \log\left(\tH_{n,W}(\mathcal{O})\right) = 0 = -\inf_{\widetilde{U}\in \mathcal{O}} \chi_{\RN(W)}(U).
\end{split}
\end{equation*}
This establishes the LDP lower bound \eqref{eq:lowerGen}.

Suppose now that $F\subseteq \tW$ is a closed set. Assume that $F\cap \RN(W)=\emptyset$ as otherwise the LDP upper bound \eqref{eq:upperGen} vacuously holds. Thus, by the compactness of $\tW$, we can find $n_0\in \I N$ such that
$$F\cap \left\{\widetilde{V}\in \tW:\delta_\Box(\widetilde{V},\RN(W))\le \frac{1}{n_0}\right\}=\emptyset.$$
Consequently, $\tH_{n,W}(F)=0$ holds for every $n\ge n_0$ by Proposition~\ref{pr:concentration 1/n}, and the upper LDP bound \eqref{eq:upperGen} follows.

\ProofOf{it:GSHConst} Here $W$ is a constant graphon, say of value~$p$. Trivially, $\RN(W)=\{W\}$, $K_W=\chi_W=\chi_{\RN(W)}$, and $\tH_{n,W}$ is the Dirac point mass on $\T W_n$, where the graphon $W_n$ is $p$ everywhere, except it is $0$ on $\cup_{i=1}^n (I_n^i\times I_n^i)$. (Recall that weighted graphs do not have loops.) Since the graphons $W_n$ converge to $W$, the LPD claimed in Part~\ref{it:GSHConst} holds.

\ProofOf{it:GSH4} It remains to show that there is no LDP if $W$ is not constant and $s(n)$ does not satisfy any of the assumptions in Parts \ref{it:GSH1}--\ref{it:GSH3} of Theorem~\ref{thm:GeneralSpeedsH}. Pick two subsequences of $n$ such that the speeds $s(n)$ satisfy two different assumptions from \ref{it:GSH1}--\ref{it:GSH3}. These two subsequences satisfy the corresponding LDPs by what we have just proved. It is enough to show that the two corresponding  rate functions differ since the (lower semi-continuous) LDP rate function on a regular topological space is unique (see e.g.,~\RS{Theorem 2.13}).

If at least one of the occurring cases is from Part~\ref{it:GSH2} then, by Proposition~\ref{pr:K}\ref{it:KNonConst}, the rate function assumes a finite positive value and thus determines $\lim_{n\to\infty} s(n)/n$ (and is also different from the $\{0,\infty\}$-valued functions $\chi_W$ and $\chi_{\RN(W)}$), giving the required. Otherwise, we have exactly the two cases of Parts~\ref{it:GSH1} and \ref{it:GSH3} occurring, one rate function $\chi_W$ is $0$ on a single point $\widetilde{W}$ while the other rate function vanishes on at least two different points by Proposition~\ref{pr:K}\ref{it:KNonConst}.
\end{proof}

\section{Exponential equivalence}\label{sec:P vs H}

The following notion of \emph{exponential equivalence} is useful for transferring an LDP proved for one sequence of measures to another (see e.g.\ \cite{DemboZeitouni10ldta}*{Theorem~4.2.13})
We use the standard definition (as in~\cite[Definition~4.2.10]{DemboZeitouni10ldta}), 
which simplifies slightly when restricted to the case of separable metric spaces.

\begin{definition}[Exponential equivalence]\label{App:DefExpEq}
Let $(\mu_n)_{n\in \I N}$ and $(\nu_n)_{n\in \I N}$ be sequences of Borel probability measures on  a separable metric space $(X,d)$. Let $s:\I N\to(0,\infty)$ satisfy $s(n)\to\infty$ as $n\to\infty$.
We say that $(\mu_n)_{n\in \I N}$ and $(\nu_n)_{n\in \I N}$ are \emph{exponentially equivalent for speed $s$} if there is a sequence of Borel probability measure $(\Theta_n)_{n\in \I N}$ on $X\times X$ such that for every $n\in \I N$ the marginals of $\Theta_n$ are $\mu_n$ and $\nu_n$, and for every $\alpha>0$ it holds that
$$
\limsup_{n\to\infty}\frac{1}{s(n)}\log\left(\Theta_n\left(\{(x,y)\in X\times X:d(x,y)>\alpha\}\right)\right)=-\infty.
$$
\end{definition}

Informally speaking, the definition states that there is a coupling such that the probability of seeing a pair of outcomes at distance bounded from 0 decays faster than $\me^{-O(s)}$.

\hide{
This notion is useful because of the following result.

\begin{theorem}[See e.g.\ \cite{DemboZeitouni10ldta}*{Theorem~4.2.13}]
\label{App:ExpEq}
Let $(\mu_n)_{n\in \I N}$ and $(\nu_n)_{n\in \I N}$ be two  sequences of Borel probability measures on  a separable metric space $(X,d)$ that are exponentially equivalent for some speed $s(n)\to\infty$. 
Let $I:X\to [0,\infty]$ be a good rate function.
Then $(\mu_n)_{n\in \mathbb{N}}$ satisfies an  LDP on $X$ with speed $s$ and rate $I$ if and only if $(\nu_n)_{n\in \mathbb{N}}$ satisfies an LDP on $X$ with speed $s$ and rate~$I$.\qqed
\end{theorem}
}

Let us show that,  for any speed~$o(n^2)$, the $n$-vertex samples from a graphon, one with edge rounding and the other without edge rounding, are exponentially equivalent.

\begin{lemma}\label{lm:ExpEq}
Let $W\in\C W$. Take any function $s:\I N\to(0,\infty)$  such that $s(n)\to\infty$ and $s(n)/n^2 \to0$ as $n\to\infty$.
Then the sequences $(\tR_{n,W})_{n\in\I N}$ and $(\tH_{n,W})_{n\in\I N}$ are exponentially equivalent for speed~$s$.
\end{lemma}

\begin{proof}
Let $n\in \I N$ and $H\in \mathcal{H}_n$.
Recall that $w_H$ denotes the edge-weight function of~$H$. Let $\Lambda_{n,H}$ be the atomic probability measure on $\W$ which, for every graph $G$ on $[n]$, gives mass
\begin{equation}\label{eq:measure}
 \Lambda_{n,H}(\{\f{G}\}):=\prod_{\{i,j\}\in E(G)} w_H(i,j)\prod_{\{i',j'\}\in E(\O G)} \left(1-w_H(i',j')\right)
\end{equation}
 to $\f{G}$. This is exactly the distribution after edge-rounding the weighted graph $H$ (and taking the corresponding graphon in~$\W$).
It is easy to see that,  for every $n\in \I N$, the assignment $(H,G)\mapsto \Lambda_{n,H}\left(\left\{\f{G}\right\}\right)$ is Borel measurable.
This allows us to define a Borel probability measure $\Lambda_{n,W}$ on $\W\times \W$ by
$$
\Lambda_{n,W}(A):=\int \Lambda_{n,H}\left(\left\{\f{G}:G\in\mathcal{G}_n,\ \left(\f{H},\f{G}\right)\in A\right\}\right)  \dd\mathbb{H}_{n,W}(\f{H}),
$$
where $A\subseteq \W\times \W$ is a Borel set.
Informally speaking, the first component is the weighted $n$-vertex sample $H$ from $W$ while the second component is an edge rounding of~$H$ (or, more precisely, we take the graphons corresponding to these graphs).
Define the probability measure $\Theta_{n,W}$ as the push-forward of $\Lambda_{n,W}$ via $\W\times \W\twoheadrightarrow\tW\times \tW$, the product of the quotient maps.
One can easily verify that the marginals of $\Theta_{n,W}$ are $\tH_{n,W}$ and~$\tR_{n,W}$.

Let us show that the coupling given by $\Theta_{n,W}$ satisfies Definition~\ref{App:DefExpEq}.
Take any $\alpha>0$.
We set
$$\mathcal{O}_\alpha:=\{(\widetilde{U},\widetilde{V})\in \tW\times \tW: \delta_\Box(\widetilde{U},\widetilde{V})>\alpha\}$$
and 
$$\mathcal{Q}_\alpha:=\{(U,V)\in \W\times \W: \delta_\Box(U,V)>\alpha\}.$$
Clearly, we have $\Theta_{n,W}(\mathcal{O}_\alpha)=\Lambda_{n,W}(\mathcal{Q}_\alpha)$. Our aim is to show that this is $\me^{-\omega(s(n))}$ as $n\to\infty$.

We need the following result. Recall that the measure $\Lambda_{n,H}$ was defined in (\ref{eq:measure}).

\begin{claim}\label{cl:Azuma graphon}
For every $k\in \I N$ there is $C>0$ such that for every $F\in \mathcal{G}_k$, $\beta>0$, $n\in \I N$ and $H\in \mathcal{H}_n$ it holds that
\begin{equation}\label{eq:Azuma}
 \Lambda_{n,H}\left(\left\{\f{G}:G\in\mathcal{G}_n,\ \left|t(F,\f{G})-t(F,\f{H})\right|>\beta\right\}\right)\le \me^{-C\beta^2 n^2}.
\end{equation}
\end{claim}
\begin{proof}
Let $F\in \mathcal{G}_k$, $n\in \I N$ and $H\in \mathcal{H}_n$.
Let $\hom(F,G)$ denote the number of \emph{homomorphisms} from $F$ to $G$, that is, functions $V(F)\to V(G)$ that send edges to edges.
Note that  $\hom(F,G)=t(F,\f{G})\,n^k$.
With this notation we have that the left-hand side of (\ref{eq:Azuma}) is equal to
\begin{equation*}\label{eq:Azuma bla}
\Lambda_{n,H}\left(\left\{\f{G}:G\in\mathcal{G}_n,\ \left|\hom(F,G)-t(F,\f{H})n^k\right|>\beta n^k\right\}\right).
\end{equation*}

Enumerate the set $\binom{[n]}{2}$ as $\left\{e_1,\dots, e_{m}\right\}$, $m:={n\choose 2}$. 
Consider the corresponding edge-exposure martingale for $t(F,\_)$ in the edge rounding process for~$H$. Formally, for $s\in \{0,\dots,m\}$ we define a random variable $X_s:\mathcal{G}_n\to \mathbb{R}$ as
$$
X_s(G):=\sum_{f\in [n]^{[k]}} \prod_{\substack{\{i,j\}\in E(F) \\ a(f(i),f(j))>s}} w_H(f(i),f(j))\prod_{\substack{\{i,j\}\in E(F) \\ a(f(i),f(j))\le s}} {\mathbbm 1}_{\{f(i),f(j)\}\in E(G)},\quad G\in\mathcal{G}_n,
$$
where $a(i,j)$ for $ij \in {[n]\choose 2}$ is the unique index in $\left[m\right]$ with $e_{a(i,j)}=\{i,j\}$.
Consider the \mbox{($\sigma$-)algebra} filtration $\left(\mathcal{B}_s\right)_{s=0}^m$ on $\mathcal{G}_n$ that is generated by the functions $\left({\mathbbm 1}_{e_i\in E(G)}\right)_{i=1}^s$, i.e., revealing edges one-by-one according to the enumeration.
It is easy to see that the sequence $\left(X_s\right)_{s=0}^m$ forms a martingale with respect to the filtration $\left(\mathcal{B}_s\right)_{s=0}^m$.
Moreover $\mathbb{E}(X_0)=t(F,\f{H})n^k$ and $X_{m}(G)=\hom(F,G)$ for every $G\in \mathcal{G}_n$.
Let $s\in [m]$.
Then we have
\begin{equation*}
\begin{split}
\left|X_{s}-X_{s-1}\right|\le & \  \left|\left\{f:\in [n]^{[k]}:e_{s}\subseteq f(\,[k]\,)\right\}\right| \le k^2 n^{k-2}.
\end{split}
\end{equation*}
By Azuma's inequality (see e.g.\ \cite[Theorem 7.2.1]{AlonSpencer16pm}), we have that, for some $C=C(k)>0$,
\begin{eqnarray*}
&&\Lambda_{n,H}\left(\left\{\f{G}\mid G\in\C G_n,\ \left|\hom(F,G)-t(F,\f{H})n^k\right|>\beta n^k\right\}\right)\\ &= &  \Lambda_{n,H}\left(\left\{\f{G}\mid G\in\C G_n,\ \left|X_{m}-\mathbb{E}\left(X_0\right)\right|>\beta n^k\right\}\right) \\
&\le & \ 2\me^{-\frac{\beta ^2n^{2k}}{2mk^4n^{2(k-2)}}}\le \me^{-C\beta ^2n^2}.
\end{eqnarray*}
This finishes the proof of the claim.
\end{proof}

We are now ready to verify Definition~\ref{App:DefExpEq} 
for any given $\alpha>0$. The Inverse Counting Lemma of Borgs et al~\cite[Theorem 2.7(b)]{BCLSV08} (see also \Lo{Lemma~10.32}) 
implies that (depending on $\alpha>0$) there are $k\in \I N$ and $\beta>0$ such that
$$
\mathcal{Q}_\alpha\subseteq  \{(U,V)\in \W\times \W: \exists F\in\C G_k\ \left|t(F,U)-t(F,V)\right|>\beta \}.$$
We have \begin{equation*}
\begin{split}
\Lambda_{n,W}\left(\mathcal{Q}_\alpha\right)=&
\int \Lambda_{n,H}\left(\left\{\f{G}:G\in\mathcal{G}_n,\ \left(\f{H},\f{G}\right)\in \mathcal{Q}_\alpha\right\}\right)  \dd\mathbb{H}_{n,W}(\f{H}) \\
\le & \int \Lambda_{n,H}\left(\left\{\f{G}:G\in \mathcal{G}_n,\ \exists F\in\C G_k\  \left|t(F,\f{H})-t(F,\f{G})\right|>\beta\right\}\right)  \dd\mathbb{H}_{n,W}(\f{H}) \\
\le & \sum_{F\in\C G_k} \int \Lambda_{n,H}\left(\left\{\f{G}:G\in \mathcal{G}_n,\ \left|t(F,\f{H})-t(F,\f{G})\right|>\beta\}\right\}\right)  \dd\mathbb{H}_{n,W}(\f{H}) \\
\le & \ 2^{{k\choose 2}}\me^{-C\beta ^2n^2}
\end{split}
\end{equation*}
where $C>0$ depends only on $k\in \I N$ and the last inequality follows from Claim~\ref{cl:Azuma graphon}.
Putting all together, we obtain
$$\limsup_{n\to\infty}\frac{1}{s(n)} \log\left(\Theta_{n,W}\left(\mathcal{O}_\alpha\right)\right)\le \limsup_{n\to\infty} -\frac{1}{s(n)} C\beta ^2n^2=-\infty$$
by the assumption that $s(n)=o(n^2)$. 

This verifies all requirements of Definition~\ref{App:DefExpEq}  and finishes the proof.
\end{proof}

\section{LDPs for the measures $\tR_{n,W}$}
\label{se:RnW}

Now we are ready to establish Theorem~\ref{thm:GeneralSpeeds} that deals with LDPs for the random (edge rounded) graph~$\I G(n,W)$.

\begin{proof}[Proof of Theorem~\ref{thm:GeneralSpeeds}]
It is well-known (see e.g.\ \cite{DemboZeitouni10ldta}*{Theorem~4.2.13}) that, for a given speed, if two sequences of probability measures on the same  space are exponentially equivalent then one sequence satisfies an LDP if and only if the other satisfies the same LDP.
Thus the combination of Lemma~\ref{lm:ExpEq} and Theorem~\ref{thm:GeneralSpeedsH} easily implies all claims of Theorem~\ref{thm:GeneralSpeeds}, apart from Part~\ref{it:GS3}. 

Let us prove Part~\ref{it:GS3}. Here the speed is $s(n)=\omega(n^2)$ and the stated rate function is~$\chi_{\FORB(W)}$. There are two steps in the proof.
\begin{enumerate}
	\item [{\bf (I)}] If $\widetilde U\not \in \FORB(W)$, then there are $\alpha>0$ and $n_0\in \I N$ such that $\tR_{n,W}\left(B_{\delta_\Box}(\widetilde{U},\alpha)\right)=0$ for every $n\ge n_0$.
	\item [{\bf (II)}]  If $\alpha>0$ and $G\in \mathcal{G}_k$ are such that $\tind(G,W)\not= 0$, then for every $k\in\I N$ we have
	\beq{eq:II}
 \lim_{n\to \infty}\frac{1}{s(n)}\log\left(\tR_{n,W}\left(B_{\delta_\Box}\left(\widetilde{f^G},\frac{1}{k}+\alpha\right)\right)\right)=0.
\eeq
\end{enumerate}

First we show how {\bf (I)} and {\bf (II)} imply Part~\ref{it:GS3} of the theorem.

Let $F\subseteq \tW$ be a $\delta_\Box$-closed set.
If $F\cap \FORB(W)\not=\emptyset$, then the LDP upper bound \eqref{eq:upperGen} is vacuously satisfied.
Otherwise for some $\alpha>0$ we have $\delta_\Box(F,\FORB(W))\ge \alpha$ because we take the distance between two disjoint closed sets in a compact metric space.
For each $\T U\in F$ we pick $\alpha_U\in (0,\alpha)$ and $n_U\in \I N$ that satisfy~{\bf (I)}. Of course, the open balls centred at $\T U$ of radius $\alpha_U$ over all $\T U\in F$ cover~$F$. 
By the compactness of $F$ there is a finite subcover; let it be given by $\T U_1,\dots,\T U_l\in F$. 
Then, for every $n\ge \max_{i\in [l]} n_{U_i}$, we have
$$\tR_{n,W}(F)\le \tR_{n,W}\left(\bigcup_{i\in [l]} B_{\delta_\Box}\left(\widetilde{U}_i,\alpha_{U_i}\right)\right)=0.$$
Thus the LDP upper bound \eqref{eq:upperGen} vacuously holds.

Let $\mathcal{O}\subseteq \tW$ be a $\delta_\Box$-open set.
If $\mathcal{O}\cap \FORB(W)=\emptyset$, then the LDP lower bound \eqref{eq:lowerGen} is vacuously satisfied.
Otherwise there are $\T U\in \FORB(W)$ and $\alpha>0$ such that
$B_{\delta_\Box}(\widetilde{U},3\alpha)\subseteq \mathcal{O}$.
It is known that the cut-distance between any graphon $U$ and its typical $n$-vertex sample tends to 0 as $n\to\infty$ (see e.g.\ \cite[Lemma~10.16]{Lovasz:lngl}). Thus there is a graph $G\in \mathcal{G}_k$, where we can additionally assume that $k> 1/\alpha$, such that  $\tind(G,U)>0$ and $\delta_\Box(f^G,U)<\alpha$. It follows that
$$
B_{\delta_\Box}\left(\widetilde{f^G},\frac{1}{k}+\alpha\right)\subseteq B_{\delta_\Box}\left(\widetilde{U},\frac{1}{k}+\alpha+\delta_\Box(f^G,U)\right)\subseteq B_{\delta_\Box}(\widetilde{U},3\alpha)\subseteq \mathcal{O}.
$$
 Since $\T U\in \FORB(W)$, we must have $\tind(G,W)>0$. Thus we have by {\bf (II)} that
$$
\liminf_{n\to \infty}\frac{1}{s(n)}\log\left(\tR_{n,W}\left(\mathcal{O}\right)\right)\ge 
\liminf_{n\to \infty}\frac{1}{s(n)}\log\left(\tR_{n,W}\left(B_{\delta_\Box}\left(\widetilde{f^G},\frac{1}{k}+\alpha\right)\right)\right)=0,$$
and the LDP lower bound \eqref{eq:lowerGen} vacuously holds.

It remains to establish each of {\bf Steps (I)} and {\bf (II)}.

\medskip\noindent{\bf Step (I).}
Let $\T U\not \in \FORB(W)$.
By definition there are $k\in \I N$ and $H\in \mathcal{G}_k$ such that $\tind(H,W)=0$ and $\tind(H,U)>0$.
Since the map $\tind(H,\_)$ is $\delta_\Box$-continuous by Theorem~\ref{th:ConvOnW}, we can find $\alpha>0$ such that
$$B_{\delta_\Box}(\widetilde{U},\alpha)\subseteq \left\{\widetilde{V}\in \W: \tind(H,V)>\frac{1}{2}\,\tind(H,U)\right\}.$$
We show that this $\alpha$ works as required with $n_0:= \lceil {k^2}/\tind(H,U)\rceil$.

Let $n\ge n_0$ and $G\in \mathcal{G}_n$ be any graph such that $\delta_\Box(f^G,U)<\alpha$.
In particular, we have
\begin{equation}\label{eq:coolio}
\tind(H,f^G)-\binom{k}{2}\frac{1}{n}>\frac{1}{2}\left(\tind(H,U)-\frac{k^2}{n}\right)>0
\end{equation}
Note that
$$
\lambda^{\oplus k}\left(\left\{x\in [0,1]^k:\exists\, l\in [n] \ \exists\, i\not=j\in [k] \ x\ind{i},x\ind{j}\in I^n_l\right\}\right)\le \binom{k}{2}\frac{1}{n}.
$$
Thus, by
 (\ref{eq:coolio}) we can find 
an injective embedding of $H$ into~$G$. It follows by Fubini-Tonelli's theorem that
$\tind(G,W)\le \tind(H,W)$.
\hide{
; formally, 
\begin{equation*}
\begin{split}
\tind(G,W)= & \ \int_{[0,1]^{{n}}} \prod_{\{i',j'\}\in E(G)} W(x\ind{i'},x\ind{j'}) \prod_{\{i',j'\}\in E(\O G)} \left(1-W(x\ind{i'},x\ind{j'})\right) \dd\lambda^{\oplus n}(x) \\
\le & \ \int_{[0,1]^{{n}}} \prod_{i,j\in [k]\atop \{a(i),a(j)\}\in E(G)} W(x\ind{a(i)},x\ind{a(j)}) \prod_{i,j\in [k]\atop \{a(i),a(j)\}\in E(\O G)} \left(1-W(x\ind{a(i)},x\ind{a(j)})\right) \dd\lambda^{\oplus n}(x) \\
= & \ \int_{[0,1]^{{k}}} \prod_{\{i,j\}\in E(H)} W(x\ind{i},x\ind{j}) \prod_{\{i,j\}\in E(\O H)} \left(1-W(x\ind{i},x\ind{j})\right) \dd\lambda^{\oplus k}(x).
\end{split}
\end{equation*}
}
Since $\tind(H,W)=0$, we have $\tind(G,W)=0$. This implies that
$\tR_{n,W}(\{\tf{G}\})\le n!\, \tind(G,W)$ is zero. Since the discrete measure $\tR_{n,W}$ is supported on $\{\tf{G}\mid G\in\C G_n\}$, we conclude that $\tR_{n,W}\left(B_{\delta_{\Box}}\left(\widetilde{U},\alpha\right)\right)=0$, as desired.

\medskip\noindent{\bf Step (II)}.
Let $\alpha>0$ and let $G\in \mathcal{G}_k$ be such that $\tind(G,W)>0$. Let $H\sim \I G(n,W)$ be the random $n$-vertex sample from $W$.
For $\gamma\in (0,1)$ define
$$
X_\gamma:=\{x\in [0,1]^{{k}}: \forall \{i,j\}\in E(G) \ W(x\ind{i},x\ind{j})\ge \gamma,\
\forall \{i,j\}\in E(\O G) \ W(x\ind{i},x\ind{j})\le 1-\gamma \}.
$$
It follows from $\tind(G,W)>0$ and the $\sigma$-additivity of $\lambda^{\oplus k}$ that there is $\gamma>0$ such that 
$\lambda^{\oplus k}\left(X_{\gamma}\right)>0$.
By the Lebesgue Density Theorem, we can find ${{x}}\in X_\gamma$ such that ${{x}}\ind{i}\not ={{x}}\ind{j}$ whenever $i\not=j$ in $[k]$ and 
$$\lim_{l\to \infty}\frac{\lambda^{\oplus k}\left(X_{\gamma}\cap B_\infty\left({{x}},\frac{1}{l}\right)\right)}{\lambda^{\oplus k}\left(B_\infty\left({{x}},\frac{1}{l}\right)\right)}=1.$$

Fix $l\in \I N$ large enough so that
\begin{equation}\label{eq:fubinis}
\frac{\lambda^{\oplus k}\left(X_{\gamma}\cap B_\infty\left({{x}},\frac{1}{l}\right)\right)}{\lambda^{\oplus k}\left(B_\infty\left({{x}},\frac{1}{l}\right)\right)}>1-\frac{\alpha}2,
\end{equation}
and the intervals $J_{i}:=\left[{{x}}\ind{i}-\frac{1}{l},{{x}}\ind{i}+\frac{1}{l}\right]$, $i\in [k]$, are pairwise disjoint.

First, we prove~\eqref{eq:II} in the case when $n$ is taken only over integers divisible by~$k$. Informally, the main idea is that if an $n$-vertex graph $H$ is a uniform blow-up of $G$ then $\delta_\Box(\f{G},\f{H})=0$; furthermore, by adding any edges inside the $k$ parts of $H$ we increase the cut-distance by at most $1/k$ and thus we can afford any further $\alpha{n\choose 2}$ rounding errors without the cut-distance going over $1/k+\alpha$, 

\newcommand{\Mod}[2]{#1\mathrm{\,mod\,}#2}
Let $n\in \I N$ be such that $m:=n/k$ is an integer. For $i\in\I N$, let $\Mod{i}{k}$ denote the (unique) integer in $[k]$ congruent to $i$ modulo~$k$.
Define
$$\mathcal{X}:=\left\{x\in [0,1]^{{n}}:\forall i\in [n] \ x\ind{i}\in J_{\Mod i k}\right\}$$
and note that $\lambda^{\oplus k}\left(\mathcal{X}\right)=\left(\frac{2}{l}\right)^n$.

Let $x$ be a random $n$-vector, chosen uniformly from~$\mathcal{X}$. Thus we put exactly $m$ of the coordinates $x_i$ into each interval $J_{j}$. Call a pair $ij\in {[n]\choose 2}$ \emph{crossing} if $\Mod{i}{k}\not=\Mod{j}{k}$. Clearly, there are $m^2{k\choose 2}$ crossing pairs. Let $B(x)$ be the set of crossing pairs $ij\in {[n]\choose 2}$ such that either $\{\Mod{i}{k},\Mod jk\}\in E(G)$ and $W(x_i,x_j)<\gamma$ or $\{\Mod{i}{k},\Mod jk\}\in E(\O G)$ and  $W(x_i,x_j)>1-\gamma$. We call the pairs in $B(x)$ \emph{$x$-bad} and call all other crossing pairs \emph{$x$-good}. 

Take any $i_1,\dots,i_k\in [n]$ which are congruent  modulo $k$ to respectively $1,2,\dots,k$. Then $(x_{i_1},\dots,x_{i_k})$ is a uniform element of $\prod_{i=1}^k J_i$ and by Fubini-Tonelli's theorem, the probability of seeing at least one $x$-bad pair in ${\{i_1,\dots,i_k\}\choose 2}$ is the probability that a uniform element of $\prod_{i=1}^k J_i$ ends in the complement of $X_\gamma$. This is at most $\alpha/2$ by~\eqref{eq:fubinis}. Thus the expected fraction (for uniform $x\in\C X$) of $x$-bad pairs inside $\{i_1,\dots,i_k\}$ is at most $\alpha/2$.
Since every crossing pair  participates in the same number of $k$-tuples in ${[n]\choose k}$ with distinct residues modulo $k$, if we take a uniform such $k$-tuple and a uniform pair $ij$ inside it (which is necessarily crossing), then we get a uniformly distributed crossing pair. Moreover, we have already shown that, regardless of the chosen $k$-tuple, the probability (over the random choices of $ij$ and $x$) that $ij$ is $x$-bad is at most $\alpha/2$. Thus the expected
value of $|B(x)|$ is at most $\frac{\alpha}2\cdot m^2{k\choose 2}$. Let 
 $$
 \C Y:=\left\{x\in \C X\mid |B(x)|\le \alpha m^2{k\choose 2}\right\}.
 $$ 
 By Markov's inequality, the probability that $x\in \C Y$ is at least $1/2$. 
 
Note that if $ij\in {[n]\choose 2}$ is an $x$-good pair then the probability of $ij$ being rounded \emph{correctly} (meaning that $ij$ is an edge if and only if $\{\Mod ik,\Mod jk\}\in E(G)$) is at least $\gamma$. Thus if $x\in \C Y$ then the rounding produces at least $t:=\big\lceil (1-\alpha) m^2{k\choose 2}\big\rceil$ correct crossing pairs with probability at least $\gamma^t$. If this happens and we overlay all vertices in the $i$-th part of $H$ (namely those $j\in[n]$ with $\Mod{j}{k}=i$) with the $i$-th vertex of $G$ then the $L^1$-distance between the corresponding graphons $\f{H}$ and $\f{G}$ is at most 
 $$
 \frac{2}{(km)^2}\left( k {m\choose 2} + \alpha m^2{k\choose 2}\right)< \frac1k+\alpha,
$$
 where the first and second terms assume the worst-case scenario that all  non-crossing and all bad pairs respectively were rounded the ``wrong'' way.

Thus the probability that $\delta_\Box(\f{G},\f{H})\le \frac1k+\alpha$ for $H\sim \I G(n,W)$ is at least the product of the following terms:
 \begin{itemize}
\item $(2/l)^n$: the probability that the uniform $x\in [0,1]^n$ is in $\C X$,
\item $1/2$: the probability that a uniform $x\in\C X$ hits $\C Y$ (i.e., that there are at least $t$ $x$-good pairs),
\item $\gamma^t$: the probability that some $t$ of the $x$-good pairs, say the first $t$ of them in the lexicographic order on ${[n]\choose 2}$, are rounded correctly.
\end{itemize}
Altogether we have
\begin{equation}\label{eq:dsdsa}
\log\left(\tR_{n,W}\left(B_{\delta_\Box}\left(\widetilde{f^G},\frac{1}{k}+\alpha\right)\right)\right)
\ge  n\log\left(\frac{2}{l}\right) + \log\left(\frac12\right)+ t\log(\gamma),
\end{equation}
 which in absolute value is clearly in $O(n^2)\subseteq o(s(n))$ as $n\to\infty$ (over the multiples of~$k$), as desired.

Consider the case of general $n$. Assume that $n\ge 4k/\alpha$. Let $n'$ be the (unique) multiple of $k$ between $n-k+1$ and $n$. Let $H'$ be the subgraph of $H\sim \I G(n,W)$ induced by $[n']$. By Lemma~\ref{lm:Delete01} applied with $s:=n'/n$ we have $\delta_\Box(\f{H},\f{H'})\le 2(n/n'-1)<2k/n\le \alpha/2$. Clearly,  the random graph $H'$ is distributed as $\I G(n',W)$. Thus
\begin{equation*}
\tR_{n,W}\left(B_{\delta_\Box}\left(\widetilde{f^G},\frac{1}{k}+\alpha\right)\right)
\ge \tR_{n',W}\left(B_{\delta_\Box}\left(\widetilde{f^G},\frac{1}{k}+\frac{\alpha}{2}\right)\right).
\end{equation*}
 If $n\to\infty$, then $n'\to\infty$ over multiples of $k$ and the right-hand side is $\me^{-o(s(n))}$ by the inequality in~\eqref{eq:II} for $\alpha/2$, which was already established in this case, giving the required.

This finishes the proof of Theorem~\ref{thm:GeneralSpeeds}.\end{proof}

\section{Coloured graphons}\label{ColGraphons}

The definitions and results of this section are needed in order to establish the lower semi-continuity of the functions $J_{\V \alpha,\V p}$ and~$R_{\V p}$ that were defined in Section~\ref{Intro:n^2}.

Fix $k\in \mathbb{N}$.
By a \emph{$k$-coloured graphon} we mean a pair $(W,\mathcal{A})$ where $W\in \mathcal{W}$ and $\mathcal{A}\in \Ak$, that is, $\mathcal{A}$ is a partition of $[0,1]$ into $k$ measurable parts. One can view the partition $\mathcal{A}$ as a $k$-colouring of $[0,1]$.
Write $\Wk:=\C W\times \Ak$ for the space of all $k$-coloured graphons.
We define the pseudo-metric $d^{(k)}_\Box$ (the analogue of $d_\Box$) on $\Wk$ as
\begin{equation*}
d^{(k)}_\Box((U,\mathcal{A}),(V,\mathcal{B})):=  
\sup_{C,D\subseteq [0,1]} \sum_{i,j\in [k]}\left| \int_{C\times D} ({\I 1}_{A_i\times A_j}U-{\I 1}_{B_i\times B_j}V) \dd\lambda^{\oplus 2}\right|
 + \sum_{i\in [k]}\lambda(A_i\triangle B_i),
\end{equation*}
 for $(U,\mathcal{A}),(V,\mathcal{B})\in \Wk$, where $\C A=(A_1,\dots,A_k)$ and $\C B=(B_1,\dots,B_k)$.
 Informally speaking, two $k$-coloured graphons are close to each other in $d^{(k)}$ if
 they have similar distributions of coloured edges across cuts, where an edge is coloured by the colours of its endpoints. The second term is added so that e.g.\ we can distinguish two constant-0 graphons with different part measures.
 
The \emph{cut-distance} for coloured graphons is then defined as
\beq{eq:DeltaK}
\delta^{(k)}_\Box((U,\mathcal{A}),(V,\mathcal{B})):=\inf_{\phi,\psi} d^{(k)}_\Box((U,\mathcal{A})^\phi,(V,\mathcal{B})^\psi),
\eeq
where the infimum is taken over measure-preserving maps $\phi,\psi:[0,1]\to [0,1]$ and we denote
$(U,\mathcal{A})^\phi:=(U^\phi,\C A^\phi)$ and $\C A^\phi:=(\phi^{-1}(A_1),\dots,\phi^{-1}(A_k))$ 
with $(A_1,\dots,A_k)$ being the parts of~$\C A$.
As in the graphon case (compare with e.g.\ \Lo{Theorem 8.13}), some other definitions give the same distance (e.g.\ it is enough to take the identity function for $\psi$). We chose this definition as it is immediately clear from it that the function $\delta^{(k)}_\Box$ is symmetric and defines a pseudo-metric.
We write $\tWk$ for the corresponding quotient, where we identify two $k$-coloured graphons 
at $\delta^{(k)}_\Box$-distance~0. \arxiv{}{Its proof is obtained by the obvious adaptation of the proof of Lov\'asz and Szegedy~\cite[Theorem~5.1]{LovaszSzegedy07gafa} 
and can be found in the arxiv version of this paper~\cite{GrebikPikhurko23arxiv}.}

\begin{theorem}\label{th:compact colorod graphon}
The metric space $(\tWk,\delta^{(k)}_\Box)$ is compact. 
\end{theorem}
\arxiv{
\begin{proof}
The proof is obtained by the obvious adaptation of the proof of Lov\'asz and Szegedy~\cite[Theorem~5.1]{LovaszSzegedy07gafa} (see also~\cite[Theorem 9.23]{Lovasz:lngl}) that the space $(\tW,\delta_\Box)$ is compact.

Let $(W_n,\mathcal A_n)_{n\in\mathbb N}$ be an arbitrary sequence of elements of~$\Wk$. We have to find a subsequence 
that converges to some element of $\Wk$ with respect to~$\delta_\Box^{(k)}$. 

When dealing with the elements of $\Wk$,
we can ignore null subsets of $[0,1]$; thus all relevant statements, e.g.\ that one partition refines another, are meant to hold almost everywhere.

For $n\in\mathbb N$, let the parts of $\mathcal A_n$ be $(A_{n,1},\dots,A_{n,k})$ and, by applying a measure-preserving bijection to $(W_n,\mathcal A_n)$, assume by Theorem~\ref{th:ISMS} that the colour classes $A_{n,1},\dots,A_{n,k}\subseteq [0,1]$ are all intervals, 
coming in this order. By passing to a subsequence, assume that, for each $i\in [k]$, the length of $A_{n,i}$ converges to some $\alpha_i\in [0,1]$ as $n\to\infty$. With $\V\alpha:=(\alpha_1,\dots,\alpha_k)$, this gives rise to the ``limiting'' partition 
$$
\mathcal A
:=(\cI{\V\alpha}{1},\dots,\cI{\V\alpha}{k})
\in\Aalpha
$$ of $[0,1]$ into intervals. 

\renewcommand{\ell}{q}
Let $m_1:=k$ and inductively for $\ell=2,3,\ldots$ let $m_\ell$ be sufficiently large such that for every graphon $W$ and a measurable partition $\mathcal A'$ of $[0,1]$ with $|\mathcal A'|\le m_{\ell-1}$ there is a measurable partition $\mathcal P=(P_1,\dots,P_m)$ of $[0,1]$ refining $\mathcal A'$ such that $m\le m_\ell$ and $d_{\Box}(W,W_{\mathcal P})\le 1/\ell$. Here $W_{\mathcal P}$ denotes the projection of $W$ to the space of $\C P$-step graphons; namely, for every $i,j\in [m]$ with $P_i\times P_j$ non-null in $\lambda^{\oplus 2}$, $W_{\C P}$ assumes the constant value $\frac1{\lambda(P_i)\lambda(P_j)}\int_{P_i\times P_j} W\dd\lambda^{\oplus 2}$ on $P_i\times P_j$ (and, say, $W_{\C P}$ is defined to be 0 on all $\lambda^{\oplus 2}$-null products $P_i\times P_j$).
Such a number $m_\ell$
exists by~\cite[Lemma 9.15]{Lovasz:lngl}, a version of the Weak Regularity Lemma for graphons.
 
For each $n\in\mathbb N$, we do the following. Let $\mathcal P_{n,1}:=\mathcal A_n$ and, inductively on $\ell=2,3,\dots$, let $\mathcal P_{n,\ell}$ be the partition with at most $m_{\ell}$ parts obtained by applying the above Weak Regularity Lemma to $(W_n,\mathcal P_{n,\ell-1})$. By adding empty parts to $\mathcal P_{n,\ell}$, for each $\ell\ge 1$, we can assume that it has the same number of parts (namely, $m_\ell$) of each colour, that is, we can denote its parts as $(P_{n,\ell,i,j})_{i\in [k], j\in [m_\ell]}$, so that $P_{n,\ell,i,j}\subseteq A_{n,i}$ for all $(i,j)\in [k]\times [m_\ell]$.
Also, define $W_{n,\ell}:=(W_n)_{\mathcal P_{n,\ell}}$ to be the projection of the graphon $W_n$ on the space of $\mathcal P_{n,\ell}$-step graphons.

Then, iteratively for $\ell=2,3,\dots$, repeat the following. Find a measure-preserving bijection $\phi:[0,1]\to[0,1]$ such that $(\mathcal P_{n,\ell})^\phi$ is a partition into intervals and $\phi$ preserves the previous partitions
$\mathcal P_{n,1},\dots,\mathcal P_{n,\ell-1}$  (each of which is a partition into intervals by induction on~$\ell$). Then, for each $m\ge \ell$, replace $(W_{n,m},\mathcal P_{n,m})$ by $(W_{n,m},\mathcal P_{n,m})^\phi$. When we are done with this step, the following properties hold for each integer~$\ell\ge 2$:
\begin{itemize}
	\item\label{it:P1} $\delta_\Box^{(k)}((W_{n,\ell},\mathcal A_n),(W_n,\mathcal A_n))\le k^2/\ell$;
	\item\label{it:P3} 
the partition $\mathcal P_{n,\ell}$ refines $\mathcal P_{n,\ell-1}$ (and, inductively, also refines $\mathcal A_{n}=\mathcal P_{n,1}$);
	\item\label{it:P2} $|\mathcal P_{n,\ell}|= k m_\ell$ with exactly $m_\ell$ parts assigned to each colour class of $\mathcal A_n$.
\end{itemize}

Next, iteratively for $\ell=1,2,\dots$, we pass to a subsequence of $n$ so that for every $(i,j)\in [k]\times [m_\ell]$, the length of the interval $P_{n,\ell,i,j}$ converges, and for every pair $(i,j),(i',j')\in [k]\times [m_\ell]$, the common value of the step-graphon $W_{n,\ell}$ on $P_{n,\ell,i,j}\times P_{n,\ell,i',j'}$ converges.
	It follows that the sequence $W_{n,\ell}$ converges pointwise to some graphon $U_\ell$ which is itself a step-function with $km_\ell$ parts that are intervals. We use diagonalisation to find a subsequence of $n$ so that, for each $\ell\in\mathbb N$, $W_{n,\ell}$ converges to some step-graphon $U_\ell$ a.e.\ as $n\to\infty$, with the step partition $\mathcal P_\ell$ of  $U_\ell$  consisting of $km_\ell$ intervals and refining the partition~$\C A$. 
\hide{Moreover, 	we can denote the parts of $\mathcal P_\ell$ as
	$(P_{\ell,i,j})_{i\in[k], m\in[m_\ell]}$, so that $P_{\ell,i,1},\dots,P_{\ell,i,m_\ell}$ form a partition of~$\cI{\V\alpha}{i}$ for each $i\in [k]$.
}

It follows that, for all $s<t$ in $\I N$, the partition $\mathcal P_{t}$ is a refinement of $\mathcal P_s$ and, moreover,  $U_s$ is the conditional expectation $\mathbb E[U_t|\mathcal P_s]$ a.e. 
As observed in the proof of Lov\'asz and Szegedy~\cite[Theorem~5.1]{LovaszSzegedy07gafa},
this (and $0\le U_t\le 1$ a.e.) implies by the Martingale Convergence Theorem that $U_\ell$ converge a.e.\ to some graphon $U$ as $\ell\to\infty$. By the Dominated Convergence Theorem, $U_\ell\to U$ also in the $L^1$-distance.

We claim that $\delta_\Box((W_n,\mathcal A_n),(U,\mathcal A))\to 0$ as $n\to\infty$ (after we passed to the subsequence defined as above). Take any $\e>0$. Fix an integer $\ell>4/\e$ such that $\|U-U_\ell\|_1\le \e/(4k^2)$. Given $\ell$, fix $n_0$ such that  for all $n\ge n_0$ we have $\|U_\ell-W_{n,\ell}\|_1\le \e/(4k^2)$ and $\sum_{i=1}^k \lambda(\cI{\V\alpha}{i}\bigtriangleup A_{n,i})\le \e/4$. Then, for every $n\ge n_0$ we have
\begin{eqnarray*}
\delta_\Box^{(k)}((U,\mathcal A),(W_n,\mathcal A_n))&\le &
d_\Box^{(k)}((U,\mathcal A),(U_\ell,\mathcal A))
+ d_\Box^{(k)}((U_\ell,\mathcal A),(W_{n,\ell},\mathcal A_n))\\
&+& \delta_\Box^{(k)}((W_{n,\ell},\mathcal A_n),(W_n,\mathcal A_n))\\
&\le& k^2 \|U-U_\ell\|_1 + k^2\|U_\ell-W_{n,\ell}\|_1 + \sum_{i=1}^k \lambda(\cI{\V\alpha}{i}\bigtriangleup A_{n,i}) + 1/\ell\\
&\le& \e/4+\e/4+\e/4+\e/4\ =\ \e.
\end{eqnarray*}
 Since $\e>0$ was arbitrary, the claim is proved. Thus the metric space $(\tWk,\delta^{(k)}_\Box)$ is indeed compact.
\end{proof}
 }{}

\section{The lower semi-continuity of $J_{\V\alpha,\V p}$ and $R_{\V p}$}\label{Rate}

For this section we fix an integer $k\ge 1$, a symmetric $k\times k$ matrix $\V p=(p_{i,j})_{i,j\in [k]}\in [0,1]^{k\times k}$ and a non-zero real vector $\V\alpha\in [0,\infty)^k\NoZero$.
We show that the functions $J_{\V\alpha,\V p}$ and  $R_{\V p}$,  that were defined in~\eqref{eq:JAlphaP} and~\eqref{eq:R} respectively,
are lower semi-continuous functions on $(\tW,\delta_\Box)$.

Let $\Gamma:\Wk\to \C W$ be the map that forgets the colouring, i.e., $\Gamma(U,\mathcal{A}):=U$
for $(U,\C A)\in \Wk$.
For $i,j\in [k]$, let the map $\Gamma_{i,j}:\Wk\to \C W$ send $(U,(A_1,\dots,A_k))\in \Wk$  to the graphon $V$ defined as
$$
V(x,y):=\left\{\begin{array}{ll} 
	U(x,y),& (x,y)\in (A_i\times A_j)\cup (A_j\times A_i),\\
	p_{i,j},& \mbox{otherwise,}
\end{array}
\right.\quad \mbox{for $x,y\in [0,1]$.}
$$

\begin{lemma}\label{lm:Lipschitz}
	The maps $\Gamma$ and $\Gamma_{i,j}$, for ${i,j\in [k]}$, are $1$-Lipschitz maps from $(\Wk,d_\Box^{(k)})$ to $(\C W,d_\Box)$.
\end{lemma}
\begin{proof}
	First, consider $\Gamma:\Wk\to \mathcal{W}$. Take arbitrary $(U,\mathcal{A}),(V,\mathcal{B})\in \Wk$.
	Let $\C A=(A_1,\dots,A_k)$ and $\mathcal{B}=(B_1,\dots,B_k)$. Clearly, the pairwise products $A_i\times A_j$ (resp.\ $B_i\times B_j$) for $i,j\in [k]$ partition $[0,1]^2$. Thus
	we have
	\begin{eqnarray}		
			d_\Box(\Gamma(U,\mathcal{A}),\Gamma(V,\mathcal{B})) &= & d_\Box(U,V) \
			= \ \sup_{C,D\subseteq [0,1]} \left |\int_{C\times D} (U-V) \dd\lambda^{\oplus 2}\right |\nonumber \\
			&\le & \sup_{C,D\subseteq [0,1]}\sum_{i,j\in [k]}\left|\int_{C\times D} (U\,{\I 1}_{A_i\times A_j}-V\,\I 1_{B_i\times B_j}) \dd\lambda^{\oplus 2}\right| \nonumber\\
			&\le &  d_{\Box}^{(k)}((U,\mathcal{A}),(V,\mathcal{B})).\label{eq:Ga}
			\end{eqnarray}
	Thus the function $\Gamma$ is indeed $1$-Lipschitz.

	The claim about $\Gamma_{i,j}$ follows by observing that
	$$d_\Box(\Gamma_{i,j}(U,\mathcal{A}),\Gamma_{i,j}(V,\mathcal{B}))\le d_\Box(\Gamma(U,\mathcal{A}),\Gamma(V,\mathcal{B}))$$
	for every $(U,\mathcal{A}),(V,\mathcal{B})\in\Wk$.\end{proof}

\begin{lemma}\label{lm:Cont}
	Let $F$ be $\Gamma$ or $\Gamma_{i,j}$ for some $i,j\in [k]$.
	Then $F$  gives rise to a well-defined function $\tWk\to\tW$ which, moreover, is $1$-Lipschitz
	as a function from $(\tWk,\delta_\Box^{(k)})$ to~$(\tW,\delta_\Box)$.
\end{lemma}
\begin{proof} 
	
	Take any $(U,\C A),(V,\C B)\in\Wk$. Let $\e>0$ be arbitrary. Fix measure-preserving maps $\phi,\psi:[0,1]\to [0,1]$ with $d_\Box^{(k)}((U,\C A)^\psi,(V,\C B)^\phi)<\delta_\Box^{(k)}((U,\C A),(V,\C B))+\e$. By Lemma~\ref{lm:Lipschitz}, 
we have 
\begin{eqnarray*}
	\delta_\Box\left(F(U,\C A),F(V,\C B)\right)&\le& 
	d_\Box\left((F(U,\C A))^\psi,(F(V,\C B))^\phi\right)\ =\
	d_\Box\left(F(U^\psi,\C A^\psi),F(V^\phi,\C B^\phi)\right)\\
	&\le& d_\Box^{(k)}((U^\psi,\C A^\psi),(V^\phi,\C B^\phi))
	\ <\ \delta_\Box^{(k)}\left((U,\C A),(V,\C B)\right)+\e.
\end{eqnarray*}
This implies both claims about $F$ as $\e>0$ was arbitrary.
\end{proof}

For $(U,(A_1,\dots,A_k))\in \Wk$, define
\beq{eq:DefIkp}
I^{(k)}_{\V p}(U,(A_1,\dots,A_k)):=\frac12\sum_{i,j\in [k]} \int_{A_i\times A_j} I_{p_{i,j}}(U) \dd\lambda^{\oplus 2}.
\eeq

	In the special case of~\eqref{eq:DefIkp}  when $k=1$ and $p_{1,1}=p$ (and we ignore the second component since ${\bf A}^{(1)}$ consists of just the trivial partition of $[0,1]$ into one part), we get the function
$I_p:\C W\to[0,\infty]$ of Chatterjee and Varadhan defined in~\eqref{eq:IpCV}.

\begin{lemma}\label{lm:lsc}
	The function $I^{(k)}_{\V p}$ gives a well-defined function $\tWk\to [0,\infty]$ which, moreover, is lower semi-continuous as a function on $(\tWk,\delta^{(k)}_\Box)$.
\end{lemma}
\begin{proof}
Note that we can write
	\beq{eq:IkWCV}
	I^{(k)}_{\V p}(U,\mathcal{A})=\sum_{1\le i\le j\le k} I_{p_{i,j}}(\Gamma_{i,j}(U,\mathcal{A})),\quad (U,\C A)\in\Wk,
	\eeq
	because $\Gamma_{i,j}(U,\C A)$ assumes value $p_{i,j}$ outside of $(A_i\times A_j)\cup (A_j\times A_i)$ while $I_p(p)=0$ for any $p\in [0,1]$. Recall that, by Theorem~\ref{th:CV},
	$I_p$ gives a well-defined function $\tW\to[0,\infty]$ for every $p\in [0,1]$. 
	Thus, by Lemma~\ref{lm:Cont}, the right-hand side of~\eqref{eq:IkWCV} does not change if we replace $(U,\C A)$ by any other element of $\Wk$ at $\delta_\Box^{(k)}$-distance 0. We conclude that $I_{\V p}^{(k)}$ gives a well-defined function on~$\tWk$.
	
	Each composition $I_p\circ \Gamma_{i,j}$ 
	is lsc as a function $(\tWk,\delta_{\Box}^{(k)})\to [0,\infty]$	
	because, for every $\rho\in\I R$, the level set $\{I_p\circ \Gamma_{ij}\le \rho\}$ is closed as the pre-image under the continuous function $\Gamma_{i,j}:\tWk\to\tW$ of the closed 
	set~$\{I_p\le \rho\}$.
	(Recall that the function $I_p:\tW\to [0,\infty]$ is lsc by Theorem~\ref{th:CV}.)
	Thus $I^{(k)}_{\V p}:\tWk\to[0,\infty]$ is lsc by~\eqref{eq:IkWCV}, as a finite sum of lsc functions.
\end{proof}

Now we are ready to show that the functions $J_{\V\alpha,\V p}$ and $R_{\V p}$ are lsc (thus proving Theorem~\ref{th:JLSC} and the first part of Theorem~\ref{th:ourLDP}). The argument showing the lower semi-continuity of these functions is motivated by the proof of the Contraction Principle. 

\begin{theorem}\label{th:lsc}
	For every symmetric matrix $\V p\in [0,1]^{k\times k}$ and every non-zero real vector $\V\alpha\in [0,\infty)^k$, the functions $J_{\V\alpha,\V p}$ and $R_{\V p}$ are lower semi-continuous on $(\tW,\delta_\Box)$.	
\end{theorem}

\begin{proof}
	Note that $J_{\V\alpha,\V p}(\T U)$ for any $U\in\C W$ is equal to the infimum of $I_{\V p}^{(k)}(V,\C A)$ over all $(V,\C A)\in \Walpha$ such that $\Gamma(V,\C A)=V$ belongs to~$\T U$. Indeed, for any partition $\C A\in \Aalpha$ of $[0,1]$, one can find 
	a measure-preserving Borel bijection	$\phi$ of $[0,1]$ such that $\C A^\phi$ is equal a.e.\ to $(\cI{\V\alpha}{1},\dots,\cI{\V\alpha}{k})$.

	In the rest of the proof, let us view $\Gamma$ and $I^{(k)}_{\V p}$ as functions on $\tWk$ (by Lemmas~\ref{lm:Cont} and~\ref{lm:lsc}). Thus we have
	\begin{equation}\label{eq:Jap2}
	J_{\V\alpha,\V p}(\T U)=\inf_{\Gamma^{-1}(\T U)\cap\tWalpha}\, I_{\V p}^{(k)},
\quad \mbox{for each $ U\in\C W$},
	\end{equation}
 where $\tWalpha\subseteq \tWk$ denotes the set of all $\delta_\Box^{(k)}$-equivalence classes that intersect $\Walpha$ (equivalently, lie entirely inside $\Walpha$).
 
Take any graphon $U\in\C W$. The pre-image $\Gamma^{-1}(\T U)$ is a closed subset of $\tWk$ by the continuity of $\Gamma$ (Lemma~\ref{lm:Cont}). Also, $\tWalpha$
	is a closed subset of $\tWk$: if 
	$(V,(B_1,\dots,B_k))\in \Wk$ is not in $\Walpha$, 
	then the $\delta_{\Box}^{(k)}$-ball of radius e.g.\ 
	$$
	 \frac12\,\left\|\,(\lambda(B_i))_{i\in [k]} - \frac1{\|\V\alpha\|_1} \V\alpha\,\right\|_1>0
	 $$
	  around it is disjoint from~$\Walpha$. Recall that the space $\tWk$ is compact by Theorem~\ref{th:compact colorod graphon}. Thus the infimum in~\eqref{eq:Jap2} is taken over a (non-empty) compact set.  As any lsc function attains its infimum on any non-empty compact set and
	$I^{(k)}_{\V p}:\tWk\to [0,\infty]$ is lsc by Lemma~\ref{lm:lsc}, 
	there is $(V,\C A)\in \Walpha$ such that $V\in\T U$ and $I_{\V p}^{(k)}(\Tk{V,\C A})=J_{\V\alpha,\V p}(\T U)$,
	where $\Tk{V,\C A}$ denotes the $\delta_\Box^{(k)}$-equivalence class of $(V,\C A)$.
	
	Thus for any $\rho\in\I R$ the level set 
	$\{J_{\V\alpha,\V p}\le \rho\}$ 
	is equal to the image of $\{I_{\V p}^{(k)}\le \rho\}\cap\tWalpha$ under~$\Gamma$. Since the function $I_{\V p}^{(k)}$ is lsc by Lemma~\ref{lm:lsc}, the level set $\{I_{\V p}^{(k)}\le \rho\}$ is a closed and thus compact subset of $\tWk$. Thus the set $\{I_{\V p}^{(k)}\le \rho\}\cap\tWalpha$ is compact. Its image $\{J_{\V\alpha,\V p}\le \rho\}$  under the continuous map $\Gamma:\tWk\to \tW$ is compact and thus closed. Since $\rho\in\I R$ was arbitrary, the function $J_{\V\alpha,\V p}:\tW\to [0,\infty]$ is lsc.

	Since $R_{\V p}(\T U)$ is equal to the infimum of $I_{\V p}^{(k)}$ over $\Gamma^{-1}(\T U)$, the same argument (except we do not need to intersect $\Gamma^{-1}(\T U)$ with $\tWalpha$ anywhere) also works for $R_{\V p}$. 
\end{proof}

\section{Proof of Theorem~\ref{th:GenLDP}}\label{GenLDP}

Before we can prove  Theorem~\ref{th:GenLDP}, we need to give two auxiliary lemmas. The first one states, informally speaking, that the measure $\tP_{\V a,\V p}$ is ``uniformly continuous'' in~$\V a$.

\begin{lemma}\label{lm:DifferentRatios}
	For every symmetric	matrix $\V p\in [0,1]^{k\times k}$, real $\e\in (0,1)$ and non-zero integer vectors $\V a,\V b\in\I N^k_{\ge 0}$, if 
	$$\|\V b-\V a\|_1\le \e\, \min\left\{\,\|\V a\|_1,\|\V b\|_1\,\right\}$$ 
	then there is a (discrete) measure $\T{\I C}$ on $\tW\times \tW$ which gives a coupling between $\tP_{\V a,\V p}$ and $\tP_{\V b,\V p}$ such that for every $(\T U,\T V)$ in the support of $\T{\I C}$ we have $\delta_\Box(U,V)\le 4\e/(1-\e)$.
	\end{lemma}

\begin{proof}
	Let $m:=\|\V a\|_1$ and $n:=\|\V b\|_1$. 
		Let $[m]=A_1\cup\dots\cup A_k$ (resp.\ $[n]=B_1\cup\dots\cup B_k$) be the partition into consecutive intervals with $a_1,\dots,a_k$ (resp.\ $b_1,\dots,b_k$) elements. For each $i\in [k]$ fix some subsets $A_i'\subseteq A_i$ and $B_i'\subseteq B_i$ of size $\min(a_i,b_i)$. Define $A':=\cup_{i=1}^k A_i'$ and
	$B':=\cup_{i=1}^k B_i'$. Fix any bijection $h:A'\to B'$ that sends each $A_i'$ to $B_i'$. We have
	\beq{eq:m-A'}
	\left|\,[m]\setminus A'\,\right|\le\sum_{i=1}^k \max(0,a_i-b_i)\le \|\V a-\V b\|_1\le \e m
	\eeq
	and similarly $\left|\,[n]\setminus B'\,\right|\le \e n$. 
	
We can couple random graphs $G\sim \I G(\V a,\V p)$ and $H\sim\I G(\V b,\V p)$ so that every pair $\{x,y\}$ in $A'$ is an edge in $G$ if and only if $\{h(x),h(y)\}$ is an edge in~$H$.
This is possible because, with $i,j\in [k]$ satisfying $(x,y)\in A_i'\times A_j'$, the probability of $\{x,y\}\in E(G)$ is $p_{i,j}$, the same as the probability with which $h(x)\in B_i$ and $h(y)\in B_j$ are made adjacent in~$H$ (so we can just use the same coin flip for both pairs). By making all edge choices to be mutually independent otherwise, we get a probability measure $\I C$ on pairs of graphs which is a coupling between $\I G(\V a,\V p)$ and $\I G(\V b,\V p)$.
 The corresponding measure $\T{\I C}$ on $\tW\times \tW$ gives  a coupling between $\tP_{\V a,\V p}$ and $\tP_{\V b,\V p}$. 

Let us show that this coupling $\T{\I C}$ satisfies the distance requirement of the lemma. Take any pair of graphs $(G,H)$ in the support of~$\I C$. Let $G':=G[A']$ be obtained from $G$ by removing all vertices in $[m]\setminus A'$. We remove at most $\e$ fraction of vertices by~\eqref{eq:m-A'}. By relabelling the vertices of~$G$, we can assume that all removed vertices come at the very end. Then the union of the intervals in the graphon $\f{G}$ corresponding to the vertices of $G'$ is an initial segment of $[0,1]$ of length $s\ge 1-\e$ and the graphon of $G'$ is the pull-back of $\f{G}$ under the map $x\mapsto sx$.
By Lemma~\ref{lm:Delete01}, we have $\delta_\Box(\f{G},\f{G'})\le 2(\frac{1}{1-\e}-1)= \frac{2\e}{1-\e}$. By symmetry, the same estimate applies to the pair $(H,H')$, where $H'$ is obtained from $H$ by deleting all vertices from~$[n]\setminus B'$. By the definition of our coupling, the function $h:V(G')\to V(H')$ is (deterministically) an isomorphism between $G'$ and~$H'$. Thus the graphons of $G'$ and $H'$ are weakly isomorphic. 
This gives by the Triangle Inequality the required upper bound on the cut-distance between $\tf{G}$ and $\tf{H}$. 
\end{proof}

The function $J_{\V\alpha,\V p}(\T V)$ is not in general continuous in $\V\alpha$ even when $\V p$ and $\T V$ are fixed. (For example, if $k=2$, $V=1-\f{K_2}$ is the limit of $K_n\sqcup K_n$, the disjoint union of two cliques of order $n$ each, and $\V p$ is the identity $2\times 2$ matrix then $J_{\V\alpha,\V p}(\T V)$ is $0$ for $\V\alpha=(\frac12,\frac12)$ and $\infty$ otherwise.) However, the following version of ``uniform semi-continuity'' in $\V\alpha$ will suffice for our purposes.

\begin{lemma}\label{lm:ContOfJ}
	Fix any symmetric $\V p\in [0,1]^{k\times k}$.
	Then for every $\eta>0$ there is $\e=\e(\eta,\V p)>0$ such that if  vectors ${\V\gamma},\V\kappa\in [0,\infty)^k$ with $\|\V\gamma\|_1=\|\V\kappa\|_1=1$ satisfy
	\begin{equation}\label{eq:cond}
		\kappa_i\le (1+\e)\gamma_i,\quad\mbox{for every $i\in [k]$},
		\end{equation}
	 then
	for every $U\in {\mathcal{W}}$ there is ${V}\in {\C W}$ with $\delta_\Box(U,{V})\le \eta$ and
	$$J_{\V\kappa,\V p}(\T{V})\le J_{{\V\gamma},\V p}(\T U) +\eta.$$
\end{lemma}
\begin{proof}
	For every fixed $p\in (0,1)$, the relative entropy function $h_p:[0,1]\to [0,\infty]$ is bounded (in fact, by  $\max\{h_p(0),h_p(1)\}<\infty$). Thus we can find a constant $C\in (0,\infty)$ that depends on $\V p$ only such that for every $x\in [0,1]$ and $i,j\in [k]$, $h_{p_{i,j}}(x)$ is either $\infty$ or at most~$C$.

Let us show that any positive $\e\le \min\{\,\frac{\eta}{3(2C+\eta)},\,\frac{\eta}2\,\}$ works. Let ${\V\gamma}$, $\V\kappa$, and $U$ be as in the lemma. 	Assume that $J_{{\V\gamma},\V p}(\T U)<\infty$ for otherwise  we can trivially take ${V}:=U$. By the choice of $C$, we have that $J_{{\V\gamma},\V p}(\T U)\le C$. 

By replacing $U\in\C W$ with a weakly isomorphic graphon, assume that, for the partition $(\cI{{\V\gamma}}{i})_{i\in [k]}$ of $[0,1]$ into intervals, 
we have
	$$\frac{1}{2}\sum_{i,j\in [k]} \int_{\cI{{\V\gamma}}{i}\times \cI{{\V\gamma}}{j}} I_{p_{i,j}}(U) \dd\lambda^{\oplus 2}<J_{{\V\gamma},\V p}(\T U)+\frac{\eta}{2}.$$
	
Let $\phi:[0,1]\to [0,1]$ be the a.e.\ defined function  which  on each interval $\cI{{\V\kappa}}{i}$ of positive length is the increasing linear function that bijectively maps this interval onto $\cI{{\V\gamma}}{i}$. 
The 
Radon-Nikodym derivative $D:=\frac{\dd \mu}{\dd\lambda}$ of the push-forward $\mu:=\phi_*\lambda$ of the Lebesgue measure $\lambda$ along $\phi$ assumes value
$\kappa_i/\gamma_i$ a.e.\ on~$\cI{{\V\gamma}}{i}$. (The union of $\cI{\V\gamma}{i}$ with $\gamma_i=0$ 
is a countable and thus null set,
and can be ignored from the point of view of~$D$.)
	Let $V:=U^\phi$. It is a graphon by the first part of Lemma~\ref{lm:AlmostMP}. 
We have by the definition of $J_{\V\kappa,\V p}$ that
	\begin{equation*}
			J_{\V\kappa,\V p}(\T{V})\le  \ \frac{1}{2}\sum_{i,j\in [k]} \int_{\cI{{\V\kappa}}{i}\times \cI{{\V\kappa}}{j}} I_{p_{i,j}}(U^\phi) \dd \lambda^{\oplus 2} 
			=  \ \frac{1}{2}\sum_{i,j\in [k]} \frac{\kappa_i\kappa_j}{\gamma_i\gamma_j}\int_{\cI{{\V\gamma}}{i}\times \cI{{\V\gamma}}{j}} I_{p_{i,j}}(U) \dd\lambda^{\oplus 2}.
	\end{equation*}
 By $\frac{\kappa_i\kappa_j}{\gamma_i\gamma_j}\le (1+\e)^2\le 1+3\e$, this  is at most
 $$
 (1+3\e)\,\frac1{2}\sum_{i,j\in [k]} \int_{\cI{{\V\gamma}}{i}\times \cI{{\V\gamma}}{j}} I_{p_{i,j}}(U) \dd\lambda^{\oplus 2} 
\le (1+3\e)  \left(J_{{\V\gamma},\V p}(\T U)+\frac{\eta}2 \right).
 $$	
 By $J_{{\V\gamma},\V p}(\T U)\le C$ and our choice of $\e$, this in turn is at most
$$
J_{{\V\gamma},\V p}(\T U) + 3\e C+(1+3\e)\frac{\eta}2\le  \ J_{{\V\gamma},\V p}(\T U)+ \eta.
$$

Thus it remains to estimate the cut-distance between $U$ and~${V}=U^\phi$. Since the Radon-Nikodym derivative $D=\frac{\dd \phi_*\lambda}{\dd\lambda}$ satisfies $D(x)\le 1+\e$ for a.e.\ $x\in [0,1]$ by~\eqref{eq:cond}, we have by Lemma~\ref{lm:AlmostMP} that $\delta_\Box(U,{V})\le 2\e\le \eta$. Thus $V$ is as desired.\end{proof}

Also, we will need the following characterisation of an LDP for compact spaces that involves only the measures of shrinking balls. Its proof can be found in e.g.~\cite[Lemma 2.3]{Varadhan16ld}.

\begin{lemma}\label{lm:LDP}
	Let $(X,d)$ be  a compact metric space, $s:\I N\to (0,\infty)$ satisfy $s(n)\to\infty$, and $I:X\to [0,\infty]$ 
	be a lower semi-continuous function on~$(X,d)$. Then
	a sequence of Borel probability measures $(\mu_n)_{n\in \I N}$ on $(X,d)$ satisfies an LDP 
	with speed $s$
	and rate function $I$
	if and only if 
	\begin{eqnarray}
		\lim_{\eta\to 0}\liminf_{n\to\infty} \frac1{s(n)}\,{\log\Big(\mu_n\big(
		 B_d(x,\eta)
			\big)\Big)} 
		&\ge& -I(x),\quad\mbox{for every $x\in X$,}\label{eq:lower}\\
		\lim_{\eta\to 0}\limsup_{n\to\infty} \frac1{s(n)}\,{ \log\Big(\mu_n\big(
			\O  B_d(x,\eta)
			\big)\Big)} &\le & -I(x),\quad\mbox{for every $x\in X$}.\label{eq:upper}\qqed
	\end{eqnarray}
\end{lemma}
	

In fact, under the assumptions of Lemma~\ref{lm:LDP}, the bounds~\eqref{eq:lowerGen} and~\eqref{eq:lower} (resp.\ \eqref{eq:upperGen} and~\eqref{eq:upper}) are equivalent to each other. So we will also refer to 
	\eqref{eq:lower} and~\eqref{eq:upper}
	as the \emph{(LDP) lower bound} and the \emph{(LDP) upper bound} respectively.

Now we are ready to prove Theorem~\ref{th:GenLDP}.

\begin{proof}[Proof of Theorem~\ref{th:GenLDP}]
Recall that we have to establish an LDP for the sequence $(\tP_{\V a_n,\V p})_{n\in\I N}$ of measures, where the integer vectors $\V a_n$, after scaling by $1/n$, converge to a non-zero vector $\V\alpha\in[0,\infty)^k\NoZero$. Since the underlying metric space $(\tW,\delta_{\Box})$ is compact and the proposed rate function $J_{\V\alpha,\V p}$ is lsc by Theorem~\ref{th:JLSC} (which is a part of the already established Theorem~\ref{th:lsc}), it suffices to prove the bounds in~\eqref{eq:lower} and~\eqref{eq:upper} of Lemma~\ref{lm:LDP} for any given graphon~$\T U\in \tW$.

Let us show the upper bound first, that is, that 
\beq{eq:red9}
 \lim_{\eta\to 0}\limsup_{n\to\infty} \frac{1}{(\|\V a_n\|_1)^2}\log \tP_{\V a_n,\V p}(\Sball(\widetilde U,\eta))\le -J_{\V\alpha,\V p}(\T U).
\eeq
 Recall that $\Sball(\widetilde U,\eta):=\{\widetilde V\mid \delta_\Box(U,V)\le \eta)\}$ denotes the closed ball of radius $\eta$ around $\widetilde U$ in $(\tW,\delta_\Box)$. 

 Take any $\eta>0$. Assume that, for example, $\eta<1/9$. Let $m\in\I N$ be sufficiently large. Define
\beq{eq:b}
 \V b:=(a_{m,1}\I 1_{\alpha_1>0},\ldots,a_{m,k}\I 1_{\alpha_k>0})\in\I N^k_{\ge 0}.
 \eeq
 In other words, we let $\V b$ to be $\V a_m$ except we set its $i$-th entry $b_i$ to be~0  for each $i\in [k]$ with $\alpha_i=0$.
 By Theorem~\ref{th:BCGPS} (the LDP by Borgs et al~\cite{BCGPS}) and Theorem~\ref{th:JLSC}, we have that
 \beq{eq:BCGPSApplied1}
 \limsup_{n\to\infty} \frac{1}{(n\,\|\V b\|_1)^2}\log\tP_{n\V b,\V p}(\Sball(\widetilde U,\eta))\le 
 -\inf_{\Sball(\widetilde U,\eta)} J_{\V b,\V p}.
 \eeq

Since $m$ is sufficiently large, we can assume that $\|\frac1n\, \V a_n-\V\alpha\|_1\le \xi\,\|\V\alpha\|_1$ for every $n\ge m$, where e.g.~$\xi:=\eta/40$. 
In particular, we have 
$\|\V a_m-\V b\|_1
\le \|\V a_m-m\V\alpha\|_1
\le \xi m\|\V\alpha\|_1$ from which it follows that 
$$
 \|\V b -m\V\alpha\|_1\le \|\V b -\V a_m\|_1+\| \V a_m - m\V\alpha\|_1\le 2\xi m\,\|\V\alpha\|_1.
$$ 
Let $n$ be sufficiently large (in particular, $n\ge m$) and let $n':= \floor{n/m}$. By above, we have
\hide{\begin{eqnarray}
\|n'\V b-\V a_n\|_1&\le& \textstyle 
 \|\V b\|_1+\|\frac{n}{m}\V b-n\V\alpha\|_1+\|n\V\alpha-\V a_n\|_1\
\le\ \|\V b\|_1+3\xi n\|\V\alpha\|_1\nonumber\\
 &\le &\textstyle
 \|\V b\|_1 + 4\xi \min\{\,\|\frac{n}{m}\V b\|_1,\,\| \V a_n\|_1\,\}\ \le\ \frac{\eta}9\,\min\{\,\|n'\V b\|_1,\,\| \V a_n\|_1\,\}.\label{eq:n'b-an}
\end{eqnarray}
}
\begin{eqnarray}
\|n'\V b-\V a_n\|_1&\le& \textstyle 
 \|\V b\|_1+\|\frac{n}{m}\V b-n\V\alpha\|_1+\|n\V\alpha-\V a_n\|_1\nonumber\\
&\le& \|\V b\|_1+3\xi n\|\V\alpha\|_1\nonumber\\
 &\le &\textstyle
 \|\V b\|_1 + 4\xi \min\{\,\|\frac{n}{m}\V b\|_1,\,\| \V a_n\|_1\,\}\nonumber\\ 
&\le& \textstyle\frac{\eta}9\,\min\{\,\|n'\V b\|_1,\,\| \V a_n\|_1\,\}.\label{eq:n'b-an}
\end{eqnarray}
 Thus, by Lemma~\ref{lm:DifferentRatios}, there is a coupling $\T{\I C}$ between $\tP_{n'\V b,\V p}$ and $\tP_{\V a_n,\V p}$ such that for every $(\widetilde{V},\widetilde{W})$ in the support of $\T{\I C}$ we have $\delta_\Box(V,{W})\le \eta/2$. Thus, if $\widetilde{W}\sim\tP_{\V a_n,\V p}$ lands in $\Sball(\widetilde{ U},\eta/2)$ then necessarily $\widetilde{V}\sim \tP_{n'\V b,\V p}$ lands in $\Sball(\widetilde{U},\eta)$. This gives that
 $$
 \tP_{\V a_n,\V p}(\Sball(\widetilde U,\eta/2))\le \tP_{n'\V b,\V p}(\Sball(\widetilde U,\eta)).
 $$
 Since this is true for every sufficiently large $n$ and it holds by~\eqref{eq:n'b-an} that, for example, 
 $\|\V a_n\|_1/(n'\,\|\V b\|_1) \le 1+\eta/9\le \sqrt{1+\eta}$, we have that
 \beq{eq:red1}
 \limsup_{n\to\infty} \frac{1+\eta}{(\|\V a_n\|_1)^2}\log \tP_{\V a_n,\V p}(\Sball(\widetilde U,\eta/2))\le \limsup_{n'\to\infty}  \frac{1}{(n'\,\|\V b\|_1)^2}\log \tP_{n'\V b,\V p}(\Sball(\widetilde U,\eta)).
 \eeq

Let us turn our attention to the right-hand side of~\eqref{eq:BCGPSApplied1}. 
Pick some $\T V\in \Sball(\T U,\eta)$ with $J_{\V b,\V p}(\T V)\le \inf_{\Sball(\T U,\eta)} J_{\V b,\V p}+\eta$. (In fact, by the lower semi-continuity of $J_{\V b,\V p}$ and the compactness of $\Sball(\T U,\eta)$, we could have required
that $J_{\V b,\V p}(\T V)= \inf_{\Sball(\T U,\eta)} J_{\V b,\V p}$.) 
Since $\V a_n/\|\V a_n\|_1$ converges to $\V\alpha/\|\V\alpha\|_1$ as $n\to\infty$ and we chose $m$ to be sufficiently large,
we can assume that $b_i=0$ if and only if $\alpha_i=0$ and that 
Lemma~\ref{lm:ContOfJ} applies for $\V\gamma:=\V b/\|\V b\|_1$ and $\V\kappa:=\V\alpha/\|\V\alpha\|_1$ (and our~$\eta$). 
The lemma gives that 
there is $\T{W}\in\Sball(\T V,\eta)$ such that $J_{\V\alpha,\V p}(\T{W})-\eta\le J_{\V b,\V p}(\T{V})$. Thus we get the following upper bound on the right-hand side of~\eqref{eq:BCGPSApplied1}:
\beq{eq:red2}
-\inf_{\Sball(\T U,\eta)} J_{\V b,\V p}\le -J_{\V b,\V p}(\T V)+\eta\le -J_{\V\alpha,\V p}(\T{W})+2\eta\le -\inf_{\Sball(\T U,2\eta)} J_{\V \alpha,\V p}+2\eta.
\eeq
 By putting~\eqref{eq:BCGPSApplied1}, \eqref{eq:red1} and~\eqref{eq:red2} together we get that, for every $\eta\in (0,1/9)$,
 $$
 \limsup_{n\to\infty} \frac{1+\eta}{(\|\V a_n\|_1)^2}\log \tP_{\V a_n,\V p}(\Sball(\widetilde U,\eta/2))\le
-\inf_{\Sball(\T U,2\eta)} J_{\V \alpha,\V p}+2\eta.
$$ 
 If we take here the limit as $\eta\to 0$ then the infimum in the right-hand side converges to~$J_{\V \alpha,\V p}(\T U)$ by the lower semi-continuity of $J_{\V \alpha,\V p}$ (established in Theorem~\ref{th:JLSC}), giving the claimed upper bound~\eqref{eq:red9}.

Let us turn to the lower bound, i.e.\ we prove~\eqref{eq:lower} for $\T U\in\tW$. As before, take any sufficiently small $\eta>0$, then sufficiently large $m\in\I N$ and define $\V b$ by~\eqref{eq:b}. By Theorems~\ref{th:BCGPS} and~\ref{th:JLSC} applied to the open ball around $\T U$ of radius $2\eta$, we have
 \beq{eq:red4}
\liminf_{n\to\infty} \frac{1}{(n\,\|\V b\|_1)^2}\log\tP_{n\V b,\V p}(B_{\delta_\Box}(\widetilde U,2\eta))\ge 
-\inf_{\Sball(\T U,\eta)} J_{\V b,\V p}.
 \eeq
 Similarly as for the upper bound, the left-hand side can be upper bounded via Lemma~\ref{lm:ContOfJ} by, for example, 
 \beq{eq:red3}
 \liminf_{n\to\infty} \frac{1-\eta}{(\|\V a_n\|_1)^2}\log \tP_{\V a_n,\V p}(\Sball(\widetilde U,3\eta)). 
 \eeq
 Since $m$ is sufficiently large,
 we can apply Lemma~\ref{lm:ContOfJ} to $\V\kappa:=\B b/\|\V b\|_1$, $\V\gamma:=\V\alpha/\|\V\alpha\|_1$, the given graphon $U$ and our chosen $\eta>0$ 
 to find $\T{V}\in \Sball(\T U, \eta)$ such that $J_{\V b,\V p}(\T{V})\le J_{\V\alpha,\V p}(\T U)+\eta$. Thus
 $$
 -\inf_{\Sball(\T U,\eta)} J_{\V b,\V p}\ge -J_{\V b,\V p}(\T{V})\ge -J_{\V\alpha,\V p}(\T U)-\eta.
 $$
 By~\eqref{eq:red4}, this is a lower bound on the expression in~\eqref{eq:red3}. Taking the limit as $\eta\to 0$ we get the required LDP lower bound~\eqref{eq:lower}. This finishes the proof of Theorem~\ref{th:GenLDP}.
 \end{proof}

\section{Proof of Theorem~\ref{th:ourLDP}}\label{ourLDP}

\begin{proof}[Proof of Theorem~\ref{th:ourLDP}] 
Recall that $W$ is a $k$-step graphon with non-null parts whose values are encoded by a symmetric $k\times k$ matrix $\V p\in [0,1]^{k\times k}$. We consider the $W$-random graph $\I G(n,W)$ where we first sample $n$ independent uniform points $x_1,\dots,x_n\in [0,1]$ and then make each pair $\{i,j\}\subseteq [n]$ an edge with probability $W(x_i,x_j)$. We have to prove an LDP for the corresponding sequence $(\tR_{n,W})_{n\in\I N}$ of measures on the metric space $(\tW,\delta_\Box)$ with speed $s(n)=(c+o(1))n^2$. Of course, it is enough to do the case $s(n)=n^2$ as this affects only the scaling factors. Then the claimed rate function is $R_{\V p}$, which was defined in~\eqref{eq:R}. Recall that the lower semi-continuity of $R_{\V p}$ was established in Theorem~\ref{th:lsc}.

Let us show the LDP lower bound. Since the underlying space $(\tW,\delta_\Box)$ is compact, it is enough to prove the bound in~\eqref{eq:lower} of Lemma~\ref{lm:LDP} for any given $\T U\in \T{\C W}$, that is, that
\beq{eq:L1}
\lim_{\eta\to 0}\liminf_{n\to\infty} \frac{1}{n^2} \log \tR_{n,W}(\Sball(\widetilde U,\eta))\ge - R_{\V p}(\T U).
\eeq
Take any $\e>0$. By the definition of $R_{\V p}$, we can fix a vector $\V\alpha=(\alpha_1,\dots,\alpha_k)\in [0,1]^k$ 
such that $\|\V\alpha\|_1=1$ and 
 $$
 R_{\V p}(\T U)\ge J_{\V \alpha,\V p}(\T U)-\e.
$$
Let  $m$ be the number of non-zero entries of $\V \alpha$. As $\V \alpha$ is non-zero, we have $m\ge 1$. We can assume by symmetry that $\alpha_1,\dots,\alpha_m$ are the non-zero entries. Define $\xi:=\frac12 \min_{i\in [m]} \alpha_i>0$. 

For each $n\in\I N$, take any integer vector $\V a_n=(a_{n,1},\dots,a_{n,k})\in \NZ^k$ such that $\|\V a_n\|_1=n$ and $\|\V a_{n}-n\,\V\alpha\|_\infty<1$ (in particular, we have $a_{n,i}=0$ if $\alpha_i=0$). 

Let $n$ be sufficiently large. In particular, for every $i\in [m]$ we have that $a_{n,i}/n\ge \xi$. 
When we generate $G\sim \I G(n,W)$ by choosing first random $x_1,\dots,x_n\in [0,1]$, 
 it holds with probability at least, very roughly, $\xi^{-n}$ that for each $i\in [k]$ the number of $x_j$'s that belong to the $i$-th part of the step graphon~$W$ is exactly~$a_{n,i}$. Conditioned on this event of positive measure, the resulting graphon
 $\tf{G}$ is distributed according to $\tP_{\V a_n,\V p}$. Thus 
 $$
 \tR_{n,W}(S)\ge \xi^{-n}\,\tP_{\V a_n,\V p}(S),\quad\mbox{for every $S\subseteq \tW$}.
 $$
 This and Theorem~\ref{th:GenLDP} give that, for every $\eta>0$, 
 \begin{eqnarray*}
 \limsup_{n\to\infty} \frac{1}{n^2} \log \tR_{n,W}(\Sball(\widetilde U,\eta))&\ge& \limsup_{n\to\infty} \frac{1}{n^2} \log \tP_{\V a_n,\V p}(\Sball(\widetilde U,\eta))\\
 &\ge& -\inf_{\Sball(\T U,\eta/2)} J_{\V\alpha,\V p}
 \ \ge\ -J_{\V\alpha,\V p}(\T U)\ \ge\ -R_{\V p}(\T U)-\e.
 \end{eqnarray*}
 Taking the limit as $\eta\to 0$, we conclude that the LDP lower bound~\eqref{eq:L1} holds within additive error~$\e$. As $\e>0$ was arbitrary, the lower bound
 holds.
 
 Let us show the upper bound~\eqref{eq:upperGen} of Definition~\ref{df:LDP} for any closed set $F\subseteq \tW$, that is, that
  \beq{eq:U3}
 \limsup_{n\to\infty} \frac1{n^2} \log \tR_{n,W}(F)\le -\inf_F R_{\V p}.
 \eeq

 For each $n\in\I N$, we can write $\tR_{n,W}(F)$ as the sum over all  $\V a\in\NZ^k$ with $\|\V a\|_1=n$ of the probability that the distribution of random independent $x_1,\dots,x_n\in [0,1]$ per $k$ steps of $W$ is given by the vector $\V a$ times the probability conditioned on $\V a$ to hit the set~$F$. This conditional probability is exactly $\tP_{\V a,\V p}(F)$. Thus $\tR_{n,W}(F)$ is a convex combination of the reals $\tP_{\V a,\V p}(F)$ and there is a vector $\V a_n\in \NZ^k$ with $\|\V a_n\|_1=n$ such that
 \beq{eq:an}
 \tR_{n,W}(F)\le \tP_{\V a_n,\V p}(F).
 \eeq
 Fix one such vector~$\V a_n$ for each $n\in\I N$.
 
 Since the set of all real vectors $\V\alpha\in [0,1]^k$ with $\|\V\alpha\|_1=1$ is compact, we can find an increasing sequence $(n_i)_{i\in\I N}$ of integers such that 
 \beq{eq:U1}
 \limsup_{n\to\infty} \frac1{n^2} \log \tR_{n,W}(F)=\lim_{i\to\infty} \frac1{n_i^2} \log \tR_{n_i,W}(F),
 \eeq
 and the scaled vectors $\frac1{n_i}\,\V a_{n_i}$ converge to some real vector $\V\alpha\in [0,1]^k$.
 
 Let $(\V b_n)_{n\in\I N}$ be obtained by filling the gaps in $(\V a_{n_i})_{n\in\I N}$, meaning that if $n=n_i$ for some  $i\in\I N$ then we let $\V b_n:=\V a_{n_i}$; otherwise we pick any $\V b_n\in\NZ^k$ that satisfies $\|\V b_n\|_1=n$ and $\|\V b_n-n\,\V\alpha\|_\infty<1$. Since the normalised vectors $\V b_{n}/\|\V b_{n}\|_1$ converge to the same limiting vector~$\V\alpha$, we have by Theorem~\ref{th:GenLDP} that
 \beq{eq:U2}
 \limsup_{i\to\infty} \frac1{n_i^2} \log \tP_{\V a_{n_i},\V p}(F)\le \limsup_{n\to\infty} \frac1{n^2} \log \tP_{\V b_{n},\V p}(F)\le -\inf_F J_{\V\alpha,\V p}.
 \eeq
 Putting~\eqref{eq:an}, \eqref{eq:U1} and~\eqref{eq:U2} together with the trivial consequence $J_{\V\alpha,\V p}\ge R_{\V p}$ of the definition of $R_{\V p}$, we get the desired upper bound~\eqref{eq:U3}. This finishes the proof of Theorem~\ref{th:ourLDP}.\end{proof}

\section{Concluding remarks}\label{se:concluding}

The existence of an LDP for $(\tR_{n,W})_{n\in\I N}$ in the case case of a general graphon $W$ and speed $\Theta(n^2)$ is the key remaining open problem.
We remark that it is not hard to show that any potential rate function for speed $s(n)=\Theta(n^2)$ has to be zero on~$\RN(W)$. On the other hand, there are graphons $W$ such that $\RN(W)=\tW$: for example, take a countable set of graphons $\{W_i\mid i\in \I N\}$ dense in $(\tW,\delta_\Box)$ and let $W$ be zero except it has a scaled copy of $W_i$ on each square $[2^{-i},2^{-i+1}]^2$. It is easy to see that the LDP for $(\tR_{n,W})_{n\in\I N}$ and speed $n^2$ must have the rate function identically~0. It is plausible that by taking different sets of the graphons $W_i$ in the above example, one may get rate functions with rather complicated zero sets.

Our results fully resolve the LDP problem for $(\tR_{n,W})_{n\in\I N}$ for every graphon $W$ with $\RN(W)=\FORB(W)$,
with the statements precisely mirrowing those of Theorem~\ref{thm:GeneralSpeedsH}. Indeed, all claimed results on the existence of an LDP  follow from Theorem~\ref{thm:GeneralSpeeds}: for example, if $s(n)=\omega(n)$ then the rate function is $\chi_{\RN(W)}=\chi_{\FORB(W)}$ with the LDP lower (resp.\ upper) bound following from Part~\ref{it:GS2} (resp.\ Part~\ref{it:GS3}) of Theorem~\ref{thm:GeneralSpeeds}. Also, the proof of Part~\ref{it:GSH4} of Theorem~\ref{thm:GeneralSpeedsH} (the non-existence of an LDP) relies only on the already established LDPs 
and the fact that the corresponding rate functions are distinct from each other, and thus the same proof works in this case as well.
 
However, it is unclear for which graphons $W$ it holds that $\RN(W)=\FORB(W)$. By the uniqueness of the rate function, Theorems~\ref{thm:GeneralSpeeds} and~\ref{thm:GeneralSpeedsH} imply that $\RN(W)=\FORB(W)$ for any $\{0,1\}$-valued graphon. On the other hand, the constant-$p$ graphon $W$ for $p\in (0,1)$ is extreme in the sense that $\RN(W)$ is the singleton $\{\T W\}$ while $\FORB(W)=\tW$ is the whole space. The earlier example of a graphon $W$ with countably many ``blocks" such that $\RN(W)=\tW$ (when $\FORB(W)\supseteq \RN(W)$ is also the whole of $\tW$) shows that we can have equality for graphons assuming values in $(0,1)$ on a set of positive measure. We do not see any good characterisation of $W$ with $\RN(W)=\FORB(W)$.

We did not include any non-existence statements in Theorem~\ref{thm:GeneralSpeeds}, since we are rather far from understanding the full picture. Of course, if $s(n)=o(n^2)$ then, by the exponential equivalence, the non-existence statements of Theorem~\ref{thm:GeneralSpeedsH} translate verbatim into those for~$(\tR_{n,W})_{n\in\I N}$. Also, further non-existence results can be obtained in those cases when there are two subsequences of $n$ for which we have LDPs (e.g.\ by Theorems~\ref{thm:GeneralSpeeds}) with different rate functions. However, the lack of general positive results for $s(n)=\Theta(n^2)$ prevents us from completing the picture.

\hide{

\section*{Acknowledgements}

Jan Greb\'\i k was supported by Leverhulme Research Project Grant RPG-2018-424, and by MSCA Postdoctoral Fellowships 2022 HORIZON-MSCA-2022-PF-01-01 project BORCA grant agreement number 101105722.
Oleg Pikhurko was supported by ERC Advanced Grant 101020255 and Leverhulme Research Project Grant RPG-2018-424.
}

\small



\normalsize
\appendix

\section{Functional analysis}\label{app:Func}

In this appendix we collect some results
that we need in this paper.

Let $\mu$ be a probability measure on a measurable space~$X$. Recall that a collection $\mathcal{F}\subseteq L^1(X,\mu)$ is {\it uniformly integrable} if
$$
\lim_{C\to\infty} \sup_{f\in\mathcal{F}} \int_{\{|f|>C\}} |f|\,\dd\mu =0.
$$
 (Since $\mu(X)<\infty$, some other equivalent definitions are possible, see e.g.,~\cite[Proposition 4.5.3]{Bogachev:mt}.)
We use the following results whose (more general) statements can be found in e.g.~\cite{Bogachev:mt} (as
Theorems 4.7.18, 4.5.9 and 4.7.10 respectively).

\begin{theorem}[Dunford--Pettis]\label{App:DP} 
A collection $\mathcal{F}\subseteq L^1(X,\mu)$ is uniformly integrable if and only if its $weak$ closure is compact in the $weak$ topology.\qqed 
\end{theorem}

\begin{theorem}[De la Vall\' ee-Poussin]\label{APP:VP}
A collection $\mathcal{F}\subseteq L^1(X,\mu)$ is uniformly integrable if and only if there is non-negative non-decreasing convex function $G:\mathbb{R}\to \mathbb{R}$ such that $\lim_{t\to \infty} \frac{G(t)}{t}=\infty$ and
$$\sup_{f\in \mathcal{F}} \int_{X} G(|f|) \ d\mu<\infty.\qqed$$
\end{theorem}

\begin{theorem}[Eberlein--\v Smulian]\label{APP:ES}
Let $A\subseteq L^1(X,\mu)$ be closed in the $weak$ topology.
Then $A$ is compact in the $weak$ topology if and only if every sequence contains a subsequence that converges in the $weak$ topology.\qqed
\end{theorem}

The following result is an easy consequence of the Hahn-Banach Theorem, for a proof see e.g.~\cite[Page~6]{EkelandTemam76cavp} 

\begin{theorem}[Mazur's lemma]\label{th:B9}
A convex subset of a normed space  is weakly closed if and only if it is closed.\qqed
\end{theorem}

\end{document}

\section*{Notation/Macros}

\verb$\ProofOf{X}$: proof of part labelled by \verb$X$

$\I G(n,W)$ \verb$\I G(n,W)$

$\L{x}:=x\log(x)$ \verb$\L{x}$, it has \verb$\left(...\right)$
 
$\l{\frac ab}$ \verb$\l{\frac ab}$: the same except no \verb$\left \right$

$f(\_)$ \verb$f(\_)$

$\I 1_X$ \verb$\I 1_X$ 

$[0,1]^k$ instead of $[0,1]^{[k]}$ ? YES!

$d_\Box$ \verb$d_\Box$

$\NZ$ \verb$\NZ$

$\iP$ \verb$\iP$

$\Sball$ \verb$\Sball$

$\W$ \verb$\W$

$\tW$ \verb$\tW$

$\widetilde{W}$ \verb$\widetilde{W}$

$\tR_{n,W}$ \verb$\tR_{n,W}$

$\tH_{n,W}$ \verb$\tH_{n,W}$

$\f{H}$ \verb$\f{H}$

$\tf{H}$ \verb$\tf{H}$

$\bf x$ \verb$\bf x$

$\FORB(W)$ \verb$\FORB(W)$

$\RN(W)$ \verb$\RN(W)$

$\mathcal{P}([0,1])$ \verb$\mathcal{P}([0,1])$

$\C G_n$ \verb"\C G_n" and \verb$\fG_n$: the set of graphs on $[n]$

$\fH_n$ \verb$\fH_n$: weighted graphs on $[n]$

$\V p$ \verb"\V p": symmetric matrix $(p_{i,j})_{i,j\in k}$

$I_p$: CV function $\C W\to\I R$

interval step graphons

$\cI{\V\alpha}{i}$ \verb"\cI{\V\alpha}{i}" : partition of $[0,1]$ into intervals

$\Ak$ \verb|\Ak|: set of partitions of $[0,1]$

$\Aalpha$ \verb|\Aalpha|: ones with ratios $\alpha_1:\dots:\alpha_k$




$\Wk$ \verb|\Wk|

$\tWk$ \verb|\tWk|

$\Tk{U,\C A}$ \verb"\Tk{U,\C A}": equivalence class of $(U,\C A)$

$\Walpha$ \verb"\Walpha"

$\tWalpha$ \verb"\tWalpha"


Use $\T U$ \verb$\T U$ for elements of $\tW$

$\tPap$ \verb|\tPap|: sample $a_i$ points from $i$-th part

$\iR_{n,W}$, $\tR_{n,W}$ \verb"\iR_{n,W}", \verb|\tR_{n,W}|: our sampling

$\Sball(U,\eta)$ \verb"\Sball(U,\eta)": closed ball

$\Sball(\widetilde U,\eta)$ \verb"\Sball(\widetilde U,\eta)"



\section*{Borgs et al Notation}

$\C W$ \verb"\C W"

$\tW$: \verb|\tW|

$x\ind{i}$ \verb$x\ind{i}$ and \verb$x_i$

$f^G$ \ra\ $\f G$ \verb|\f{G}|

$\T{f}^G$ \ra\ $\tf G$ \verb"\tf{G}"

$\T f$: equivalence class of $f\in\C W$

$\Delta_m$: rational simplex

uniform interval step graphon $W=(p_{ij})_{i,j\in [k]}$

speed, rate function

(1-dim) entropy $h_p$

$I_W(U)=\frac12\sum_{ij} \int h_{p_{ij}}(U) :\C W\to [0,\infty]$ 

$B(\T f,\e)=\{g\in\C W\mid \delta_\Box(\T f,g)\le \e\}$

$S(\T f,\e)=\{\T g\in\tW\mid \delta_\Box(\T f,\T g)\le \e\}$ \ra\ $\Sball(U,\e)$ \verb"\Sball(U,\e)"

$\iP_{kn,W}$ \ra\ $\iP_{\V a_n,\V p}$ \verb|\iP_{\V a_n,\V p}| for $\V a_n=n(1,\dots,1)$

$\T {\iP}_{kn,W}$ \ra\ $\tP_{\V a_n,\V p}$ \verb|\tP_{\V a_n,\V p}| for $\V a_n=n(1,\dots,1)$

closed set $\T F\subseteq \T W$

open set $\T U\subseteq \T W$ \ra\ $\T G$ \verb"\T G"

$J_W=\sup\int I_W$ \ra\ $J_{\V a,\V p}$

\end{document}